\documentclass[11pt, twoside]{amsart}

\usepackage[T1]{fontenc}
\usepackage{amssymb,amsmath,mathtools,amsthm}
\usepackage{tikz}
\usepackage{graphicx}
\usepackage{hyperref}

\theoremstyle{plain}
\newtheorem{theor}{Theorem}[section]
\newtheorem{lem}[theor]{Lemma}

\newtheorem{cor}[theor]{Corollary}

\theoremstyle{definition}
\newtheorem{defin}[theor]{Definition}

\newtheorem{notation}[theor]{Notation}
\newtheorem{exam}[theor]{Example}

\theoremstyle{remark}

\newtheorem{rem}[theor]{Remark}

\numberwithin{equation}{section}

\newcommand{\mc}{\mathcal}

\DeclareMathOperator{\Ent}{Ent}

\newcommand{\mbbE}{\mathbb{E}}
\newcommand{\mbbP}{\mathbb{P}}
\newcommand{\mbbN}{\mathbb{N}}
\newcommand{\mbbR}{\mathbb{R}}
\newcommand{\mbbZ}{\mathbb{Z}}

\newcommand{\mbG}{\mathbf{G}}

\title[Exponential random graph models with soft clique constraints]
{Exponential random graph models with \\ soft clique constraints}

\author{Yasmin Tousinejad and Vera Koponen}

\address{Department of Mathematics, Uppsala University, Sweden.}
\email{yasmin.tousinejad@math.uu.se}

\address{Department of Mathematics, Uppsala University, Sweden.}
\email{vera.koponen@math.uu.se}

\date{31 August, 2026}

\begin{document}

\begin{abstract}
Let $r\geq3$ be fixed, and let $\mathbf{G}_n$ be the set of all simple graphs with vertex set $[n]=\{1,\ldots,n\}$.
We consider an exponential random graph model which gives higher probability to $G \in \mathbf{G}_n$
than to $H \in \mathbf{G}_n$ if $G$ has fewer $r$-cliques than $H$. But all graphs in $\mathbf{G}_n$ have positive
probability. The degree to which graphs with fewer $r$-cliques are given higher probability is determined by a 
positive weight $w$. We prove that, asymptotically almost surely as $n \to \infty$, a random graph from $\mathbf{G}_n$
has a vertex partition into $r-1$ parts of roughly equal size, the density of edges between the parts is close to $1/2$,
and for every $\varepsilon > 0$ the density of edges within any part is less than $\varepsilon$.
The asymptotic structural properties are independent of the weight $w$ as long as it is positive.
We also extend the result to the context of several clique sizes, each one with its own weight.
\end{abstract}

\maketitle

\section{Introduction}\label{sec:introduction}

\noindent
Studies of large graphs, or other structures, subject to certain hard constraints have a long history in the field
of random structures. 
The classical result of Erd\H{o}s, Kleitman and Rothschild
\cite{ErdosKleitmanRothschild1976}
states that the proportion of bipartite graphs among all triangle-free graphs
with vertices $1,\ldots,n$ tends to $1$ as $n\to\infty$.
More generally, Kolaitis, Pr\"omel and Rothschild
\cite{KolaitisPromelRothschild1987}
proved that, for every fixed $r\geq3$, the proportion of $(r-1)$-partite graphs among all
$K_r$-free graphs with vertices $1,\ldots,n$ tends to $1$ as $n\to\infty$.
We can say that asymptotically almost surely (a.a.s.) a $K_r$-free graph is $(r-1)$-partite.
Later studies have considered the a.a.s. structure of $H$-free graphs for other ``forbidden'' subgraphs $H$
(e.g. \cite{BBS, BBS2, HPS}),
and yet other investigations have considered forbidding certain {\em induced} subgraphs
(e.g. \cite{BB, KKOT, PS, Ree}).
Some studies have considered the a.a.s. structure under constraints other than forbidding a {\em finite}
number of subgraph or
induced subgraph configurations (e.g. \cite{BBCLS, BMSW, Kop12, Wor}), and some have considered structures other than graphs
(e.g. \cite{KR, Ter}).

In all mentioned examples, the probability distribution on structures with vertex set $[n] := \{1, \ldots, n\}$
that satisfy the given constraints is the uniform one.
One can reformulate the problem by considering all graphs (or some other kind of structure) with vertex set
$[n]$, assigning probability 0 to graphs that do not satisfy the constraints and equal probability to the other graphs.
An alternative is to not assign probability 0 to graphs that violate the constraints, but to instead consider a
probability distribution on graphs with vertex set $[n]$ with the following property:
A graph with more violations of the constraints is less likely than a graph with fewer violations of the constraints,
but graphs with violations can still have positive probability.
Since the constraints may be violated in this setting, it is more appropriate to call them \emph{soft constraints}.
This idea is present in so-called \emph{exponential random graph models (ERGMs)} \cite{Robins2007,ChatterjeeDiaconis2013}.
The same idea also underlies probability distributions defined by so-called
\emph{Markov logic networks (MLNs)} \cite{RichardsonDomingos2006},
which are formal specification models in the field of
\emph{Statistical Relational Artificial Intelligence (SRAI)} \cite{BKNP, RKNP, GT}.
MLNs can be defined for finite relational structures, not just graphs, but when restricted to graphs, they
are special cases of ERGMs.
Not much is known about \emph{structural} properties of large random structures defined by fixed-weight MLNs.
The investigations most directly related to this topic that we are aware of are
\cite{KoponenDomainSizeAsymptotics, TousinejadKoponen2026}.

Our main result concerns an ERGM, or equivalently an MLN, on random graphs which ``penalizes'' (for any fixed $r \geq 3$)
$r$-cliques, that is, copies of the complete graph on $r$ vertices, denoted $K_r$.
In other words, a graph with more $r$-cliques will be less likely than a graph with fewer $r$-cliques, 
but all graphs have positive probability.
The strength of the penalization is expressed by a positive weight $w$.
The precise definition of the distribution is given in 
Definition~\ref{definition of distribution}.
Informally speaking, our main result,
Theorem~\ref{thm:main-clique},
shows that, for every fixed positive weight $w$ and every $\varepsilon > 0$, with probability tending to 1
as the number $n$ of vertices tends to infinity, the following holds for a random graph:
\begin{enumerate}
\item There is a partition of the vertex set into $r-1$ parts, each of size close to $\tfrac{n}{r-1}$.
\item The density of edges between any two different parts is close to $\tfrac{1}{2}$.
\item For every $\varepsilon > 0$, if $n$ is large enough then the density of edges within any part is less than $\varepsilon$.
\end{enumerate}
Note that the asymptotic almost sure structure of a random graph does {\em not} depend on the weight $w$ 
(as long as $w > 0$).

We also prove
(Corollary~\ref{cor:exact-lower-colorability})
that if $r \geq 4$ then the probability that a random graph is $(r-2)$-partite tends to 0 
exponentially fast as the number of vertices tends to infinity.

A further consequence of our results is a stochastic block model limit for fixed induced subgraphs
(Corollary~\ref{cor:sbm-limit}). Informally speaking, for every fixed $m$, as $n$ tends to infinity,
the subgraph induced by $[m]$ converges in distribution to a stochastic block model
with $r-1$ classes such that each of the $m$ vertices is assigned independently and uniformly
to one of these classes. Given these assignments, there are no edges within a class,
and each possible edge between different classes is included independently with probability $\tfrac{1}{2}$.
The limiting distribution depends on $r$, but not on $w$.

We can refine the probability model considered so that we consider some integers $3 \leq r_1 < \ldots < r_k$
and for each $r_i$ we penalize an $r_i$-clique with weight $w_i > 0$. 
(The larger $w_i$ is, the more strongly $r_i$-cliques are penalized.)
We show
(Theorem~\ref{thm:multi-clique-collapse})
that, for all choices of weights $w_1, \ldots, w_k > 0$, 
the asymptotic behaviour of random graphs under this model is determined by $r_1$.
More precisely, with probability tending to 1 as the number $n$ of vertices tends to infinity,
a random graph will satisfy conditions~(1) -- (3) above if we let $r := r_1$.

Our main results are proved by using graphons
\cite{BorgsChayesLovaszSosVesztergombi2008,Lovasz2012,Janson2013}, which are analytic objects used to represent limits of sequences of dense graphs.
The article is organized as follows.
Just below we introduce the general notation and terminology that will be used.
In Section~\ref{sec:model} we define the probability model and state the main results
(in Theorem~\ref{thm:main-clique}).
Section~\ref{sec:graphons} defines the concept of graphon, as well as related concepts, and proves some basic results
that will be used later.
In Section~\ref{sec:entropy-stability} we work with the concepts of entropy of graphons,
homomorphism density of $K_r$ in graphons, and the cut distance of a graphon from a 
``balanced $(r-1)$-partite graphon'',
and show how they are related.
In Section~\ref{sec:free-energy} we prove Theorem~\ref{thm:main-clique}.
Section~\ref{sec:examples-consequences}
derives some consequences of Theorem~\ref{thm:main-clique} which were briefly described above.

\subsection*{General notation and terminology}\label{Notation and terminology}

By $\mbbN$ we denote the set of positive integers. For $n \in \mbbN$, $[n] := \{1, \ldots, n\}$.
For $n\in\mbbN$, let $\mbG_n$ be the set of all undirected graphs without loops (i.e. simple graphs)
with vertex set $[n]$.
For every finite graph $G$, let $E(G)$ denote its edge set and let $\chi(G)$ denote its chromatic number. 

If $n\in\mbbN$, $G\in\mbG_n$, 
and $A\subseteq [n]$, set
\[
e_G(A)\coloneqq \bigl|\{\{u,v\}\in E(G):\ u,v\in A\}\bigr|.
\]
If in addition $B\subseteq [n]$ is disjoint from $A$, define
\[
e_G(A,B)\coloneqq \bigl|\{\{u,v\}\in E(G):\ |\{u,v\}\cap A|=1,\ |\{u,v\}\cap B|=1\}\bigr|.
\]
For $n\in\mbbN$ and graphs $G,H\in\mbG_n$, define
the {\em $n^2$-normalized edge-edit distance} between $G$ and $H$ by
\[
d_{\mathrm{edit}}(G,H)
\coloneqq
\frac{1}{n^2}\,|E(G)\triangle E(H)|.
\]
Here $\triangle$ denotes symmetric difference. Thus $d_{\mathrm{edit}}(G,H)$ is the number of edge insertions and deletions needed to turn $G$ into $H$, normalized by $n^2$.

For $n\in\mbbN$ and $p\in[0,1]$, the notation
\[
\mathsf H_n\sim G(n,p)
\]
means that $\mathsf H_n$ is a random element of $\mbG_n$ sampled
by including an unordered pair $\{i,j\}$ as an edge with probability $p$, independently of whether other pairs have
been included or not.
In particular, $\mathsf H_n\sim G(n,\tfrac12)$ means that every graph in $\mbG_n$ has probability
$2^{-\binom n2}$, so $\mathsf H_n$ is uniformly distributed on $\mbG_n$.

For every integer $s\ge2$, 
$K_s$ denotes a complete graph with $s$ vertices. 
For every finite graph $G$, let $K_s(G)$ denote the number of \emph{unordered} copies of $K_s$ in $G$, in other words, the number of subgraphs of $G$ that are isomorphic to $K_s$.

\section{The probability model and the main results}\label{sec:model}

\noindent
Throughout this article $r\ge 3$ is a fixed integer and $w>0$ is a fixed real number interpreted as the ``{\em weight}'' 
that will be associated to
every $r$-tuple of vertices of a graph that does not form an $r$-clique.

\subsection{The probability distribution}

\begin{defin}\label{definition of N-r}{\rm
For every $n \in \mbbN$ and $G \in \mbG_n$, let
\begin{equation}\label{eq:Nphi-r}
N_{r}(G):=n^r-r!\,K_r(G),
\end{equation}
so $N_r(G)$ is the number of ordered $r$-tuples
$(v_1,\ldots,v_r)\in[n]^r$ (with repetitions allowed) that do not
consist of $r$ distinct vertices that form an $r$-clique in $G$.
}\end{defin}

\begin{defin}\label{definition of distribution}{\rm
For each $n \in \mbbN$ and $G\in\mbG_n$, we define its \emph{unnormalized mass} by
\[
\mu_{n,r}(G)\coloneqq \exp\bigl(w\,N_{r}(G)\bigr).
\]
We then define the \emph{partition function} (or \emph{normalizing constant}) by
\[
Z_{n,r}\coloneqq \sum_{J\in\mbG_n}\mu_{n,r}(J).
\]
Finally, we define the \emph{probability measure} of the model by
\[
\mbbP_{n,r}(A)\coloneqq
\frac{1}{Z_{n,r}}\sum_{G\in A}\mu_{n,r}(G)
\qquad \text{ for all } \ A\subseteq\mbG_n.
\]
For a single graph $G\in\mbG_n$, we also set
\[
\mbbP_{n,r}(G)
\coloneqq
\mbbP_{n,r}(\{G\})
=
\frac{\mu_{n,r}(G)}{Z_{n,r}}.
\]
}\end{defin}

\noindent
Since $w$ is fixed in the model, we suppress the dependence on $w$ in the notation $\mu_{n,r}$, $Z_{n,r}$, and $\mbbP_{n,r}$.
Note that if $G, H \in \mbG_n$ and $K_r(G) < K_r(H)$, then $\mbbP_{n, r}(G) > \mbbP_{n, r}(H)$.
The distribution $\mbbP_{n, r}$ is exactly the probability distribution on $\mbG_n$ obtained from a
Markov logic network \cite{RichardsonDomingos2006}
with one soft constraint of the form 
\[
\bigvee_{1 \leq i < j \leq r} \neg E(x_i, x_j) \quad \text{ with weight } w,
\]
where $\neg E(x_i, x_j)$ expresses that there is no edge between $x_i$ and $x_j$.

\begin{rem}\label{rem:pairwise-distinct-groundings} {\rm (On counting only tuples without repetitions)
In Definition~\ref{definition of distribution} we count \emph{all} ordered $r$-tuples (also those with repetitions)
\[
(x_1,\dots,x_r)\in [n]^r.
\]
Another natural convention is to count only those ordered tuples whose entries are pairwise distinct.
Under that convention, the total number of allowed ordered tuples is
\[
n(n-1)\cdots(n-r+1),
\]
and the falsifying tuples are still exactly the $r!$ orderings of the vertices of each copy of $K_r$.
So the number of satisfied groundings becomes
\[
N_{r}^{\mathrm{dist}}(G)
=
n(n-1)\cdots(n-r+1)-r!\,K_r(G).
\]

If we use the same weight $w$, then the corresponding unnormalized mass is
\[
\mu_{n,r}^{\mathrm{dist}}(G)
\coloneqq
\exp\bigl(w\,N_{r}^{\mathrm{dist}}(G)\bigr)
=
\exp\Bigl(w\bigl[n(n-1)\cdots(n-r+1)-r!\,K_r(G)\bigr]\Bigr).
\]
Compare this with our present convention,
\[
\mu_{n,r}(G)
=
\exp\bigl(w[n^r-r!\,K_r(G)]\bigr).
\]
Subtracting the exponents gives
\[
\mu_{n,r}(G)
=
\exp\Bigl(w\bigl[n^r-n(n-1)\cdots(n-r+1)\bigr]\Bigr)\,
\mu_{n,r}^{\mathrm{dist}}(G).
\]
The factor
\[
\exp\Bigl(w\bigl[n^r-n(n-1)\cdots(n-r+1)\bigr]\Bigr)
\]
is \emph{deterministic} in the sense that it depends only on $n$, $r$, and $w$, and not on the graph $G$.
Therefore this factor changes the partition function, but it disappears when we normalize to obtain probabilities.
Indeed, if
\[
Z_{n,r}^{\mathrm{dist}}
\coloneqq
\sum_{J\in\mbG_n}\mu_{n,r}^{\mathrm{dist}}(J),
\]
then
\[
Z_{n,r}
=
\exp\Bigl(w\bigl[n^r-n(n-1)\cdots(n-r+1)\bigr]\Bigr)\,
Z_{n,r}^{\mathrm{dist}},
\]
and hence
\[
\mbbP_{n,r}(G)
=
\frac{\mu_{n,r}(G)}{Z_{n,r}}
=
\frac{\mu_{n,r}^{\mathrm{dist}}(G)}{Z_{n,r}^{\mathrm{dist}}}.
\]
So, for the same clause weight $w$, the two counting conventions define exactly the same
probability measure on $\mbG_n$.
Thus the choice between these two conventions does not affect which graphs are more or less likely. 
}\end{rem}

\subsection{The equivalent exponential random graph model}

The distribution $\mbbP_{n, r}$ can be defined in the setting of exponential random graph models \cite{Robins2007}.

\begin{defin}\label{def:ERGM}
Fix $n,k\in\mbbN$.
An {\em Exponential Random Graph Model (ERGM)} on $\mbG_n$ is a probability distribution of the form
\[
\mbbP_{n,\vartheta}(G)
:=
\frac{\exp\bigl(\sum_{i=1}^k \vartheta_i\,T_i(G)\bigr)}
{Z_n^{\mathrm{ERGM}}(\vartheta)}
\qquad \text{for all } \  G\in\mbG_n,
\]
where $T_1,\dots,T_k$ are real-valued functions on $\bigcup_{n\in\mbbN} \mbG_n$, 
$\vartheta=(\vartheta_1,\dots,\vartheta_k)\in\mbbR^k$
is a parameter vector, and
\[
Z_n^{\mathrm{ERGM}}(\vartheta)
:=
\sum_{J\in\mbG_n}\exp\bigl(\sum_{i=1}^k \vartheta_i\,T_i(J)\bigr)
\]
is the normalizing constant.
All else being equal, a positive value of $\vartheta_i$ rewards larger values of $T_i(G)$ in the exponent,
while a negative value penalizes them.
\end{defin}

\noindent
Note that if $k = 1$, $\vartheta_1 = w$, and $T_1(G) = N_{r}(G)$, then
\[
\mbbP_{n, r}(G) = \mbbP_{n, \vartheta}(G) \quad \text{ for all } \ G \in \mbG_n.
\]
Hence $\mbbP_{n, r}$ is an ERGM.

\subsection{The reduced partition function}

It will be convenient to work with the quantities defined below.

\begin{defin}\label{definition of reduced partition function}
(a) For all $n \in \mbbN$, let 
\begin{equation}\label{eq:Ztilde-r}
\begin{aligned}
&\widetilde Z_{n,r} \coloneq \widetilde Z_{n, r}(w)
\coloneqq
\sum_{J\in\mbG_n}\exp\bigl(-r!\,w\,K_r(J)\bigr) 
\ \text{ and } \\
&F_{n, r} \coloneq F_{n,r}(w)\coloneqq \frac{1}{n^2}\ln \widetilde Z_{n,r}.
\end{aligned}
\end{equation}
We call $\widetilde Z_{n,r}$ the {\em reduced partition function} and $F_{n,r}(w)$ the {\em reduced free energy}.

\noindent
(b) For a graph $G\in\mbG_n$, we call
\[
\exp\bigl(-r!\,w\,K_r(G)\bigr)
\]
the \emph{reduced mass} of $G$. Thus $\widetilde Z_{n,r}$ is the sum of the
reduced masses of all graphs in $\mbG_n$.
\end{defin}

The partition function $Z_{n,r}$ sums the unnormalized masses of all
graphs in $\mbG_n$. As the next lemma shows, its logarithm contains the
graph-independent term $wn^r$. This term changes the normalizing constant
but not the probability measure, so we will often remove it and work with the reduced
partition function and reduced free energy.

\begin{lem}\label{cor:ERGM}
For every $n\in\mbbN$ and every $G\in\mbG_n$,
\[
\mu_{n,r}(G)
=
\exp\bigl(w\,n^r\bigr)\exp\bigl(-r!\,w\,K_r(G)\bigr).
\]
Moreover,
\[
Z_{n,r}
=
\exp\bigl(w\,n^r\bigr)\,\widetilde Z_{n,r},
\]
and
\begin{equation}\label{eq:showing-ergm}
\mbbP_{n,r}(G)
=
\frac{\exp\bigl(-r!\,w\,K_r(G)\bigr)}
{\sum_{J\in\mbG_n}\exp\bigl(-r!\,w\,K_r(J)\bigr)}=\frac{\exp\bigl(-r!\,w\,K_r(G)\bigr)}{\widetilde Z_{n,r}}.
\end{equation}
\end{lem}

\begin{proof}
Insert \eqref{eq:Nphi-r} into the definition of $\mu_{n,r}(G)$:
\[
\mu_{n,r}(G)
=
\exp\bigl(w\,[n^r-r!\,K_r(G)]\bigr).
\]
Now factor out the deterministic term $e^{w n^r}$.
The formulas for $Z_{n,r}$ and $\mbbP_{n,r}$ follow immediately.
\end{proof}

\subsection{Main results}\label{Main results}

Recall that we have fixed an integer $r \geq 3$ and a real $w > 0$.
The notation
\[
\mathsf G_n\sim \mbbP_{n,r}
\]
means that $\mathsf G_n$ is a random graph from $\mbG_n$ under the distribution $\mbbP_{n,r}$.
(So if $J \in \mbG_n$ then $\mbbP(\mathsf G_n=J) = \mbbP_{n,r}(J)$.)
Recall also notation introduced in Section~\ref{Notation and terminology}.
Some of the notions appearing in the main result below are defined only later:
A {\em graphon} is defined in Definition~\ref{defin:graphon}.
For a graph $G\in\mbG_n$, its {\em empirical graphon} $W_G$ is defined in Definition~\ref{def:empirical-graphon}.
The {\em cut distance} $\delta_\square$
between two graphons is defined in Definition~\ref{defin:cutdistance}, 
and the cut distance from a graphon to a family is defined in Definition~\ref{defin:cutdistance-family}.
The family $B_q^\star$ is defined in Definition~\ref{def:Bqstar}; informally, it consists of the balanced $q$-block graphons that are equal to $0$ on the diagonal blocks and equal to $\tfrac12$ on the off-diagonal blocks.

\begin{theor}
\label{thm:main-clique}
Fix $r\ge 3$, set $q=r-1$.
Let $\mathsf G_n\sim\mbbP_{n,r}$.
Then the following hold.

\begin{enumerate}
\item[\textup{(i)}]
The reduced free energy satisfies
\[
\lim_{n\to\infty}F_{n,r}(w)
=
\frac{\ln 2}{2}\left(1-\frac1q\right).
\]

\item[\textup{(ii)}]
For every $\varepsilon>0$, there are constants $c=c(\varepsilon,r)>0$ and
$N=N(\varepsilon,r,w)$ such that, for every $n\ge N$,
\[
\mbbP\left(
\delta_\square(W_{\mathsf G_n},B_q^\star)\ge \varepsilon
\right)
\le
e^{-cn^2}.
\]

\item[\textup{(iii)}]
For every $\varepsilon>0$, there are constants $c=c(\varepsilon,r)>0$ and
$n_0=n_0(\varepsilon,r,w)$ such that, for every $n\ge n_0$, with probability at
least $1-e^{-cn^2}$, the graph $\mathsf G_n$ has a partition
\[
[n]=V_1\sqcup\cdots\sqcup V_q
\]
with the following properties:
\begin{enumerate}
\item Each part is almost balanced:
\[
\bigl||V_i|-n/q\bigr|\le \varepsilon n
\qquad\text{for every }i\in[q].
\]

\item Each between-part edge density is close to $\tfrac12$:
\[
\left|\frac{e_{\mathsf G_n}(V_i,V_j)}{|V_i||V_j|}-\frac12\right|\le \varepsilon
\qquad\text{for all distinct }i,j\in[q].
\]

\item The total number of edges inside the parts satisfies
\[
\sum_{i=1}^q e_{\mathsf G_n}(V_i)
\le
\frac{\varepsilon n^2}{10q}.
\]

\item The total density of edges inside the parts is at most $\varepsilon$:
\[
\frac{\sum_{i=1}^q e_{\mathsf G_n}(V_i)}{\sum_{i=1}^q \binom{|V_i|}{2}}
\le \varepsilon.
\]
\end{enumerate}
In particular, if $\mathsf G_n'$ is obtained from $\mathsf G_n$ by deleting all
edges whose endpoints lie in the same part, then $\mathsf G_n'$ is $q$-partite and
\[
d_{\mathrm{edit}}(\mathsf G_n,\mathsf G_n')
\le
\frac{\varepsilon}{10q}
\le
\frac{\varepsilon}{2}.
\]
\end{enumerate}
\end{theor}

\subsection{Related work}
\label{sec:related-work}

In the introduction we mentioned several studies of graphs, or other relational structures, subject to ``hard'' constraints.
We also mentioned \cite{KoponenDomainSizeAsymptotics, TousinejadKoponen2026}
about structures subject to ``soft'' constraints defined by Markov logic networks.
As we use graphons to prove our main results
we like to mention \cite{AWR}
which considers Markov logic networks and a multi-relational analogue of a graphon in the context when 
one considers several different edge relations.
However, the results in \cite{AWR}, which generalize the basic theory of graphons, are not of help in the present context
as far as we can see.

In \cite{ChatterjeeDiaconis2013} Chatterjee and Diaconis consider ERGM distributions on $\mbG_n$ of the form
\[
\mbbP_n^T(G) = \frac{\exp\big(n^2 T(G)\big)}{\sum_{J \in \mbG_n} \exp\big(n^2 T(J)\big)}
\]
where $T$ is a function from $\bigcup_{n\in\mbbN}\mbG_n$ into the reals.
If we rewrite the distribution $\mbbP_{n, r}$ so that, for an appropriate choice of $T$ and all $n$, 
$\mbbP_{n, r} = \mbbP_n^T$
then, for all $n$ and all $G \in \mbG_n$, we must have 
\[
T(G) := \frac{ w N_r(G)}{n^2}.
\]
Recall that $N_r(G)$ is the number of $r$-tuples of vertices of $G$ that do not form an $r$-clique.
So if $G \in \mbG_n$ does not have an $r$-clique (e.g. if $G$ is $(r-1)$-partite),
then $N_r(G) = n^r$ where $r \geq 3$.
It follows that $\max_{G \in \mbG_n} T(G)$ is unbounded as $n\to\infty$.
But the results of \cite{ChatterjeeDiaconis2013} assume that 
$\max_{G \in \mbG_n} T(G)$ is bounded as $n\to\infty$.
Therefore their results do not apply to the present context.

\section{Graphon preliminaries}\label{sec:graphons}

\subsection{Basic definitions}

\begin{defin}[Graphon]\label{defin:graphon}
A \emph{graphon} is a measurable function
\[
W:[0,1]^2\to[0,1]
\]
such that $W(x,y)=W(y,x)$ for almost every $(x,y)\in[0,1]^2$. Let $\mc W$ denote the family of these measurable functions.
\end{defin}

Let $\lambda$ denote \emph{Lebesgue measure} on $[0,1]$. For each integer $m\ge 1$, let $\lambda^m$ denote the \emph{product Lebesgue measure} on $[0,1]^m$.

Graphons are the standard limit objects for dense graphs. They also define random graph models. Let $W$ be a graphon and let $n\in\mbbN$. Choose
$X_1,\dots,X_n$ independently and uniformly from $[0,1]$.
Conditional on $X_1,\dots,X_n$, put each edge $\{i,j\}$ with $i<j$ into the graph independently with probability
\[
W(X_i,X_j).
\]

\begin{defin}[Empirical graphon]\label{def:empirical-graphon}
Let $n\in\mbbN$, and let $G$ be a graph with vertex set $[n]$ and adjacency matrix
$(a_{ij})_{1\le i,j\le n}$, where
\[
a_{ij}=a_{ji}
\qquad\text{and}\qquad
a_{ii}=0.
\]
Split $[0,1]$ into $n$ intervals of equal length by setting
\[
I_i=\Bigl[\frac{i-1}{n},\frac{i}{n}\Bigr)
\quad\text{for }1\le i<n,
\qquad
I_n=\Bigl[\frac{n-1}{n},1\Bigr].
\]
The \emph{empirical graphon} of $G$ is the graphon
$W_G:[0,1]^2\to[0,1]$ defined as follows. If
$x\in I_i$ and $y\in I_j$, then
\[
W_G(x,y)\coloneqq a_{ij}.
\]
Thus $W_G$ records the adjacency matrix of $G$ on the $n^2$ rectangles
$I_i\times I_j$. It is measurable, symmetric, and takes only the values
$0$ and $1$. Also, $W_G=0$ on each rectangle $I_i\times I_i$, because
$a_{ii}=0$.
\end{defin}

\begin{defin}[Homomorphism density]\label{defin:hom-density}
Let $H$ be a finite simple graph with vertex set $V(H)=\{1,\dots,k\}$.
For a graphon $W$, the \emph{homomorphism density} of $H$ in $W$ is
\[
t(H,W)
\coloneqq
\int_{[0,1]^k}\prod_{\{i,j\}\in E(H)}W(x_i,x_j)\,dx_1\cdots dx_k.
\]
\end{defin}

\begin{rem}\label{rem:Kr-hom-density}
For $H=K_r$ this becomes
\[
t(K_r,W)
=
\int_{[0,1]^r}\prod_{1\le a<b\le r}W(x_a,x_b)\,dx_1\cdots dx_r.
\]
The integrand is nonnegative.
So if $t(K_r,W)=0$, then the product inside the integral must vanish almost everywhere.
\end{rem}

From now on, a graphon $W$ is called a \emph{graphon of zero $K_r$-density} if
\[
t(K_r,W)=0,
\]
that is, if its $K_r$-homomorphism density is zero.

\begin{defin}[Cut norm]\label{defin:cutnorm}
For an integrable function $U:[0,1]^2\to\mbbR$, the cut norm is
\[
\|U\|_\square
\coloneqq
\sup_{\substack{S,T\subseteq[0,1]\\ S,T\text{ measurable}}}
\left|\int_{S\times T}U(x,y)\,dx\,dy\right|.
\]
\end{defin}

\begin{defin}[Cut distance]\label{defin:cutdistance}
A measurable map $\phi:[0,1]\to[0,1]$ is \emph{measure-preserving} if
\[
\lambda(\phi^{-1}(A))=\lambda(A)
\qquad\text{for every Lebesgue-measurable }A\subseteq[0,1].
\]

A \emph{measure-preserving bijection} is a bijection
$\phi:[0,1]\to[0,1]$ such that both $\phi$ and $\phi^{-1}$ are
Lebesgue-measurable and $\phi$ is measure-preserving.

If $W$ is a graphon and $\phi$ is a measure-preserving map, define
\[
W^\phi(x,y)\coloneqq W(\phi(x),\phi(y)).
\]
Thus $W^\phi$ is obtained from $W$ by applying $\phi$ to both variables.
When $\phi$ is a measure-preserving bijection, we call $W^\phi$ a
\emph{relabeling} of $W$.

For graphons $W_1$ and $W_2$, their \emph{cut distance} is
\[
\delta_\square(W_1,W_2)
\coloneqq
\inf_\phi \|W_1-W_2^\phi\|_\square,
\]
where the infimum is taken over all measure-preserving bijections
$\phi:[0,1]\to[0,1]$.
\end{defin}

Relabeling by a measure-preserving bijection does not change the cut norm.
Indeed, let $F:[0,1]^2\to\mbbR$ be a measurable function with
\[
\int_{[0,1]^2}|F(x,y)|\,dx\,dy<\infty,
\]
and let $\phi$ be a measure-preserving bijection. Define
\[
F^\phi(x,y)\coloneqq F(\phi(x),\phi(y)).
\]
Then
\[
\|F^\phi\|_\square=\|F\|_\square .
\]
This follows by changing variables in the integral over $S\times T$. Since
$\phi$ is measure-preserving and bijective, the sets $\phi(S)$ and
$\phi(T)$ run through the same measurable subsets of $[0,1]$ as $S$ and
$T$. Thus the supremum in the cut norm is unchanged.

\begin{defin}\label{defin:cutdistance-family}
Let $W$ be a graphon, and let $\mc F$ be a nonempty family of graphons. The
\emph{cut distance from $W$ to $\mc F$} is
\[
\delta_\square(W,\mc F)
\coloneqq
\inf_{V\in\mc F}\delta_\square(W,V).
\]
\end{defin}

The cut distance $\delta_\square$ is a pseudometric on $\mc W$. Indeed, let
$U,V,W\in\mc W$. To see symmetry, let $\phi$ be a measure-preserving
bijection. By invariance of the cut norm under relabeling,
\[
\|U-V^\phi\|_\square
=
\|U^{\phi^{-1}}-V\|_\square
=
\|V-U^{\phi^{-1}}\|_\square.
\]
Taking the infimum over $\phi$ gives
$\delta_\square(U,V)=\delta_\square(V,U)$. For the triangle inequality, if
$\phi$ and $\psi$ are measure-preserving bijections, then
\begin{align*}
\|U-W^{\psi\circ\phi}\|_\square
&\le
\|U-V^\phi\|_\square
+
\|V^\phi-(W^\psi)^\phi\|_\square\\
&=
\|U-V^\phi\|_\square
+
\|V-W^\psi\|_\square.
\end{align*}
Taking the infimum over $\phi$ and $\psi$ proves the triangle inequality.
Nonnegativity and $\delta_\square(U,U)=0$ are immediate.

\begin{notation}[Reduced graphon space]\label{not:reduced-graphon-space}
Define
\[
W\sim V
\qquad\Longleftrightarrow\qquad
\delta_\square(W,V)=0,
\]
and let
\[
\widetilde{\mc W}\coloneqq \mc W/\sim,
\qquad
\pi:\mc W\to\widetilde{\mc W}
\]
be the quotient space and quotient map. We use $\widetilde W\coloneqq \pi(W)$ and, for $n\in\mbbN$ and $G\in\mbG_n$, $\widetilde W_G\coloneqq \pi(W_G)$. For reduced graphons define
\[
d_\square(\widetilde U,\widetilde V)
\coloneqq
\delta_\square(U,V),
\]
where $U$ and $V$ are arbitrary representatives. This definition does not depend on the chosen representatives. Indeed, if
$U'\sim U$ and $V'\sim V$, then
\[
\delta_\square(U',V')
\le
\delta_\square(U',U)+\delta_\square(U,V)+\delta_\square(V,V')
=
\delta_\square(U,V).
\]
The same argument with $(U,V)$ and $(U',V')$ interchanged gives
\[
\delta_\square(U,V)\le \delta_\square(U',V').
\]
Hence
\[
\delta_\square(U',V')=\delta_\square(U,V).
\]
Thus $d_\square$ is well defined on $\widetilde{\mc W}$. It is a metric because
nonnegativity, symmetry, and the triangle inequality are inherited from
$\delta_\square$. Also,
\[
d_\square(\widetilde U,\widetilde V)=0
\]
holds exactly when $U\sim V$, which means $\widetilde U=\widetilde V$. The metric space
$(\widetilde{\mc W},d_\square)$ is compact; see \cite[Thm.~9.23]{Lovasz2012}.
For $\mc F\subseteq\mc W$, we use
\[
\pi(\mc F)\coloneqq \{\pi(W):W\in\mc F\}\subseteq\widetilde{\mc W}
\]
for its quotient image.
\end{notation}

\begin{rem}[Zero cut distance]\label{rem:zero-cut-distance}
For graphons $U$ and $V$, the condition
\[
\delta_\square(U,V)=0
\]
means that $U$ and $V$ represent the same graph limit. Equivalently, there
exist measure-preserving maps
$\phi,\psi:[0,1]\to[0,1]$ such that
\[
U^\phi=V^\psi
\qquad\text{almost everywhere};
\]
see \cite[Cor.~10.34 and Cor.~10.35(a)]{Lovasz2012}.

The maps in this characterization need not be bijections. Consequently,
the condition $\delta_\square(U,V)=0$ is weaker than almost-everywhere
equality after a single measure-preserving relabeling.

By Notation~\ref{not:reduced-graphon-space}, graphons at zero cut distance
represent the same point of $\widetilde{\mc W}$. This quotient space is the
reduced graphon space used in
\cite[Sec.~1.3]{ChatterjeeVaradhan2011}.
\end{rem}

\begin{notation}\label{not:L1norm}
For an integrable function $U:[0,1]^2\to\mbbR$, the $L^1$-norm is defined as
\[
\|U\|_1\coloneqq \int_{[0,1]^2}|U(x,y)|\,dx\,dy.
\]
\end{notation}

\begin{notation}[Indicator functions]\label{not:indicator}
For any set or event $A$, $\mathbf{1}_A$ denotes its
indicator. Thus $\mathbf{1}_A(z)=1$ when $z\in A$ and $\mathbf{1}_A(z)=0$ otherwise.
\end{notation}

\begin{lem}\label{lem:cauchy-schwarz}
We will use the following two forms of the Cauchy-Schwarz inequality.

\smallskip
\noindent\textup{(i)} If $S\subseteq[0,1]^2$ is measurable and
$f:S\to\mbbR$ is measurable with
\[
\int_S f^2<\infty,
\]
then
\[
\left(\int_S |f|\right)^2
\le
\lambda^2(S)\int_S f^2.
\]

\smallskip
\noindent\textup{(ii)} If $m\ge1$ is an integer and
$x_1,\dots,x_m\in\mbbR$, then
\[
\left(\sum_{a=1}^m x_a\right)^2
\le
m\sum_{a=1}^m x_a^2.
\]
Equivalently,
\[
\sum_{a=1}^m x_a^2
\ge
\frac1m\left(\sum_{a=1}^m x_a\right)^2.
\]
Consequently, if $q\in\mbbN$, $n\in\mbbZ_{\ge0}$,
$m_1,\dots,m_q\in\mbbZ_{\ge0}$, and
\[
\sum_{a=1}^q m_a=n,
\]
then
\[
\sum_{a=1}^q \binom{m_a}{2}
=
\frac12\left(\sum_{a=1}^q m_a^2-n\right)
\ge
\frac12\left(\frac{n^2}{q}-n\right).
\]
\end{lem}

\begin{proof}
For \textup{(i)}, apply \cite[Thm.~11.35]{Rudin1976} to the functions
$|f|$ and $1$ on $S$. We get
\[
\int_S |f|
\le
\left(\int_S f^2\right)^{1/2}
\left(\int_S 1\right)^{1/2}.
\]
Since
\[
\int_S 1=\lambda^2(S),
\]
it follows that
\[
\left(\int_S |f|\right)^2
\le
\lambda^2(S)\int_S f^2.
\]

For \textup{(ii)}, apply \cite[Thm.~1.35]{Rudin1976} to the two lists
\[
x_1,\dots,x_m
\qquad\text{and}\qquad
1,\dots,1.
\]
This gives
\[
\left(\sum_{a=1}^m x_a\right)^2
\le
\left(\sum_{a=1}^m x_a^2\right)
\left(\sum_{a=1}^m 1\right)
=
m\sum_{a=1}^m x_a^2.
\]
Dividing by $m$, we obtain
\[
\sum_{a=1}^m x_a^2
\ge
\frac1m\left(\sum_{a=1}^m x_a\right)^2.
\]

Now apply \textup{(ii)} with $m=q$ and $x_a=m_a$. Since
\[
\sum_{a=1}^q m_a=n,
\]
we have
\[
\sum_{a=1}^q m_a^2
\ge
\frac{n^2}{q}.
\]
Therefore
\[
\sum_{a=1}^q \binom{m_a}{2}
=
\frac12\sum_{a=1}^q m_a(m_a-1)
=
\frac12\left(\sum_{a=1}^q m_a^2-n\right)
\ge
\frac12\left(\frac{n^2}{q}-n\right).
\]

\end{proof}

\begin{lem}\label{lem:cutnorm-weighted}
Let $F:[0,1]^2\to\mbbR$ be bounded and measurable, and let
$f,g:[0,1]\to[0,1]$ be measurable. Then
\[
\left|\int_{[0,1]^2}F(x,y)\,f(x)\,g(y)\,dx\,dy\right|
\le \|F\|_\square.
\]
\end{lem}

\begin{proof}
By \cite[Eq.~(4.2) and Rem.~4.1]{Janson2013},
\[
\|F\|_\square
=
\sup_{\substack{\phi,\psi:[0,1]\to[0,1]\\
\phi,\psi\text{ measurable}}}
\left|
\int_{[0,1]^2}
F(x,y)\phi(x)\psi(y)\,dx\,dy
\right|.
\]
Taking $\phi=f$ and $\psi=g$ gives
\[
\left|
\int_{[0,1]^2}
F(x,y)f(x)g(y)\,dx\,dy
\right|
\le
\|F\|_\square.
\]
\end{proof}

\begin{lem}[Clique density is Lipschitz in cut distance]\label{lem:Kr-cut-lipschitz}
Let $r\ge 2$. For all graphons $U$ and $V$,
\[
|t(K_r,U)-t(K_r,V)|\le \binom{r}{2}\,\|U-V\|_\square.
\]
Consequently,
\[
|t(K_r,U)-t(K_r,V)|\le \binom{r}{2}\,\delta_\square(U,V).
\]
In particular, the map $W\mapsto t(K_r,W)$ is continuous with respect to cut distance.
\end{lem}

\begin{proof}
By \cite[Rem.~2.1]{Janson2013}, every Lebesgue-measurable graphon on
$[0,1]^2$ is equal almost everywhere to a Borel-measurable one. Such a
replacement changes neither homomorphism densities nor cut norms. We may
therefore replace $U$ and $V$ by Borel-measurable representatives and
continue to denote them by $U$ and $V$. In particular, whenever all but one of the variables are fixed below, the
resulting one-variable function is measurable.

Let
\[
m\coloneqq \binom{r}{2},
\qquad
D\coloneqq U-V.
\]
Choose once and for all an ordering of the $\binom{r}{2}$ edges of $K_r$.
Equivalently, choose an ordering of the pairs
\[
(i_1,j_1),\dots,(i_m,j_m),
\qquad
1\le i_\ell<j_\ell\le r,
\]
so that every pair appears exactly once.

Vertex $s$ carries the single variable $x_s\in[0,1]$.

The symbols $i_\ell$ and $j_\ell$ record which two vertices form the $\ell$-th edge in the chosen list.

For example, when $r=4$, one possible ordering is
\[
\begin{aligned}
(i_1,j_1)&=(1,2), & (i_2,j_2)&=(1,3), & (i_3,j_3)&=(1,4),\\
(i_4,j_4)&=(2,3), & (i_5,j_5)&=(2,4), & (i_6,j_6)&=(3,4).
\end{aligned}
\]

By Remark~\ref{rem:Kr-hom-density},
\[
t(K_r,W)=\int_{[0,1]^r}\prod_{\ell=1}^{m}W(x_{i_\ell},x_{j_\ell})\,dx_1\cdots dx_r.
\]

This means that we choose numbers
\[
x_1,\dots,x_r\in[0,1],
\]
one number for each vertex of the clique.
For every edge $\{i,j\}$ of the clique, we include the factor $W(x_i,x_j)$.
Then we multiply all these factors together.
Finally, we average over all choices of $(x_1,\dots,x_r)\in[0,1]^r$.
Since the set $[0,1]^r$ has total measure $1$, the integral here is literally an average.

Applying this formula to $U$ and to $V$, we obtain
\[
t(K_r,U)-t(K_r,V)
=
\int_{[0,1]^r}
\left(
\prod_{\ell=1}^{m}U(x_{i_\ell},x_{j_\ell})
-
\prod_{\ell=1}^{m}V(x_{i_\ell},x_{j_\ell})
\right)
\,dx_1\cdots dx_r.
\]

To compare the two products, it is helpful to change the edge factors one at a time.
For a fixed point $(x_1,\dots,x_r)\in[0,1]^r$, use
\[
U_\ell\coloneqq U(x_{i_\ell},x_{j_\ell}),
\qquad
V_\ell\coloneqq V(x_{i_\ell},x_{j_\ell}).
\]
Now define
\[
P_0\coloneqq \prod_{\ell=1}^{m}V_\ell,
\]
and for $s=1,\dots,m$ define
\[
P_s\coloneqq
\left(\prod_{\ell=1}^{s}U_\ell\right)
\left(\prod_{\ell=s+1}^{m}V_\ell\right).
\]
Thus $P_0$ is the product built entirely from $V$, while $P_m$ is the product built entirely from $U$.

Therefore
\[
\prod_{\ell=1}^{m}U_\ell-\prod_{\ell=1}^{m}V_\ell
=
P_m-P_0
=
\sum_{s=1}^{m}(P_s-P_{s-1}).
\]
For each $s$, only the $s$-th factor changes when one passes from $P_{s-1}$ to $P_s$, so
\[
P_s-P_{s-1}
=
\left(\prod_{\ell<s}U_\ell\right)(U_s-V_s)\left(\prod_{\ell>s}V_\ell\right).
\]
Substituting back the definitions of $U_\ell$ and $V_\ell$, we get
\begin{align*}
\prod_{\ell=1}^{m}U(x_{i_\ell},x_{j_\ell})
-
\prod_{\ell=1}^{m}V(x_{i_\ell},x_{j_\ell})
&=
\sum_{s=1}^{m}
\left(\prod_{\ell<s}U(x_{i_\ell},x_{j_\ell})\right)
D(x_{i_s},x_{j_s})
\left(\prod_{\ell>s}V(x_{i_\ell},x_{j_\ell})\right).
\end{align*}
If $s=1$ or $s=m$, one of the products is empty; by convention, an empty product is equal to $1$.

Substituting this into the integral and using the triangle inequality, we obtain
\[
|t(K_r,U)-t(K_r,V)|\le \sum_{s=1}^{m} I_s,
\]
where
\begin{align}
I_s
&\coloneqq
\left|
\int_{[0,1]^r}
\left(\prod_{\ell<s}U(x_{i_\ell},x_{j_\ell})\right)
D(x_{i_s},x_{j_s})
\left(\prod_{\ell>s}V(x_{i_\ell},x_{j_\ell})\right)
\,dx_1\cdots dx_r
\right|.
\label{eq:Is-Kr-cut}
\end{align}

We now estimate one such term.

Fix $s\in\{1,\dots,m\}$ and use
\[
(i_s,j_s)=(a,b).
\]
Let
\[
\widehat{x}=(x_t)_{t\in[r]\setminus\{a,b\}},
\qquad
d\widehat{x}\coloneqq
\prod_{t\in[r]\setminus\{a,b\}}dx_t,
\]
where the coordinates are taken in increasing order. Thus $\widehat{x}$
denotes all variables except $x_a$ and $x_b$. When $r=2$,
$[0,1]^0$ is interpreted as a one-point probability space, so integration
with respect to $d\widehat{x}$ means evaluation at that point.

Because every graphon takes values in $[0,1]$, the integrand in \eqref{eq:Is-Kr-cut} is measurable and bounded in absolute value by $1$.
Thus it is integrable on $[0,1]^r$.
We can therefore write the integral by first averaging over the two variables
$x_a$ and $x_b$, and then averaging over the remaining $r-2$ variables, which we collect in $\widehat{x}$.

The distinguished edge $\{a,b\}$ contributes the factor
\[
D(x_a,x_b).
\]

Every other edge of $K_r$ falls into exactly one of the following three types.

An edge of the form $\{a,c\}$ with $c\notin\{a,b\}$ contributes a factor
that depends on $x_a$ but not on $x_b$.

An edge of the form $\{b,c\}$ with $c\notin\{a,b\}$ contributes a factor
that depends on $x_b$ but not on $x_a$.

An edge of the form $\{c,d\}$ with $c,d\notin\{a,b\}$ contributes a factor
that depends on neither $x_a$ nor $x_b$; once $\widehat{x}$ is fixed, it is
constant.

Because $K_r$ is a simple graph, there is only one edge joining $a$ and $b$.
So the factor
\[
D(x_a,x_b)
\]
is the only factor that depends on both variables at the same time.

Figure~\ref{fig:Kr-cut-lipschitz-picture} illustrates this factorization in the case $K_4$.

\begin{figure}[htbp]
\centering
\includegraphics[width=0.75\linewidth]{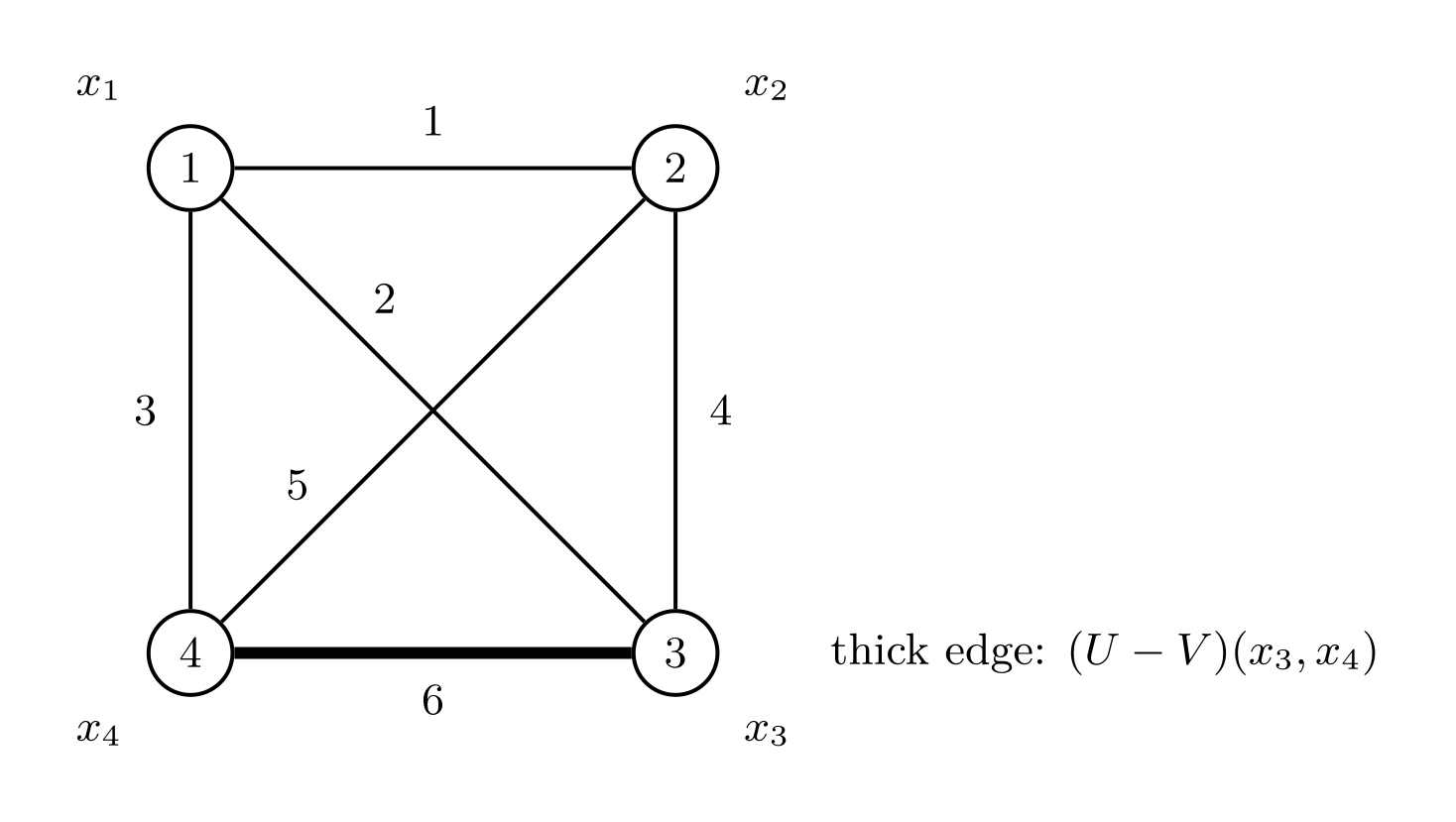}
\caption{In the integral defining $t(K_4,W)$, vertex $i$ carries the variable $x_i$, and the edge $\{i,j\}$ contributes the factor $W(x_i,x_j)$. If the distinguished edge is $\{3,4\}$, then that edge contributes the factor $(U-V)(x_3,x_4)$. Once $x_1$ and $x_2$ are fixed, the remaining factors split into a function of $x_3$, a function of $x_4$, and a constant.}
\label{fig:Kr-cut-lipschitz-picture}
\end{figure}

Therefore, for each fixed choice of $\widehat{x}$, the whole integrand in
\eqref{eq:Is-Kr-cut} can be written in the form
\[
C(\widehat{x})\,D(x_a,x_b)\,A_{\widehat{x}}(x_a)\,B_{\widehat{x}}(x_b),
\]
where $A_{\widehat{x}},B_{\widehat{x}}:[0,1]\to[0,1]$ are measurable and
$C(\widehat{x})\in[0,1]$. More concretely, $A_{\widehat{x}}$ is the product
of all remaining factors coming from edges that touch $a$ but not $b$,
$B_{\widehat{x}}$ is the product of all remaining factors coming from edges
that touch $b$ but not $a$, and $C(\widehat{x})$ is the product of all
remaining factors coming from edges that touch neither $a$ nor $b$. Empty
products are interpreted as $1$. These functions take values in $[0,1]$
because every graphon factor does.

Using this factorization and averaging first in $x_a$ and $x_b$, we get
\begin{equation}\label{eq:Is-fubini-cut-bound}
\begin{aligned}
I_s
&=
\left|
\int_{[0,1]^{r-2}}
\int_{[0,1]^2}
C(\widehat{x})\,D(x_a,x_b)\,A_{\widehat{x}}(x_a)\,B_{\widehat{x}}(x_b)
\,dx_a\,dx_b\,d\widehat{x}
\right| \\
&\le
\int_{[0,1]^{r-2}}
\left|
\int_{[0,1]^2}
D(x_a,x_b)\,\bigl(C(\widehat{x})A_{\widehat{x}}(x_a)\bigr)\,B_{\widehat{x}}(x_b)
\,dx_a\,dx_b
\right|
\,d\widehat{x}.
\end{aligned}
\end{equation}

For each fixed $\widehat{x}$, define
\[
f_{\widehat{x}}(x_a)\coloneqq C(\widehat{x})A_{\widehat{x}}(x_a),
\qquad
g_{\widehat{x}}(x_b)\coloneqq B_{\widehat{x}}(x_b).
\]
Then $f_{\widehat{x}},g_{\widehat{x}}:[0,1]\to[0,1]$ are measurable, so
Lemma~\ref{lem:cutnorm-weighted} applies and gives
\[
\left|
\int_{[0,1]^2}
D(x_a,x_b)\,f_{\widehat{x}}(x_a)\,g_{\widehat{x}}(x_b)
\,dx_a\,dx_b
\right|
\le
\|D\|_\square.
\]
Substituting this bound into \eqref{eq:Is-fubini-cut-bound} yields
\[
I_s
\le
\int_{[0,1]^{r-2}}\|D\|_\square\,d\widehat{x}
=
\|D\|_\square,
\]
because the set $[0,1]^{r-2}$ has total measure $1$.

This bound is the same for every $s\in[m]$.
Hence
\[
|t(K_r,U)-t(K_r,V)|
\le
\sum_{s=1}^{m} I_s
\le
\sum_{s=1}^{m}\|D\|_\square
=
m\,\|D\|_\square
=
\binom{r}{2}\,\|U-V\|_\square.
\]

This proves the first inequality.

We now pass from cut norm to cut distance.
Let $\phi$ be any measure-preserving bijection of $[0,1]$.
By Definition~\ref{defin:cutdistance},
\[
V^\phi(x,y)=V(\phi(x),\phi(y)).
\]
Define
\[
\Phi:[0,1]^r\to[0,1]^r,
\qquad
\Phi(x_1,\dots,x_r)\coloneqq (\phi(x_1),\dots,\phi(x_r)).
\]
Since $\phi$ preserves Lebesgue measure on $[0,1]$, the map $\Phi$ preserves the product measure on $[0,1]^r$.

Therefore
\[
t(K_r,V^\phi)=t(K_r,V).
\]
Applying the inequality already proved to $U$ and $V^\phi$, we obtain
\[
|t(K_r,U)-t(K_r,V)|
=
|t(K_r,U)-t(K_r,V^\phi)|
\le
\binom{r}{2}\,\|U-V^\phi\|_\square.
\]
Now take the infimum over all measure-preserving bijections $\phi$.
By the definition of $\delta_\square$,
\[
|t(K_r,U)-t(K_r,V)|
\le
\binom{r}{2}\,\delta_\square(U,V).
\]

Thus $W\mapsto t(K_r,W)$ is Lipschitz with respect to cut distance, and in particular it is continuous.
\end{proof}

\subsection{Entropy}

Entropy tells us, on the logarithmic scale, how many ways there are to choose edges and non-edges so that the edge density has a fixed value. The relevant one-variable function is the binary entropy
\[
h(p)\coloneqq -p\ln p-(1-p)\ln(1-p),\qquad p\in[0,1],
\]
with the convention $0\ln 0\coloneqq 0$.
If $pN$ is an integer, then the number of $0$--$1$ strings of length $N$ with exactly $pN$ ones is
$\binom{N}{pN}$, and for large $N$ the logarithm of this number has leading term $Nh(p)$. Thus
$h(p)$ measures the entropy contribution per possible edge when the edge density is $p$. In
particular,
\[
h(0)=h(1)=0,
\qquad
0\le h(p)\le \ln 2,
\qquad
h\Bigl(\frac12\Bigr)=\ln 2.
\]
So the entropy is zero when the presence or absence of an edge is completely determined, and it is
largest at $p=\tfrac12$, where the two possibilities are most evenly balanced.

\begin{defin}[Graphon entropy]\label{def:graphon-entropy}
For a graphon $W$, define

\begin{equation}\label{eq:entropy}
\Ent(W)\coloneqq
\frac12\int_{[0,1]^2} h\bigl(W(x,y)\bigr)\,dx\,dy.
\end{equation}

The factor $\frac12$ is needed because the square $[0,1]^2$ counts both $(x,y)$ and $(y,x)$,
while an edge is an unordered pair.
\end{defin}

\begin{lem}\label{lem:entropy-quadratic}
For every $p\in[0,1]$,
\begin{equation}\label{eq:entropy-quadratic}
h(p)\le \ln 2-2\Bigl(p-\frac12\Bigr)^2.
\end{equation}
\end{lem}

\begin{proof}
Set
\[
\psi(p)\coloneqq \ln 2-h(p)-2\Bigl(p-\frac12\Bigr)^2.
\]
For $p\in(0,1)$,
\[
h'(p)=\ln\frac{1-p}{p},
\]
so
\[
\psi'(p)
=
\ln\frac{p}{1-p}-4\Bigl(p-\frac12\Bigr).
\]
Also,
\[
\psi''(p)
=
\frac{1}{p(1-p)}-4
\ge
0,
\]
because $p(1-p)\le1/4$. Thus $\psi$ is convex on $(0,1)$ and extends continuously
to $[0,1]$. Since
\[
\psi\Bigl(\frac12\Bigr)=0,
\qquad
\psi'\Bigl(\frac12\Bigr)=0,
\]
the point $p=1/2$ is a global minimizer. Hence $\psi(p)\ge0$ for every $p\in[0,1]$,
which gives \eqref{eq:entropy-quadratic}.
\end{proof}

\begin{lem}\label{lem:entropy-zero-cut}
If $U,V\in\mc W$ satisfy $\delta_\square(U,V)=0$, then
\[
\Ent(U)=\Ent(V).
\]
Consequently,
\[
\Ent(\widetilde W)\coloneqq \Ent(W),
\qquad \widetilde W=\pi(W),
\]
is a well-defined function on $\widetilde{\mc W}$.
\end{lem}

\begin{proof}
By Remark~\ref{rem:zero-cut-distance}, there are measure-preserving maps
$\phi,\psi:[0,1]\to[0,1]$ such that $U^\phi=V^\psi$ almost everywhere. The product maps
$\phi\times\phi$ and $\psi\times\psi$ preserve $\lambda^2$, hence
\[
\Ent(U)=\Ent(U^\phi)=\Ent(V^\psi)=\Ent(V).
\]
\end{proof}

Graphon entropy controls the number of finite graphs with prescribed
large-scale structure. More precisely, for a cut-closed family $\mc F$, the
$n^2$-scale upper bound is governed by
$\sup_{W\in\mc F}\Ent(W)$, rather than by the entropy of an arbitrarily
chosen member of $\mc F$. In Section~\ref{sec:free-energy},
Lemma~\ref{lem:entropy-counting} uses such an entropy supremum to bound the
number of graphs whose empirical graphons lie in a given family. We
combine this result with Lemma~\ref{lem:near-gap-Kr} in the proof of
Theorem~\ref{thm:main-clique}\textup{(ii)} to show that the probability of
staying a fixed cut distance away from $B_q^\star$ decreases exponentially
in $n^2$. Lemma~\ref{lem:boundary-max-Kr} shows that, among graphons with
zero $K_r$-density, entropy is maximized exactly by the graphons at zero cut
distance from $B_q^\star$.

\begin{lem}\label{lem:empirical-Kr-density}
For every $n\in\mbbN$, every graph $G$ with vertex set $[n]$, and every integer $r\ge2$,
\[
t(K_r,W_G)=\frac{r!\,K_r(G)}{n^r}.
\]
\end{lem}

\begin{proof}
Let $(a_{ij})_{1\le i,j\le n}$ be the adjacency matrix of $G$, and let
$I_1,\dots,I_n$ be the intervals from
Definition~\ref{def:empirical-graphon}. Splitting $[0,1]^r$ into the boxes
\[
I_{v_1}\times\cdots\times I_{v_r},
\qquad
v_1,\dots,v_r\in[n],
\]
gives
\begin{equation}\label{eq:empirical-Kr-adjacency-expansion}
t(K_r,W_G)
=
\frac1{n^r}
\sum_{v_1,\dots,v_r\in[n]}
\prod_{1\le s<t\le r}a_{v_s,v_t}.
\end{equation}

Since every entry of the adjacency matrix is either $0$ or $1$, the
summand corresponding to $(v_1,\dots,v_r)$ is equal to $1$ exactly when
the vertices $v_1,\dots,v_r$ are pairwise distinct and form a copy of
$K_r$ in $G$. Indeed, if $v_s=v_t$ for some $s<t$, then the product
contains the factor
\[
a_{v_s,v_t}=a_{v_s,v_s}=0,
\]
because $G$ has no loops.

Thus the sum in \eqref{eq:empirical-Kr-adjacency-expansion} counts the
ordered $r$-tuples obtained by ordering the vertices of copies of $K_r$
in $G$. Every unordered copy has exactly $r!$ such orderings, and therefore
\begin{equation}\label{eq:empirical-Kr-ordered-count}
\sum_{v_1,\dots,v_r\in[n]}
\prod_{1\le s<t\le r}a_{v_s,v_t}
=
r!\,K_r(G).
\end{equation}
Combining \eqref{eq:empirical-Kr-adjacency-expansion} and
\eqref{eq:empirical-Kr-ordered-count} gives
\[
t(K_r,W_G)=\frac{r!\,K_r(G)}{n^r},
\]
as required.
\end{proof}

\begin{rem}
\label{rem:empirical-Kr-normalization}
For $n\ge r$, Lemma~\ref{lem:empirical-Kr-density} gives
\[
t(K_r,W_G)
=
\frac{n(n-1)\cdots(n-r+1)}{n^r}\,
\frac{K_r(G)}{\binom nr}.
\]
Thus the graphon homomorphism-density normalization differs from the usual clique density
based on unordered copies by the factor
\[
\frac{n(n-1)\cdots(n-r+1)}{n^r}.
\]
For fixed $r$,
\[
\lim_{n\to\infty}
\frac{n(n-1)\cdots(n-r+1)}{n^r}
=
1.
\]
\end{rem}

\section{Entropy maximizers and stability}\label{sec:entropy-stability}

\noindent
In this section we define a family of ``balanced $q$-partite graphons'' (recall $q := r-1$) and prove results
which relate the distance of a graphon $W$ to a balanced $q$-partite graphon to the homomorphism density of
$K_r$ in $W$ and the entropy of $W$.
In particular, we show that the balanced $q$-partite graphons are the graphons that maximize entropy among the
graphons with $K_r$-homomorphism density 0.
This will be used in the proof of Theorem~\ref{thm:main-clique}
in Section~\ref{sec:free-energy}.

\subsection{The family \texorpdfstring{$B_q^\star$}{Bq*} and the balanced Tur\'an graphons}

\begin{defin}\label{def:Bqstar}
Let $q\ge 2$.
The family $B_q^\star$ consists of all graphons $W:[0,1]^2\to[0,1]$ for which there exists a measurable
partition
\[
[0,1]=A_1\sqcup\cdots\sqcup A_q,
\qquad
\lambda(A_i)=\frac1q
\quad\text{for every }i\in[q],
\]
such that
\[
W(x,y)=0
\quad\text{for almost every\ }(x,y)\in A_i\times A_i
\qquad\text{for every }i,
\]
and
\[
W(x,y)=\tfrac12
\quad\text{for almost every\ }(x,y)\in A_i\times A_j
\qquad\text{for all distinct }i,j\in[q].
\]
Equivalently, the values of $W$ on the $q^2$ block products are given by the
following block-value matrix:
\[
\begin{bmatrix}
0 & \tfrac12 & \cdots & \tfrac12\\
\tfrac12 & 0 & \cdots & \tfrac12\\
\vdots & \vdots & \ddots & \vdots\\
\tfrac12 & \tfrac12 & \cdots & 0
\end{bmatrix}
\]
with blocks of equal size.
\end{defin}

So $B_q^\star$ is the graphon analogue of the family of balanced complete $q$-partite graphs, except that the
off-diagonal blocks have value $\tfrac12$ instead of $1$.

\begin{notation}[Balanced complete $q$-partite graphons]\label{not:Uqstar}
Define
\[
\mc U_q^\star
\coloneqq
\left\{
\mathbf{1}_{\displaystyle\bigcup_{\substack{1\le i,j\le q\\ i\neq j}}A_i\times A_j}:\;
\begin{array}{l}
A_1,\dots,A_q\subseteq[0,1]\text{ are measurable},\\
{}[0,1]=A_1\sqcup\cdots\sqcup A_q,\\
\lambda(A_i)=\frac1q\text{ for every }i\in[q]
\end{array}
\right\}.
\]
These are the balanced complete $q$-partite $\{0,1\}$-valued graphons. By construction,
\[
B_q^\star
=
\left\{
W\in\mc W:\;
W=\frac12\,U\ \text{almost everywhere for some }U\in\mc U_q^\star
\right\}.
\]
\end{notation}

\begin{notation}[Canonical representative of \texorpdfstring{$B_q^\star$}{Bq*}]
\label{not:Wqstar}
We fix one balanced partition of $[0,1]$ by setting
\[
A_a\coloneqq \Bigl[\frac{a-1}{q},\frac{a}{q}\Bigr)
\quad\text{for }a=1,\dots,q-1,
\qquad
A_q\coloneqq \Bigl[\frac{q-1}{q},1\Bigr].
\]
Thus
\[
[0,1]=A_1\sqcup\cdots\sqcup A_q,
\qquad
\lambda(A_a)=\frac1q
\quad\text{for every }a\in[q].
\]

Define the canonical element $W_q^\star\in B_q^\star$ by
\[
W_q^\star(x,y)
\coloneqq
\begin{cases}
0, & \text{if }x,y\in A_a\text{ for some }a\in[q],\\[1mm]
\frac12, & \text{if }x\in A_a,\ y\in A_b
\text{ for some distinct }a,b\in[q].
\end{cases}
\]
Equivalently,
\[
W_q^\star
=
\frac12\,
\mathbf{1}_{\displaystyle
\bigcup_{\substack{1\le a,b\le q\\ a\neq b}}A_a\times A_b}.
\]

Every graphon in $B_q^\star$ is equal almost everywhere to a relabeling
of $W_q^\star$. Indeed, suppose that $V\in B_q^\star$ corresponds to a
balanced partition $[0,1]=B_1\sqcup\cdots\sqcup B_q$. For each $i\in[q]$, both $B_i$ and $A_i$ are measurable subsets of
$[0,1]$ with Lebesgue measure $1/q$. Equip these sets with their normalized
restricted Lebesgue measures. By \cite[Cor.~A.11]{Janson2013}, there is a
measure-preserving bijection
\[
\phi_i:B_i\to A_i.
\]
Since the same normalization factor $q$ is used on both sets, $\phi_i$ also
preserves the original Lebesgue measure. Pasting
$\phi_1,\dots,\phi_q$ gives a measure-preserving bijection
$\phi:[0,1]\to[0,1]$ such that
\[
\phi(B_i)=A_i
\quad\text{up to null sets for every }i\in[q].
\]

Therefore
\[
V=(W_q^\star)^\phi
\qquad\text{almost everywhere.}
\]
In particular,
\[
\delta_\square(V,W_q^\star)=0
\qquad\text{for every }V\in B_q^\star.
\]

Define
\[
\widetilde W_q^\star\coloneqq \pi(W_q^\star),
\qquad
\widetilde B_q^\star\coloneqq \pi(B_q^\star).
\]
Then
\[
\widetilde B_q^\star=\{\widetilde W_q^\star\}.
\]
Consequently, for every graphon $W$,
\[
\delta_\square(W,B_q^\star)=\delta_\square(W,W_q^\star).
\]
\end{notation}

\begin{lem}\label{lem:graphon-turan}
Fix $r\ge 3$ and set $q=r-1$.
Let $U$ be a $\{0,1\}$-valued graphon and assume
\[
t(K_r,U)=0.
\]
Then
\[
\int_{[0,1]^2}U(x,y)\,dx\,dy\le 1-\frac1q.
\]
Moreover, equality holds if and only if
\[
\delta_\square(U,\mc U_q^\star)=0.
\]
\end{lem}

\begin{proof}
This is the $K_r$-free case of the graphon Tur\'an theorem; see
\cite[Cor.~16.11]{Lovasz2012}. In the present notation, its equality case is
exactly the condition
\[
\delta_\square(U,\mc U_q^\star)=0.
\]
\end{proof}

\begin{lem}[Entropy maximizers among graphons of zero \texorpdfstring{$K_r$}{Kr}-density]\label{lem:boundary-max-Kr}
Fix $r\ge 3$ and set $q=r-1$.
Then
\[
\sup\{\Ent(W):\ t(K_r,W)=0\}
=
\frac{\ln 2}{2}\Bigl(1-\frac1q\Bigr).
\]
A graphon $W$ with $t(K_r,W)=0$ attains this supremum if and only if
\[
\delta_\square(W,B_q^\star)=0.
\]
Equivalently, the reduced maximizer set is the singleton
$\widetilde B_q^\star=\{\widetilde W_q^\star\}$.
\end{lem}

\begin{rem}
Lemma~\ref{lem:boundary-max-Kr} can also be deduced from
\cite[Thm.~8.2]{ChatterjeeDiaconis2013}. In that theorem, take $H=K_r$ and
$p=1/2$. Since $\chi(K_r)=r$, the graphon $g$ appearing there is the balanced
$(r-1)$-partite $\{0,1\}$-valued graphon. Hence $pg=\frac12 g$ is a
representative of the reduced class $\widetilde W_q^\star$; equivalently,
\[
\pi(pg)=\widetilde W_q^\star.
\]

The off-diagonal support of $g$ has measure $1-1/q$. Therefore
\[
\Ent\left(\frac12g\right)
=
\frac{\ln 2}{2}\left(1-\frac1q\right),
\]
which also gives the value of the constrained maximum.

For $p=1/2$, the Erd\H{o}s-R\'enyi rate function $I_{\frac12}$, defined in
\eqref{eq:CV-rate-half} below, satisfies
\[
I_{\frac12}(W)
=
\frac{\ln 2}{2}-\Ent(W).
\]
Thus minimizing $I_{\frac12}$ under the constraint $t(K_r,W)=0$ is the same as
maximizing $\Ent(W)$ under the same constraint. We nevertheless give a direct
proof below, using the graphon Tur\'an theorem, because it is short and makes
the equality cases explicit.
\end{rem}

\begin{proof}[Proof of Lemma~\ref{lem:boundary-max-Kr}]
Let $W$ be a graphon satisfying
\[
t(K_r,W)=0.
\]
Set
\[
S\coloneqq\{(x,y)\in[0,1]^2:W(x,y)>0\},
\qquad
U\coloneqq\mathbf{1}_S.
\]
Since $W$ is measurable and symmetric almost everywhere, the function $U$ is
also measurable and symmetric almost everywhere. Also, $U$ only takes the
values $0$ and $1$. Hence $U$ is a graphon.

We first show that $U$ has zero $K_r$-density. By
Remark~\ref{rem:Kr-hom-density}, the integrand defining $t(K_r,W)$ is
nonnegative. Since $t(K_r,W)=0$, we have
\[
\prod_{1\le a<b\le r} W(x_a,x_b)=0
\]
for almost every $(x_1,\dots,x_r)\in[0,1]^r$. If
\[
\prod_{1\le a<b\le r} U(x_a,x_b)=1,
\]
then every factor $U(x_a,x_b)$ is equal to $1$. By the definition of $U$,
this means that every corresponding factor $W(x_a,x_b)$ is strictly positive.
Thus
\[
\prod_{1\le a<b\le r} W(x_a,x_b)>0.
\]
This can happen only on a null set. Therefore
\[
t(K_r,U)=0.
\]

We now bound the entropy of $W$. Since $W=0$ on $S^c$ and $h(0)=0$,
Definition~\ref{def:graphon-entropy} gives
\[
\Ent(W)
=
\frac12\int_S h(W(x,y))\,dx\,dy.
\]
By Lemma~\ref{lem:entropy-quadratic}, we have $h(p)\le \ln 2$ for every
$p\in[0,1]$. Hence
\begin{equation}
\label{eq:boundary-support-entropy-bound}
\Ent(W)
\le
\frac{\ln 2}{2}\lambda^2(S)
=
\frac{\ln 2}{2}\int_{[0,1]^2}U(x,y)\,dx\,dy.
\end{equation}
Since $U$ is a $\{0,1\}$-valued graphon and $t(K_r,U)=0$,
Lemma~\ref{lem:graphon-turan} gives
\begin{equation}
\label{eq:boundary-support-turan-bound}
\int_{[0,1]^2}U(x,y)\,dx\,dy
\le
1-\frac1q.
\end{equation}
Combining \eqref{eq:boundary-support-entropy-bound} and
\eqref{eq:boundary-support-turan-bound}, we get
\[
\Ent(W)
\le
\frac{\ln 2}{2}\left(1-\frac1q\right).
\]
Therefore
\[
\sup\{\Ent(W):t(K_r,W)=0\}
\le
\frac{\ln 2}{2}\left(1-\frac1q\right).
\]

Next we show that this upper bound is attained. Consider the canonical
graphon $W_q^\star$ from Notation~\ref{not:Wqstar}. The graphon
$W_q^\star$ has value $0$ on the diagonal blocks and value $1/2$ on the
off-diagonal blocks. Since $q=r-1$, any choice of $r$ points in $[0,1]$
contains two points in the same canonical block. For that pair, the
corresponding factor in the $K_r$-integrand is $0$. Hence
\begin{equation}
\label{eq:boundary-Wqstar-zero-Kr}
t(K_r,W_q^\star)=0.
\end{equation}

The diagonal blocks have total measure $1/q$, and the off-diagonal blocks
have total measure $1-1/q$. Using Definition~\ref{def:graphon-entropy},
together with
\[
h(0)=0
\qquad\text{and}\qquad
h\left(\frac12\right)=\ln 2,
\]
we obtain
\begin{equation}
\label{eq:boundary-Wqstar-entropy}
\Ent(W_q^\star)
=
\frac12\left(1-\frac1q\right)h\left(\frac12\right)
=
\frac{\ln 2}{2}\left(1-\frac1q\right).
\end{equation}
Thus
\[
\sup\{\Ent(W):t(K_r,W)=0\}
=
\frac{\ln 2}{2}\left(1-\frac1q\right).
\]

It remains to identify the graphons that attain this value.

Suppose first that $W$ is an entropy maximizer under the constraint
\[
t(K_r,W)=0.
\]
Use the same notation as above,
\[
S=\{(x,y)\in[0,1]^2:W(x,y)>0\},
\qquad
U=\mathbf{1}_S.
\]
Since $W$ attains the value in \eqref{eq:boundary-Wqstar-entropy}, both
inequalities \eqref{eq:boundary-support-entropy-bound} and
\eqref{eq:boundary-support-turan-bound} must be equalities.

Equality in \eqref{eq:boundary-support-turan-bound}, by the equality case in
Lemma~\ref{lem:graphon-turan}, gives
\[
\delta_\square(U,\mc U_q^\star)=0.
\]
Equality in \eqref{eq:boundary-support-entropy-bound} gives
\[
h(W(x,y))=\ln 2
\qquad
\text{for almost every }(x,y)\in S.
\]
By Lemma~\ref{lem:entropy-quadratic}, equality $h(p)=\ln 2$ is possible only
when $p=1/2$. Hence
\[
W(x,y)=\frac12
\qquad
\text{for almost every }(x,y)\in S.
\]
Also, by the definition of $S$, we have $W=0$ almost everywhere on $S^c$.
Therefore
\[
W=\frac12\,U
\qquad
\text{almost everywhere}.
\]

By Notation~\ref{not:Uqstar}, the family $B_q^\star$ consists exactly of the
graphons $\frac12 V$ with $V\in\mc U_q^\star$, up to almost-everywhere
equality. Since
\[
\delta_\square(U,\mc U_q^\star)=0,
\]
we get
\[
\delta_\square\left(\frac12U,B_q^\star\right)=0.
\]
Indeed, for every $\eta>0$, choose $V_\eta\in\mc U_q^\star$ such that
\[
\delta_\square(U,V_\eta)<2\eta.
\]
Then $\frac12 V_\eta\in B_q^\star$, and by homogeneity of the cut norm,
\[
\delta_\square\left(\frac12U,\frac12V_\eta\right)
\le
\frac12\,\delta_\square(U,V_\eta)
<
\eta.
\]
Since $W=\frac12U$ almost everywhere, it follows that
\[
\delta_\square(W,B_q^\star)=0.
\]

Conversely, suppose that
\[
\delta_\square(W,B_q^\star)=0.
\]
By Notation~\ref{not:Wqstar}, this is equivalent to
\[
\delta_\square(W,W_q^\star)=0.
\]
Using Lemma~\ref{lem:Kr-cut-lipschitz} and
\eqref{eq:boundary-Wqstar-zero-Kr}, we get
\[
t(K_r,W)=t(K_r,W_q^\star)=0.
\]
Using Lemma~\ref{lem:entropy-zero-cut} and
\eqref{eq:boundary-Wqstar-entropy}, we get
\[
\Ent(W)=\Ent(W_q^\star)
=
\frac{\ln 2}{2}\left(1-\frac1q\right).
\]
Thus $W$ satisfies the constraint $t(K_r,W)=0$ and attains the maximum
entropy.

We have shown that the entropy maximizers are exactly the graphons $W$
satisfying
\[
\delta_\square(W,B_q^\star)=0.
\]
Finally, Notation~\ref{not:Wqstar} gives
\[
\widetilde B_q^\star=\{\widetilde W_q^\star\},
\]
so the reduced maximizer set is the singleton
$\{\widetilde W_q^\star\}$.
\end{proof}

\begin{rem}
\label{rem:relation-to-fixed-parameter-ergms}
By Lemma~\ref{lem:empirical-Kr-density}, the reduced mass of
$G\in\mbG_n$ is
\[
\exp\bigl(-r!\,w\,K_r(G)\bigr)
=
\exp\left(-n^2L_n\,t(K_r,W_G)\right),
\qquad
L_n\coloneqq wn^{r-2}.
\]
The corresponding graphon functional is therefore
\[
T_n(W)=-L_n\,t(K_r,W).
\]

For fixed $L>0$, let
\[
\widetilde{\mc F}_L^\star
\coloneqq
\pi\left(
\operatorname*{arg\,max}_{W\in\mc W}
\left\{
\Ent(W)-L\,t(K_r,W)
\right\}
\right).
\]
For the ERGM with zero edge coefficient and clique coefficient $-L$,
\cite[Thm.~3.2]{ChatterjeeDiaconis2013} gives exponential concentration
around $\widetilde{\mc F}_L^\star$. The concentration rate is not given
uniformly in $L$. Moreover, \cite[Thm.~7.1]{ChatterjeeDiaconis2013} gives
\[
\lim_{L\to\infty}
\sup_{\widetilde W\in\widetilde{\mc F}_L^\star}
d_\square\left(\widetilde W,\widetilde W_q^\star\right)
=
0.
\]
These results describe an iterated limit in which one first lets
$n\to\infty$ with $L$ fixed and then lets $L\to\infty$.
In the fixed-weight MLN considered here,
\[
L_n=wn^{r-2},
\qquad
\lim_{n\to\infty}L_n=+\infty.
\]
Thus $n$ and $L_n$ grow together.
Since the concentration
rate for fixed $L$ is not uniform in $L$, the fixed-coefficient theorem does
not directly give a bound of the form $e^{-cn^2}$ with $c>0$ independent of
$n$ along this sequence.
\end{rem}

\subsection{Entropy gap away from the maximizers}

The next two lemmas show that once we stay a fixed cut distance away from the
reduced maximizer class represented by $B_q^\star$, the entropy drops.

\begin{lem}\label{lem:turan-stab}
Fix $r\ge 3$ and set $q=r-1$.
For every $\varepsilon>0$ there exists $\gamma_r(\varepsilon)>0$ such that the following holds.

If $U$ is a $\{0,1\}$-valued graphon and satisfies
\[
t(K_r,U)=0,
\qquad
\delta_\square(U,\mc U_q^\star)\ge \varepsilon,
\]
then
\[
\int_{[0,1]^2}U\le \Bigl(1-\frac1q\Bigr)-\gamma_r(\varepsilon).
\]
\end{lem}

\begin{proof}
Fix $\varepsilon>0$ and define
\[
m_\varepsilon
\coloneqq
\sup\left\{
\int_{[0,1]^2} U :
\begin{array}{l}
U \text{ is a $\{0,1\}$-valued graphon},\\
t(K_r,U)=0,\\
\delta_\square(U,\mc U_q^\star)\ge \varepsilon
\end{array}
\right\}.
\]
If the set inside the supremum is empty, take
\[
\gamma_r(\varepsilon)
=
\frac12\left(1-\frac1q\right).
\]
The conclusion then holds automatically. Hence assume that the set is nonempty.

By Lemma~\ref{lem:graphon-turan},
\[
m_\varepsilon\le 1-\frac1q.
\]
We claim that the inequality is strict.

Assume instead that
\[
m_\varepsilon=1-\frac1q.
\]
By the definition of the supremum, there exists a sequence $(U_n)$ of $\{0,1\}$-valued graphons such that
\[
t(K_r,U_n)=0,
\qquad
\delta_\square(U_n,\mc U_q^\star)\ge \varepsilon,
\qquad
\lim_{n\to\infty}\int_{[0,1]^2}U_n=1-\frac1q.
\]
Let $\widetilde U_n\in\widetilde{\mc W}$ be the equivalence class of $U_n$.
By Notation~\ref{not:reduced-graphon-space}, the metric space
$(\widetilde{\mc W},d_\square)$ is compact; see
\cite[Thm.~9.23]{Lovasz2012}. Hence, after passing to a subsequence, we may
assume that
\[
\lim_{n\to\infty}d_\square(\widetilde U_n,\widetilde U)=0
\]
for some $\widetilde U\in\widetilde{\mc W}$.
Choose any representative $U\in\widetilde U$.
Since the metric on $\widetilde{\mc W}$ is induced by $\delta_\square$, this means exactly that
\[
\lim_{n\to\infty}\delta_\square(U_n,U)=0.
\]

By Lemma~\ref{lem:Kr-cut-lipschitz}, the map $t(K_r,\cdot)$ is continuous with respect to cut distance. Therefore
\[
\lim_{n\to\infty}t(K_r,U_n)=t(K_r,U)=0.
\]
Also, the limit still stays at least $\varepsilon$ away from $\mc U_q^\star$.
Indeed,
\[
\delta_\square(U,\mc U_q^\star)
\ge
\delta_\square(U_n,\mc U_q^\star)-\delta_\square(U_n,U),
\]
and letting $n\to\infty$ gives

\begin{equation}\label{eq:delta-away-from-Uqstar}
\delta_\square(U,\mc U_q^\star)\ge \varepsilon.
\end{equation}

We also need the edge density to pass to the limit.
For graphons $W$ and $V$, and for any measure-preserving bijection $\phi$,
\[
\int_{[0,1]^2}V
=
\int_{[0,1]^2}V^\phi
\]
by a change of variables. Therefore
\[
\left|\int_{[0,1]^2}(W-V)\right|
=
\left|\int_{[0,1]^2}(W-V^\phi)\right|.
\]
By Definition~\ref{defin:cutnorm}, we may choose $S=T=[0,1]$ in the
supremum defining the cut norm, so
\[
\left|\int_{[0,1]^2}(W-V^\phi)\right|
\le
\|W-V^\phi\|_\square.
\]
Taking the infimum over all measure-preserving bijections $\phi$ gives
\[
\left|\int_{[0,1]^2}(W-V)\right|
\le
\delta_\square(W,V).
\]
Apply this to $(W,V)=(U_n,U)$ and let $n\to\infty$:

\begin{equation}\label{eq:edge-density-limit}
\int_{[0,1]^2} U
=
\lim_{n\to\infty}\int_{[0,1]^2} U_n
=
1-\frac1q.
\end{equation}

Now define the positivity set of $U$ by
\[
S\coloneqq \{(x,y)\in[0,1]^2:\ U(x,y)>0\},
\qquad
V\coloneqq \mathbf{1}_S.
\]
The function $V$ is measurable, $\{0,1\}$-valued, and symmetric almost everywhere; hence it is a graphon. Moreover,
\[
0\le U\le V\le 1.
\]

We show first that $V$ is a graphon of zero $K_r$-density.
If
\[
\prod_{1\le a<b\le r}V(x_a,x_b)=1,
\]
then each factor $U(x_a,x_b)$ is strictly positive, so
\[
\prod_{1\le a<b\le r}U(x_a,x_b)>0.
\]
But the integrand defining $t(K_r,U)$ is nonnegative and has integral $0$.
So this can happen only on a set of measure $0$.
Hence
\[
t(K_r,V)=0.
\]

Because $U\le V$, equation~\eqref{eq:edge-density-limit} gives
\[
\int_{[0,1]^2}V\ge \int_{[0,1]^2}U=1-\frac1q.
\]
On the other hand, $V$ is a $\{0,1\}$-valued graphon of zero $K_r$-density, so Lemma~\ref{lem:graphon-turan}
gives
\[
\int_{[0,1]^2}V\le 1-\frac1q.
\]
Thus
\[
\int_{[0,1]^2}V=1-\frac1q.
\]
By the equality case in Lemma~\ref{lem:graphon-turan}, this implies
\[
\delta_\square(V,\mc U_q^\star)=0.
\]

Finally,
\[
V-U\ge 0
\]
pointwise and
\[
\int_{[0,1]^2}(V-U)
=
\int_{[0,1]^2}V-\int_{[0,1]^2}U
=
0.
\]
A nonnegative integrable function with integral $0$ is zero almost everywhere.
So
\[
U=V
\qquad\text{almost everywhere}.
\]
Therefore
\[
\delta_\square(U,\mc U_q^\star)=0,
\]
contradicting the earlier bound $\delta_\square(U,\mc U_q^\star)\ge \varepsilon$ in \eqref{eq:delta-away-from-Uqstar}.

This contradiction shows that
\[
m_\varepsilon<1-\frac1q.
\]
Now define
\[
\gamma_r(\varepsilon)\coloneqq \Bigl(1-\frac1q\Bigr)-m_\varepsilon>0.
\]
This is exactly the desired gap.
\end{proof}

\begin{lem}[Entropy gap away from \texorpdfstring{$B_q^\star$}{Bq*}]\label{lem:gap-Kr}
Fix $r\ge 3$ and set $q=r-1$.
For every $\varepsilon>0$ there exists $c_{r,\varepsilon}>0$ such that every graphon $W$
satisfying
\[
t(K_r,W)=0,
\qquad
\delta_\square(W,B_q^\star)\ge \varepsilon
\]
also satisfies
\[
\Ent(W)\le \frac{\ln 2}{2}\Bigl(1-\frac1q\Bigr)-c_{r,\varepsilon}.
\]
\end{lem}

\begin{proof}
Fix $\varepsilon>0$ and let $W$ be a graphon with
\[
t(K_r,W)=0,
\qquad
\delta_\square(W,B_q^\star)\ge \varepsilon.
\]
As before, set
\[
S\coloneqq \{(x,y)\in[0,1]^2:\ W(x,y)>0\},
\qquad
U\coloneqq \mathbf{1}_S.
\]
The function $U$ is measurable, $\{0,1\}$-valued, and symmetric almost everywhere; hence it is a graphon.

Since $t(K_r,W)=0$, Remark~\ref{rem:Kr-hom-density} implies that
\[
\prod_{1\le a<b\le r} W(x_a,x_b)=0
\]
for almost every $(x_1,\dots,x_r)\in[0,1]^r$. If
\[
\prod_{1\le a<b\le r} U(x_a,x_b)=1,
\]
then every factor $W(x_a,x_b)$ is strictly positive, and therefore
\[
\prod_{1\le a<b\le r} W(x_a,x_b)>0.
\]
This can happen only on a null set. Hence
\[
t(K_r,U)=0.
\]
Moreover, $W=0$ on $S^c$, while $h(u)\le\ln 2$ for every $u\in[0,1]$. Therefore
\[
\Ent(W)
=
\frac12\int_S h(W(x,y))\,dx\,dy
\le
\frac{\ln 2}{2}\lambda^2(S)
=
\frac{\ln 2}{2}\int_{[0,1]^2}U.
\]

We now split into two cases.

\smallskip
\noindent\textbf{Case 1: }$\delta_\square(U,\mc U_q^\star)\ge \varepsilon/2$.

Then Lemma~\ref{lem:turan-stab} gives
\[
\int_{[0,1]^2}U
\le
\Bigl(1-\frac1q\Bigr)-\gamma_r(\varepsilon/2).
\]
Hence
\[
\Ent(W)
\le
\frac{\ln 2}{2}\Bigl(1-\frac1q\Bigr)
-
\frac{\ln 2}{2}\,\gamma_r(\varepsilon/2).
\]
Thus, in Case~1, the claimed bound holds with entropy loss $\frac{\ln 2}{2}\gamma_r(\varepsilon/2)$.

\smallskip
\noindent\textbf{Case 2: }$\delta_\square(U,\mc U_q^\star)< \varepsilon/2$.

Choose $U^\star\in \mc U_q^\star$ such that
\[
\delta_\square(U,U^\star)<\frac{\varepsilon}{2}.
\]
Set
\[
V^\star\coloneqq \frac12\,U^\star.
\]
By Notation~\ref{not:Uqstar}, the family $B_q^\star$ consists exactly of graphons that are equal almost everywhere to $\frac12 U$ for some $U\in\mc U_q^\star$. Since $U^\star\in\mc U_q^\star$ and $V^\star=\frac12U^\star$, it follows that $V^\star\in B_q^\star$.
Since $\delta_\square(W,B_q^\star)\ge \varepsilon$, we have
\[
\delta_\square(W,V^\star)\ge \varepsilon.
\]
Also, by homogeneity of the cut distance,
\[
\delta_\square\Bigl(\frac12U,V^\star\Bigr)
=
\delta_\square\Bigl(\frac12U,\frac12U^\star\Bigr)
=
\frac12\,\delta_\square(U,U^\star)
<
\frac{\varepsilon}{4}.
\]
So the triangle inequality gives
\[
\delta_\square\Bigl(W,\frac12U\Bigr)
\ge
\delta_\square(W,V^\star)-\delta_\square\Bigl(\frac12U,V^\star\Bigr)
>
\frac{3\varepsilon}{4}.
\]
Since
\[
\delta_\square\Bigl(W,\frac12U\Bigr)
\le
\Bigl\|W-\frac12U\Bigr\|_\square
\le
\Bigl\|W-\frac12U\Bigr\|_1,
\]
it follows that
\[
\Bigl\|W-\frac12U\Bigr\|_1>\frac{3\varepsilon}{4}.
\]
Because $U=\mathbf{1}_S$ and $W=0$ on $S^c$, this means
\[
\int_S\left|W-\frac12\right|>\frac{3\varepsilon}{4}.
\]

Since $U$ is $\{0,1\}$-valued and is a graphon of zero $K_r$-density, Lemma~\ref{lem:graphon-turan} gives
\[
\lambda^2(S)=\int_{[0,1]^2}U\le 1-\frac1q.
\]
By Lemma~\ref{lem:cauchy-schwarz}\textup{(i)},
\[
\left(\int_S\left|W-\frac12\right|\right)^2
\le
\lambda^2(S)\int_S\left(W-\frac12\right)^2.
\]
Therefore
\[
\int_S\left(W-\frac12\right)^2
\ge
\frac{\left(\frac{3\varepsilon}{4}\right)^2}{1-\frac1q}.
\]

For every $(x,y)\in S$, we also have $W(x,y)\in[0,1]$. Hence
Lemma~\ref{lem:entropy-quadratic} gives
\[
h(W(x,y))
\le
\ln 2-2\Bigl(W(x,y)-\frac12\Bigr)^2
\qquad\text{for every }(x,y)\in S.
\]
Using this inequality in the entropy identity above, we obtain
\begin{align*}
\Ent(W)
&\le
\frac12\int_S
\left(\ln 2-2\Bigl(W(x,y)-\frac12\Bigr)^2\right)\,dx\,dy\\
&=
\frac{\ln 2}{2}\,\lambda^2(S)
-
\int_S\Bigl(W(x,y)-\frac12\Bigr)^2\,dx\,dy\\
&\le
\frac{\ln 2}{2}\Bigl(1-\frac1q\Bigr)
-
\frac{9}{16}\frac{\varepsilon^2}{(1-\frac1q)}.
\end{align*}
Here the last inequality uses the two estimates already proved above:
\[
\lambda^2(S)\le 1-\frac1q
\qquad\text{and}\qquad
\int_S\Bigl(W(x,y)-\frac12\Bigr)^2\,dx\,dy
\ge
\frac{9}{16}\frac{\varepsilon^2}{(1-\frac1q)}.
\]

Combine the two cases and define
\[
c_{r,\varepsilon}
\coloneqq
\min\left\{
\frac{\ln 2}{2}\,\gamma_r(\varepsilon/2),
\frac{9}{16}\frac{\varepsilon^2}{(1-\frac1q)}
\right\}>0.
\]
Then
\[
\Ent(W)\le \frac{\ln 2}{2}\Bigl(1-\frac1q\Bigr)-c_{r,\varepsilon}.
\]
\end{proof}

The graphon removal lemma in \cite[Lem.~2.3]{HladkyHuPiguet2021} is stated
with the strict assumption $t(K_r,W)<\delta$. We will use the following
non-strict version, obtained by applying that lemma with a smaller threshold.
\begin{lem}[Graphon \texorpdfstring{$K_r$}{Kr} removal]\label{lem:graphon-Kr-removal}
Fix $r\ge 3$.
For every $\eta>0$ there exists $\delta=\delta_{K_r}(\eta)>0$ such that whenever a graphon $W$ satisfies
\[
t(K_r,W)\le \delta,
\]
there exists a graphon $W'$ with
\[
t(K_r,W')=0,
\qquad
\|W-W'\|_1\le \eta.
\]
\end{lem}

\begin{proof}
Apply \cite[Lem.~2.3]{HladkyHuPiguet2021} with $F=K_r$ and error
$\eta$. This gives a number $\delta_0>0$ such that, whenever
\[
    t(K_r,W)<\delta_0,
\]
there is a graphon $W'$ satisfying
\[
    t(K_r,W')=0
    \qquad\text{and}\qquad
    \|W-W'\|_1\le \eta .
\]

Now define
\[
    \delta_{K_r}(\eta)\coloneqq \frac{\delta_0}{2}.
\]
If $t(K_r,W)\le \delta_{K_r}(\eta)$, then
\[
    t(K_r,W)\le \frac{\delta_0}{2}<\delta_0.
\]
Thus $W$ satisfies the strict hypothesis of
\cite[Lem.~2.3]{HladkyHuPiguet2021}, applied with $F=K_r$ and error
$\eta$. Hence there exists a graphon $W'$ such that
\[
t(K_r,W')=0
\qquad\text{and}\qquad
\|W-W'\|_1\le\eta.
\]
This is the required graphon.
\end{proof}

\begin{lem}[Entropy is continuous in \texorpdfstring{$L^1$}{L1}]\label{lem:entropy-L1-continuity}
For graphons $U$ and $V$, and for every $\alpha>0$,
\begin{equation}\label{eq:Ent-L1-continuity}
|\Ent(U)-\Ent(V)|
\le
\frac12\,\omega(\alpha)
+
\frac{\ln 2}{2}\,\frac{\|U-V\|_1}{\alpha},
\end{equation}
where
\[
\omega(\alpha)
\coloneqq
\sup\{|h(x)-h(y)|:\ x,y\in[0,1],\ |x-y|\le \alpha\}.
\]
In particular, $\Ent$ is uniformly continuous on $\mc W$ with respect to $\|\cdot\|_1$.
For every $\tau>0$ there exists $\eta>0$ such that, for all graphons $U$ and $V$,
\[
\|U-V\|_1\le\eta
\qquad\Longrightarrow\qquad
|\Ent(U)-\Ent(V)|\le\tau.
\]
\end{lem}

\begin{proof}
Fix $\alpha>0$ and define the two sets
\[
A\coloneqq \{(x,y)\in[0,1]^2:\ |U(x,y)-V(x,y)|\le \alpha\},
\]
\[
B\coloneqq \{(x,y)\in[0,1]^2:\ |U(x,y)-V(x,y)|> \alpha\}.
\]
Then
\begin{align*}
|\Ent(U)-\Ent(V)|
&=
\frac12\left|\int_{[0,1]^2}(h(U)-h(V))\right|\\
&\le
\frac12\int_A |h(U)-h(V)|
+
\frac12\int_B |h(U)-h(V)|.
\end{align*}
On the set $A$ we have
\[
|h(U)-h(V)|\le \omega(\alpha),
\]
so
\[
\int_A |h(U)-h(V)|\le \omega(\alpha).
\]
On the set $B$ we use only the bound $0\le h\le \ln 2$, and since $B \subseteq [0,1]^2$, we get
\[
\int_B |h(U)-h(V)|\le (\ln 2)\,\lambda^2(B).
\]
Also, since $|U-V|>\alpha$ on $B$,
\[
\|U-V\|_1\ge \alpha\,\lambda^2(B),
\qquad\text{so}\qquad
\lambda^2(B)\le \frac{\|U-V\|_1}{\alpha}.
\]
Putting these estimates together yields \eqref{eq:Ent-L1-continuity}.

Finally, $h$ is uniformly continuous on the compact interval $[0,1]$, and therefore
\[
\lim_{\alpha\to0}\omega(\alpha)
=
0.
\]
Fix $\tau>0$. Choose $\alpha>0$ such that
\[
\frac12\,\omega(\alpha)\le\frac\tau2,
\]
and then choose $\eta>0$ such that
\[
\frac{\ln 2}{2}\,\frac{\eta}{\alpha}\le\frac\tau2.
\]
If $\|U-V\|_1\le\eta$, then \eqref{eq:Ent-L1-continuity} gives
\[
|\Ent(U)-\Ent(V)|\le\tau.
\]
This proves the uniform continuity statement.
\end{proof}

\begin{lem}\label{lem:near-gap-Kr}
Fix $r\ge 3$ and set $q=r-1$.
For every $\varepsilon>0$ there exist $\delta_{r,\varepsilon}>0$ and $c'_{r,\varepsilon}>0$ such
that every graphon $W$ satisfying
\[
t(K_r,W)\le \delta_{r,\varepsilon}
\quad \text{and}\quad
\delta_\square(W,B_q^\star)\ge \varepsilon,
\]
also satisfies
\[
\Ent(W)\le \frac{\ln 2}{2}\Bigl(1-\frac1q\Bigr)-c'_{r,\varepsilon}.
\]
\end{lem}

\begin{proof}
Fix $\varepsilon>0$.
Apply Lemma~\ref{lem:gap-Kr} with $\varepsilon/2$.
This gives a constant $c_{r,\varepsilon/2}>0$ such that every $V$ that is a graphon of zero $K_r$-density and satisfies
\[
\delta_\square(V,B_q^\star)\ge \frac{\varepsilon}{2}
\]
satisfies
\[
\Ent(V)\le \frac{\ln 2}{2}\Bigl(1-\frac1q\Bigr)-c_{r,\varepsilon/2}.
\]

Choose $\eta>0$ so small that
\[
\eta\le \frac{\varepsilon}{2}
\]
and
\[
\sup_{\substack{U,V\in\mc W\\ \|U-V\|_1\le \eta}}
|\Ent(U)-\Ent(V)|
\le
\frac12\,c_{r,\varepsilon/2}.
\]
This is possible by Lemma~\ref{lem:entropy-L1-continuity}.

Now apply Lemma~\ref{lem:graphon-Kr-removal} with this $\eta$.
We obtain $\delta_{r,\varepsilon}>0$ such that every graphon $W$ with
\[
t(K_r,W)\le \delta_{r,\varepsilon}
\]
admits a graphon $W'$ of zero $K_r$-density satisfying
\[
\|W-W'\|_1\le \eta.
\]

Suppose that
\[
t(K_r,W)\le \delta_{r,\varepsilon}
\quad \text{and}\quad
\delta_\square(W,B_q^\star)\ge \varepsilon.
\]
Choose a corresponding graphon $W'$ such that $\|W-W'\|_1\le \eta$.
Since
\[
\delta_\square(W,W')\le \|W-W'\|_\square\le \|W-W'\|_1\le \eta,
\]
the triangle inequality gives
\[
\delta_\square(W',B_q^\star)
\ge
\delta_\square(W,B_q^\star)-\delta_\square(W,W')
\ge
\varepsilon-\eta
\ge
\frac{\varepsilon}{2}.
\]
So Lemma~\ref{lem:gap-Kr} implies
\[
\Ent(W')
\le
\frac{\ln 2}{2}\Bigl(1-\frac1q\Bigr)-c_{r,\varepsilon/2}.
\]
Also,
\[
|\Ent(W)-\Ent(W')|\le \frac12\,c_{r,\varepsilon/2}
\]
by our choice of $\eta$.
Therefore
\[
\Ent(W)
\le
\Ent(W')+\frac12\,c_{r,\varepsilon/2}
\le
\frac{\ln 2}{2}\Bigl(1-\frac1q\Bigr)-\frac12\,c_{r,\varepsilon/2}.
\]
Thus the lemma holds with
\[
c'_{r,\varepsilon}\coloneqq \frac12\,c_{r,\varepsilon/2}>0.
\]
\end{proof}

\section{Proof of Theorem~\ref{thm:main-clique}}\label{sec:free-energy}

\noindent
Recall that we have fixed $r\ge 3$ and $w > 0$ and we continue to let $q :=r-1$.
We first prove some technical results in sections~\ref{A general vanishing lemma}
--~\ref{Counting cut-closed families}.
Then we prove parts (i), (ii), and (iii) of Theorem~\ref{thm:main-clique}, in this order,
in sections~\ref{The reduced free-energy limit}
--~\ref{Transforming cut-distance closeness to a vertex partition}.

\begin{notation}[Balanced $q$-partite entropy level]\label{not:sq}
We write
\begin{equation}\label{eq:sq-definition}
s_q
\coloneqq
\frac{\ln 2}{2}\left(1-\frac1q\right).
\end{equation}
\end{notation}

Corollary~\ref{cor:ERGM} shows that the graph-independent factor $e^{wn^r}$
disappears after normalization. Thus the probability measure $\mbbP_{n,r}$ is determined by
the reduced masses
\[
\exp\bigl(-r!\,w\,K_r(G)\bigr),
\qquad G\in\mbG_n.
\]
By Lemma~\ref{lem:empirical-Kr-density},
\[
t(K_r,W_G)=\frac{r!\,K_r(G)}{n^r}.
\]
Consequently,
\begin{equation}\label{eq:reduced-mass-graphon}
\exp\bigl(-r!\,w\,K_r(G)\bigr)
=
\exp\bigl(-w\,t(K_r,W_G)n^r\bigr).
\end{equation}
This identity connects the finite-graph penalty with the graphon arguments in this section.

\subsection{A general vanishing lemma}\label{A general vanishing lemma}

The next lemma records the counting estimate that the
graphon $K_r$-homomorphism density
$t(K_r,W_{\mathsf G_n})=r!\,K_r(\mathsf G_n)/n^r$ goes to zero in probability for every fixed $r\ge 3$ and
every fixed $w>0$, where $\mathsf G_n\sim\mbbP_{n,r}$.

\begin{lem}\label{lem:Kr-density-vanishes}
Fix $r\ge 3$ and $w>0$.
For every $\varepsilon>0$ and every $n\in\mbbN$, define
\[
\mc A_{n,\varepsilon}^{(r)}
\coloneqq
\{\,G\in\mbG_n:\ K_r(G)\ge \varepsilon n^r\,\}.
\]
Then
\[
\mbbP_{n,r}\bigl(\mc A_{n,\varepsilon}^{(r)}\bigr)
\le
|\mbG_n|\,e^{-r!\,w\,\varepsilon n^r}
=
\exp\Bigl(\binom{n}{2}\ln 2-r!\,w\,\varepsilon n^r\Bigr).
\]
In particular, for every $\varepsilon>0$,
\[
\lim_{n\to\infty}
\mbbP_{n,r}\bigl(\mc A_{n,\varepsilon}^{(r)}\bigr)
=
0.
\]
Equivalently, if $\mathsf G_n\sim\mbbP_{n,r}$, then for every $\delta>0$,
\[
\lim_{n\to\infty}
\mbbP\left(
t(K_r,W_{\mathsf G_n})\ge\delta
\right)
=
0,
\]
because this event is
$\{\mathsf G_n\in\mc A_{n,\delta/r!}^{(r)}\}$. Thus
$t(K_r,W_{\mathsf G_n})$ converges to $0$ in probability.
\end{lem}

\begin{proof}
By \eqref{eq:showing-ergm},
\[
\mbbP_{n,r}(G)=\frac{e^{-r!\,w\,K_r(G)}}{\widetilde Z_{n,r}}.
\]
The denominator is at least $1$, because a graph with no edges has no copy of $K_r$ and contributes $1$.
So for every event $\mc A\subseteq \mbG_n$,
\[
\mbbP_{n,r}(\mc A)
=
\sum_{G\in\mc A}\mbbP_{n,r}(G)
\le
\sum_{G\in\mc A}e^{-r!\,w\,K_r(G)}.
\]
Apply this with $\mc A=\mc A_{n,\varepsilon}^{(r)}$.
On this event we have $K_r(G)\ge \varepsilon n^r$, hence
\[
\mbbP_{n,r}\bigl(\mc A_{n,\varepsilon}^{(r)}\bigr)
\le
|\mc A_{n,\varepsilon}^{(r)}|\,e^{-r!\,w\,\varepsilon n^r}
\le
|\mbG_n|\,e^{-r!\,w\,\varepsilon n^r}.
\]
Since $|\mbG_n|=2^{\binom n2}$, we have
\[
|\mbG_n|\,e^{-r!\,w\,\varepsilon n^r}
=
\exp\Bigl(\binom n2\ln 2-r!\,w\,\varepsilon n^r\Bigr).
\]
Because $r\ge 3$, the term $-r!\,w\,\varepsilon n^r$ is of strictly higher order than
$\binom n2\ln 2=O(n^2)$, so the exponent has limit $-\infty$.
\end{proof}

\subsection{Graphs with a fixed positive clique density}

A fixed positive lower bound on $t(K_r,W_G)$ gives a negative term of order $n^r$ in
the exponent. The total number of labeled graphs is only $\exp(O(n^2))$. Since
$r\ge3$, the penalty is stronger than this counting term.

\begin{lem}[Fixed positive clique density is negligible]
\label{lem:large-Kr-density-negligible}
Fix $0<\delta\le1$. For every $M>0$, there exists
$N=N(M,\delta,r,w)$ such that, for every $n\ge N$,
\begin{equation}\label{eq:large-Kr-density-negligible}
\sum_{\substack{G\in\mbG_n\\ t(K_r,W_G)\ge\delta}}
\exp\bigl(-r!\,w\,K_r(G)\bigr)
\le
\exp(-Mn^2).
\end{equation}
\end{lem}

\begin{proof}
If $t(K_r,W_G)\ge\delta$, then \eqref{eq:reduced-mass-graphon} gives
\[
\exp\bigl(-r!\,w\,K_r(G)\bigr)
\le
\exp(-w\delta n^r).
\]
Since $|\mbG_n|=2^{\binom n2}$, the total contribution of graphs satisfying
$t(K_r,W_G)\ge\delta$ is bounded by
\begin{equation}
\label{eq:large-Kr-density-total-bound}
\sum_{\substack{G\in\mbG_n\\ t(K_r,W_G)\ge\delta}}
\exp\bigl(-r!\,w\,K_r(G)\bigr)
\le
\exp(E_n),
\end{equation}
where
\begin{equation}
\label{eq:large-Kr-density-exponent}
E_n
\coloneqq
\binom n2\ln 2-w\delta n^r.
\end{equation}

We now compare $E_n$ with the $n^2$ scale. By
\eqref{eq:large-Kr-density-exponent},
\[
\frac{E_n}{n^2}
=
\frac{\binom n2}{n^2}\ln 2
-
w\delta n^{r-2}.
\]
Since
\[
\frac{\binom n2}{n^2}
=
\frac{n(n-1)}{2n^2}
=
\frac12-\frac{1}{2n},
\]
we get
\begin{equation}
\label{eq:large-Kr-density-normalized-exponent}
\frac{E_n}{n^2}
=
\frac{\ln 2}{2}
-
\frac{\ln 2}{2n}
-
w\delta n^{r-2}.
\end{equation}
Because $r\ge3$ and $w\delta>0$,
\[
\lim_{n\to\infty}
\left(
\frac{\ln 2}{2}
-
\frac{\ln 2}{2n}
-
w\delta n^{r-2}
\right)
=
-\infty.
\]
Therefore, by \eqref{eq:large-Kr-density-normalized-exponent}, there is
$N=N(M,\delta,r,w)$ such that, for every $n\ge N$,
\[
\frac{E_n}{n^2}\le -M.
\]
Equivalently,
\[
E_n\le -Mn^2.
\]
Combining this with \eqref{eq:large-Kr-density-total-bound}, we obtain, for
every $n\ge N$,
\[
\sum_{\substack{G\in\mbG_n\\ t(K_r,W_G)\ge\delta}}
\exp\bigl(-r!\,w\,K_r(G)\bigr)
\le
\exp(E_n)
\le
\exp(-Mn^2).
\]
This proves the lemma.
\end{proof}

Thus graphs with $t(K_r,W_G)\ge\delta$ have negligible total reduced mass on the
$n^2$ scale, for every fixed $\delta>0$. The main contribution must come from graphs
with small empirical $K_r$-density. Lemma~\ref{lem:boundary-max-Kr} identifies the
largest entropy at zero $K_r$-density, and Lemma~\ref{lem:near-gap-Kr} gives a strict
entropy loss away from $B_q^\star$ when the $K_r$-density is sufficiently small.

\subsection{Counting cut-closed families}\label{Counting cut-closed families}

We next convert entropy bounds for graphons into counting bounds for labeled graphs.

\begin{defin}[Cut-closed family of graphons]\label{def:cut-closed}
A family $\mc F\subseteq\mc W$ is called \emph{cut-closed} if, whenever
$W_m\in\mc F$ for every $m$ and
\[
\lim_{m\to\infty}\delta_\square(W_m,W)=0,
\]
then $W\in\mc F$.
\end{defin}

\begin{lem}\label{lem:cut-closed-saturation}
Let $\mc F\subseteq\mc W$ be cut-closed, and let $\widetilde{\mc F}$ be its image in
$\widetilde{\mc W}$. Then the following statements hold.
\begin{enumerate}
\item If $W\in\mc F$, $V\in\mc W$, and $\delta_\square(W,V)=0$, then $V\in\mc F$.
\item For every $n\in\mbbN$ and every $G\in\mbG_n$,
\[
W_G\in\mc F
\quad\Longleftrightarrow\quad
\widetilde W_G\in\widetilde{\mc F}.
\]
\item The set $\widetilde{\mc F}$ is closed in the metric $d_\square$.
\end{enumerate}
\end{lem}

\begin{proof}
For the first statement, take the constant sequence $W_m=W$. Then
\[
\lim_{m\to\infty}\delta_\square(W_m,V)=0,
\]
so cut-closedness gives $V\in\mc F$.

The forward implication in the second statement follows from the definition of
$\widetilde{\mc F}$. Conversely, suppose that
$\widetilde W_G\in\widetilde{\mc F}$. Then some $V\in\mc F$ satisfies
$\delta_\square(W_G,V)=0$. The first statement gives $W_G\in\mc F$.

For the third statement, let $\widetilde W_m\in\widetilde{\mc F}$ and suppose that
\[
\lim_{m\to\infty}d_\square(\widetilde W_m,\widetilde W)=0.
\]
Choose a representative $W$ of $\widetilde W$. For each $m$, choose a representative
$V_m\in\mc F$ of $\widetilde W_m$. By the definition of $d_\square$,
\[
\delta_\square(V_m,W)
=
d_\square(\widetilde W_m,\widetilde W).
\]
Hence
\[
\lim_{m\to\infty}\delta_\square(V_m,W)=0.
\]
Cut-closedness gives $W\in\mc F$, and therefore
$\widetilde W\in\widetilde{\mc F}$.
\end{proof}

\begin{notation}[Uniform measure on labeled graphs]
\label{not:uniform-graph-measure}
For each $n\in\mbbN$, let $\mbbP_n^{\mathrm{unif}}$ be the uniform
probability measure on $\mbG_n$. Since $\mbG_n$ contains
$2^{\binom n2}$ graphs, we have
\begin{equation}\label{eq:uniform-graph-measure}
\mbbP_n^{\mathrm{unif}}(A)
=
2^{-\binom n2}|A|,
\qquad
A\subseteq\mbG_n.
\end{equation}

We consider $\mbG_n$ as a probability space equipped with
$\mbbP_n^{\mathrm{unif}}$, and let $\mathsf R_n$ be the canonical random
graph on this space; thus
\[
\mathsf R_n(G)=G
\qquad\text{for every }G\in\mbG_n.
\]
Equivalently, each possible edge is present independently with probability
$\frac12$. Hence
\[
\mathsf R_n\sim G\left(n,\frac12\right).
\]

For every $\widetilde{\mc A}\subseteq\widetilde{\mc W}$, the event
$\widetilde W_{\mathsf R_n}\in\widetilde{\mc A}$ consists exactly of the
graphs whose empirical graphon class belongs to $\widetilde{\mc A}$. Hence
\begin{equation}\label{eq:uniform-graph-event}
\mbbP_n^{\mathrm{unif}}
\bigl(\widetilde W_{\mathsf R_n}\in\widetilde{\mc A}\bigr)
=
\mbbP_n^{\mathrm{unif}}
\bigl(
\{G\in\mbG_n:\widetilde W_G\in\widetilde{\mc A}\}
\bigr).
\end{equation}
\end{notation}

\begin{theor}
\label{thm:CV-closed-upper}
For a graphon $W\in\mc W$, define
\begin{equation}\label{eq:CV-rate-half}
I_{\frac12}(W)
\coloneqq
\frac12\int_{[0,1]^2}
\left[
W\ln(2W)+(1-W)\ln\bigl(2(1-W)\bigr)
\right]\,dx\,dy,
\end{equation}
where $0\ln0=0$.

The functional $I_{\frac12}$ is constant on zero-cut-distance classes.
Thus, if $\widetilde W=\pi(W)$, then
\[
I_{\frac12}(\widetilde W)
\coloneqq
I_{\frac12}(W)
\]
is well-defined. Moreover,
\begin{equation}\label{eq:CV-rate-entropy}
I_{\frac12}(\widetilde W)
=
\frac{\ln 2}{2}-\Ent(\widetilde W).
\end{equation}

If $\widetilde{\mc F}\subseteq\widetilde{\mc W}$ is closed, then
\begin{equation}\label{eq:CV-closed-upper-bound}
\limsup_{n\to\infty}
\frac1{n^2}
\ln
\mbbP_n^{\mathrm{unif}}
\left(
\left\{
G\in\mbG_n:
\widetilde W_G\in\widetilde{\mc F}
\right\}
\right)
\le
-\inf_{\widetilde W\in\widetilde{\mc F}}
I_{\frac12}(\widetilde W).
\end{equation}
Here $\inf\varnothing=+\infty$, $\sup\varnothing=-\infty$, and $\ln0=-\infty$.
\end{theor}

\begin{proof}
Recall that the binary entropy function is defined by
\[
h(u)= -u\ln u-(1-u)\ln(1-u),
\qquad u\in[0,1],
\]
with the convention $0\ln 0=0$. For $0<u<1$, we have
\[
\begin{aligned}
u\ln(2u)+(1-u)\ln\bigl(2(1-u)\bigr)
&=
u\bigl(\ln 2+\ln u\bigr)
+(1-u)\bigl(\ln 2+\ln(1-u)\bigr)\\
&=
\bigl(u+(1-u)\bigr)\ln 2
+u\ln u+(1-u)\ln(1-u)\\
&=
\ln 2-h(u).
\end{aligned}
\]
The same identity holds for $u=0$ and $u=1$ under the convention
$0\ln 0=0$. Hence it holds for every $u\in[0,1]$.
Integrating this identity gives
\[
I_{\frac12}(W)
=
\frac{\ln 2}{2}-\Ent(W).
\]
By Lemma~\ref{lem:entropy-zero-cut}, entropy is unchanged under zero cut
distance. Hence $I_{\frac12}$ is also unchanged under zero cut distance,
so $I_{\frac12}(\widetilde W)$ is well-defined and
\eqref{eq:CV-rate-entropy} follows.

Under the measure $\mbbP_n^{\mathrm{unif}}$ from
Notation~\ref{not:uniform-graph-measure}, the random graph
$\mathsf R_n$ has distribution $G(n,\frac12)$. We may therefore apply
\cite[Thm.~2.3]{ChatterjeeVaradhan2011} with $p=\frac12$.
The graphon associated with a finite graph in
\cite[Eq.~(4)]{ChatterjeeVaradhan2011} agrees almost everywhere with the
empirical graphon from Definition~\ref{def:empirical-graphon}. By
Remark~\ref{rem:zero-cut-distance}, the reduced graphon space used here is
the same reduced cut-metric space as the one used in
\cite{ChatterjeeVaradhan2011}. Finally, for $p=\frac12$, the rate
function in \cite[Eq.~(7)]{ChatterjeeVaradhan2011} is exactly
\eqref{eq:CV-rate-half}. The closed-set upper bound in
\cite[Thm.~2.3]{ChatterjeeVaradhan2011}, together with
\eqref{eq:uniform-graph-event}, therefore gives
\eqref{eq:CV-closed-upper-bound}.
\end{proof}

\begin{lem}
\label{lem:CV-counting-upper}
Let $\widetilde{\mc F}\subseteq\widetilde{\mc W}$ be closed.
For every $n\in\mbbN$,
\begin{equation}\label{eq:CV-closed-counting-identity}
\begin{aligned}
\mbbP_n^{\mathrm{unif}}
\left(
\left\{
G\in\mbG_n:
\widetilde W_G\in\widetilde{\mc F}
\right\}
\right)
&=
2^{-\binom n2}
\left|
\left\{
G\in\mbG_n:
\widetilde W_G\in\widetilde{\mc F}
\right\}
\right|.
\end{aligned}
\end{equation}

Consequently,
\begin{equation}\label{eq:CV-closed-counting-upper}
\limsup_{n\to\infty}
\frac1{n^2}
\ln
\left|
\left\{
G\in\mbG_n:
\widetilde W_G\in\widetilde{\mc F}
\right\}
\right|
\le
\sup_{\widetilde W\in\widetilde{\mc F}}
\Ent(\widetilde W).
\end{equation}

If $\widetilde{\mc F}\neq\varnothing$, then for every $\gamma>0$ there is
$N=N(\widetilde{\mc F},\gamma)$ such that, for every $n\ge N$,
\begin{equation}\label{eq:CV-finite-counting-upper}
\left|
\left\{
G\in\mbG_n:
\widetilde W_G\in\widetilde{\mc F}
\right\}
\right|
\le
\exp\left(
\left[
\sup_{\widetilde W\in\widetilde{\mc F}}
\Ent(\widetilde W)+\gamma
\right]n^2
\right).
\end{equation}
\end{lem}

\begin{proof}
Equation~\eqref{eq:CV-closed-counting-identity} follows immediately from
the definition \eqref{eq:uniform-graph-measure}.

Set
\[
N_n(\widetilde{\mc F})
\coloneqq
\left|
\left\{
G\in\mbG_n:
\widetilde W_G\in\widetilde{\mc F}
\right\}
\right|.
\]
Using \eqref{eq:CV-closed-counting-identity}, with $\ln0=-\infty$ when
$N_n(\widetilde{\mc F})=0$, we obtain
\[
\frac1{n^2}\ln N_n(\widetilde{\mc F})
=
\frac{\binom n2}{n^2}\ln 2
+
\frac1{n^2}
\ln
\mbbP_n^{\mathrm{unif}}
\left(
\left\{
G\in\mbG_n:
\widetilde W_G\in\widetilde{\mc F}
\right\}
\right).
\]
Taking the upper limit and applying
\eqref{eq:CV-closed-upper-bound} gives
\[
\limsup_{n\to\infty}
\frac1{n^2}\ln N_n(\widetilde{\mc F})
\le
\frac{\ln 2}{2}
-
\inf_{\widetilde W\in\widetilde{\mc F}}
I_{\frac12}(\widetilde W).
\]
By \eqref{eq:CV-rate-entropy}, the right-hand side equals
\[
\sup_{\widetilde W\in\widetilde{\mc F}}
\Ent(\widetilde W).
\]
This proves \eqref{eq:CV-closed-counting-upper}.

Finally, suppose that $\widetilde{\mc F}\neq\varnothing$ and fix
$\gamma>0$. By \eqref{eq:CV-closed-counting-upper}, for all sufficiently
large $n$,
\[
\frac1{n^2}\ln N_n(\widetilde{\mc F})
\le
\sup_{\widetilde W\in\widetilde{\mc F}}
\Ent(\widetilde W)+\gamma.
\]
Exponentiating proves \eqref{eq:CV-finite-counting-upper}.
\end{proof}

\begin{lem}\label{lem:entropy-counting}
Let $\mc F\subseteq\mc W$ be a nonempty cut-closed family. For every $\gamma>0$, there
exists $N=N(\mc F,\gamma)$ such that, for every $n\ge N$,
\begin{equation}\label{eq:entropy-counting-finite-upper}
\left|
\left\{
G\in\mbG_n:\ W_G\in\mc F
\right\}
\right|
\le
\exp\left(
\left[
\sup_{W\in\mc F}\Ent(W)+\gamma
\right]n^2
\right).
\end{equation}
Equivalently, there is a nonnegative sequence $\alpha_n(\mc F)$ satisfying
\[
\lim_{n\to\infty}\alpha_n(\mc F)=0
\]
such that
\begin{equation}\label{eq:entropy-counting-o-upper}
\left|
\left\{
G\in\mbG_n:\ W_G\in\mc F
\right\}
\right|
\le
\exp\left(
\left[
\sup_{W\in\mc F}\Ent(W)+\alpha_n(\mc F)
\right]n^2
\right)
\end{equation}
for every sufficiently large $n$.
\end{lem}

\begin{proof}
Let $\widetilde{\mc F}$ be the image of $\mc F$ in $\widetilde{\mc W}$.
Lemma~\ref{lem:cut-closed-saturation} shows that $\widetilde{\mc F}$ is closed and that
\[
W_G\in\mc F
\quad\Longleftrightarrow\quad
\widetilde W_G\in\widetilde{\mc F}.
\]
Therefore
\[
\left|
\left\{
G\in\mbG_n:\ W_G\in\mc F
\right\}
\right|
=
\left|
\left\{
G\in\mbG_n:\ \widetilde W_G\in\widetilde{\mc F}
\right\}
\right|.
\]
Entropy is constant on every zero-cut-distance class by
Lemma~\ref{lem:entropy-zero-cut}, so
\[
\sup_{\widetilde W\in\widetilde{\mc F}}\Ent(\widetilde W)
=
\sup_{W\in\mc F}\Ent(W).
\]
Now apply \eqref{eq:CV-finite-counting-upper} to $\widetilde{\mc F}$. This proves
\eqref{eq:entropy-counting-finite-upper}.

For the second form, set
\[
S_{\mc F}\coloneqq \sup_{W\in\mc F}\Ent(W)
\]
and
\[
N_n(\mc F)
\coloneqq
\left|
\left\{
G\in\mbG_n:\ W_G\in\mc F
\right\}
\right|.
\]
If $N_n(\mc F)=0$, set $b_n=0$. If $N_n(\mc F)>0$, set
\[
b_n
\coloneqq
\max\left\{
0,
\frac1{n^2}\ln N_n(\mc F)-S_{\mc F}
\right\}.
\]

We now show that $\lim_{n\to\infty}b_n=0$. Fix $\gamma>0$. By
\eqref{eq:entropy-counting-finite-upper}, for all sufficiently large $n$,
\[
N_n(\mc F)
\le
\exp\left((S_{\mc F}+\gamma)n^2\right).
\]
If $N_n(\mc F)=0$, then $b_n=0$. If $N_n(\mc F)>0$, the last inequality gives
\[
\frac1{n^2}\ln N_n(\mc F)-S_{\mc F}
\le
\gamma,
\]
and therefore
\[
0\le b_n\le\gamma.
\]
Since $\gamma>0$ was arbitrary, it follows that
\[
\lim_{n\to\infty}b_n=0.
\]

Set
\[
\alpha_n(\mc F)\coloneqq b_n+\frac1n.
\]
Then
\[
\lim_{n\to\infty}\alpha_n(\mc F)=0.
\]
If $N_n(\mc F)>0$, the definition of $b_n$ gives
\[
\frac1{n^2}\ln N_n(\mc F)
\le
S_{\mc F}+b_n
\le
S_{\mc F}+\alpha_n(\mc F).
\]
Exponentiating gives \eqref{eq:entropy-counting-o-upper}. If
$N_n(\mc F)=0$, then \eqref{eq:entropy-counting-o-upper} is trivial. Thus
\eqref{eq:entropy-counting-o-upper} holds for every sufficiently large $n$.
\end{proof}

\begin{lem}\label{lem:cut-closed-constraints}
Fix $r\ge3$ and set $q=r-1$.
The following subsets of $\mc W$ are cut-closed.
\begin{enumerate}
\item For every $\delta\in[0,1]$,
\[
\mc F_\delta^{(r)}
\coloneqq
\{W\in\mc W:\;t(K_r,W)\le\delta\}.
\]
\item For every $\varepsilon>0$,
\[
\mc D_{\varepsilon,q}
\coloneqq
\{W\in\mc W:\;\delta_\square(W,B_q^\star)\ge\varepsilon\}.
\]
\item For every $\delta\in[0,1]$ and $\varepsilon>0$,
\[
\mc F_{\delta,\varepsilon}^{(r,q)}
\coloneqq
\mc F_\delta^{(r)}\cap\mc D_{\varepsilon,q}.
\]
\end{enumerate}
\end{lem}

\begin{proof}

\begin{enumerate}
\item Fix $\delta\in[0,1]$. We show that
\[
\mc F_\delta^{(r)}
=
\{W\in\mc W:\;t(K_r,W)\le\delta\}
\]
is cut-closed.

Let $(W_m)_{m\ge1}$ be a sequence of graphons in $\mc F_\delta^{(r)}$, and suppose that
\[
\lim_{m\to\infty}\delta_\square(W_m,W)=0
\]
for some graphon $W\in\mc W$.

Since each $W_m$ belongs to $\mc F_\delta^{(r)}$, we have
\[
t(K_r,W_m)\le\delta
\qquad\text{for every }m.
\]
By Lemma~\ref{lem:Kr-cut-lipschitz}, the map
\[
W\mapsto t(K_r,W)
\]
is continuous in cut distance. Hence
\[
\lim_{m\to\infty}t(K_r,W_m)=t(K_r,W).
\]
Taking the limit in
\[
t(K_r,W_m)\le\delta
\]
gives
\[
t(K_r,W)\le\delta.
\]
Therefore
\[
W\in\mc F_\delta^{(r)}.
\]
So $\mc F_\delta^{(r)}$ is cut-closed.

\item Fix $\varepsilon>0$. We prove that
\[
\mc D_{\varepsilon,q}
=
\{W\in\mc W:\;\delta_\square(W,B_q^\star)\ge\varepsilon\}
\]
is cut-closed.

If $\mc H$ is a nonempty family of graphons, then the function
\[
W\mapsto \delta_\square(W,\mc H)
\]
is continuous in cut distance. Indeed, let $U,V\in\mc W$. For every $H\in\mc H$, the triangle inequality gives
\[
\delta_\square(U,H)
\le
\delta_\square(U,V)+\delta_\square(V,H).
\]
Taking the infimum over all $H\in\mc H$ gives
\[
\delta_\square(U,\mc H)
\le
\delta_\square(U,V)+\delta_\square(V,\mc H).
\]
So
\[
\delta_\square(U,\mc H)-\delta_\square(V,\mc H)
\le
\delta_\square(U,V).
\]
Now if we switch the roles of $U$ and $V$, we get
\[
\delta_\square(V,\mc H)-\delta_\square(U,\mc H)
\le
\delta_\square(U,V).
\]
Therefore
\[
\left|
\delta_\square(U,\mc H)-\delta_\square(V,\mc H)
\right|
\le
\delta_\square(U,V).
\]
Thus $W\mapsto\delta_\square(W,\mc H)$ is continuous.

Now apply this with
\[
\mc H=B_q^\star.
\]
This family is nonempty, because $W_q^\star\in B_q^\star$ by
Notation~\ref{not:Wqstar}. Hence
\[
W\mapsto \delta_\square(W,B_q^\star)
\]
is continuous in cut distance.

Now let $(W_m)_{m\ge1}$ be a sequence of graphons in $\mc D_{\varepsilon,q}$, and suppose that
\[
\lim_{m\to\infty}\delta_\square(W_m,W)=0
\]
for some graphon $W\in\mc W$.

Since each $W_m$ belongs to $\mc D_{\varepsilon,q}$, we have
\[
\delta_\square(W_m,B_q^\star)\ge\varepsilon
\qquad\text{for every }m.
\]
By the continuity established above, we have
\[
\lim_{m\to\infty}\delta_\square(W_m,B_q^\star)
=
\delta_\square(W,B_q^\star).
\]
Taking the limit in
\[
\delta_\square(W_m,B_q^\star)\ge\varepsilon
\]
gives
\[
\delta_\square(W,B_q^\star)\ge\varepsilon.
\]
Therefore
\[
W\in\mc D_{\varepsilon,q}.
\]
So $\mc D_{\varepsilon,q}$ is cut-closed.

\item Fix $\delta\in[0,1]$ and $\varepsilon>0$. We show that
\[
\mc F_{\delta,\varepsilon}^{(r,q)}
=
\mc F_\delta^{(r)}\cap\mc D_{\varepsilon,q}
\]
is cut-closed.

Let $(W_m)_{m\ge1}$ be a sequence of graphons in
$\mc F_{\delta,\varepsilon}^{(r,q)}$, and suppose that
\[
\lim_{m\to\infty}\delta_\square(W_m,W)=0
\]
for some graphon $W\in\mc W$.

Since
\[
W_m\in\mc F_{\delta,\varepsilon}^{(r,q)}
=
\mc F_\delta^{(r)}\cap\mc D_{\varepsilon,q},
\]
we have
\[
W_m\in\mc F_\delta^{(r)}
\qquad\text{and}\qquad
W_m\in\mc D_{\varepsilon,q}
\]
for every $m$.

By item \textup{(1)}, $\mc F_\delta^{(r)}$ is cut-closed. Hence
\[
W\in\mc F_\delta^{(r)}.
\]
By item \textup{(2)}, $\mc D_{\varepsilon,q}$ is cut-closed. Hence
\[
W\in\mc D_{\varepsilon,q}.
\]
Therefore
\[
W\in
\mc F_\delta^{(r)}\cap\mc D_{\varepsilon,q}
=
\mc F_{\delta,\varepsilon}^{(r,q)}.
\]
So $\mc F_{\delta,\varepsilon}^{(r,q)}$ is cut-closed.
\end{enumerate}
\end{proof}

\subsection{The reduced free-energy limit}\label{The reduced free-energy limit}

\begin{lem}[Balanced multipartite lower bound]\label{lem:balanced-q-partite-lower}
Let $n\in\mbbN$, and use
\[
n=qm+a,
\qquad
m,a\in\mbbZ_{\ge0},
\qquad
0\le a<q.
\]
Then, with $s_q$ as in Notation~\ref{not:sq},
\begin{equation}\label{eq:balanced-q-partite-lower}
\widetilde Z_{n,r}
\ge
\exp\left(
s_q n^2
-
\frac{\ln 2}{2}a\left(1-\frac{a}{q}\right)
\right)
\ge
\exp\left(
s_q n^2-\frac{q\ln 2}{8}
\right).
\end{equation}
\end{lem}

\begin{proof}
Write $[n]$ as the disjoint union of $q$ possibly empty sets
\[
[n]=V_1\sqcup\cdots\sqcup V_q
\]
whose sizes differ by at most one. Thus $a$ sets have size $m+1$, and the
remaining $q-a$ sets have size $m$. Put
\[
n_i\coloneqq |V_i|,
\qquad i=1,\dots,q.
\]

Let $M$ be the number of unordered pairs of vertices lying in different parts. Then
\[
M
=
\sum_{1\le i<j\le q} n_i n_j.
\]
Since
\[
\left(\sum_{i=1}^q n_i\right)^2
=
\sum_{i=1}^q n_i^2
+
2\sum_{1\le i<j\le q}n_i n_j,
\]
we have
\begin{equation}
\label{eq:M-from-part-square-sum}
M
=
\frac12\left(n^2-\sum_{i=1}^q n_i^2\right).
\end{equation}

We now compute the sum of the squared part sizes. By the choice of the partition,
\begin{align*}
\sum_{i=1}^q n_i^2
&=
(q-a)m^2+a(m+1)^2  \\
&=
qm^2+2a m+a.
\end{align*}
On the other hand, since $n=qm+a$,
\[
\frac{n^2}{q}
=
qm^2+2a m+\frac{a^2}{q}.
\]
Therefore
\begin{equation}
\label{eq:balanced-part-square-sum}
\sum_{i=1}^q n_i^2
=
\frac{n^2}{q}
+
a\left(1-\frac{a}{q}\right).
\end{equation}
Substituting \eqref{eq:balanced-part-square-sum} into
\eqref{eq:M-from-part-square-sum} gives
\begin{equation}
\label{eq:balanced-M-value}
M
=
\frac12\left(1-\frac1q\right)n^2
-
\frac12a\left(1-\frac{a}{q}\right).
\end{equation}

Now choose an arbitrary subset of these $M$ possible between-part edges, and put no
edges inside the parts. This gives exactly $2^M$ labeled graphs on $[n]$.

Every graph constructed in this way is $q$-partite. Hence every clique in such a graph
uses at most one vertex from each part, so every clique has size at most $q$. Since
$q=r-1$, none of these graphs contains a copy of $K_r$. Therefore each of these
graphs has reduced mass
\[
\exp(-r!\,w\,K_r(G))=1.
\]
Thus all these graphs contribute $1$ each to $\widetilde Z_{n,r}$, and so
\[
\widetilde Z_{n,r}
\ge
2^M.
\]
Multiplying \eqref{eq:balanced-M-value} by $\ln 2$ and then using
\eqref{eq:sq-definition}, we obtain
\[
\begin{aligned}
(\ln 2)M
&=
\frac{\ln 2}{2}\left(1-\frac1q\right)n^2
-
\frac{\ln 2}{2}a\left(1-\frac{a}{q}\right) \\
&=
s_q n^2
-
\frac{\ln 2}{2}a\left(1-\frac{a}{q}\right).
\end{aligned}
\]
Consequently,
\[
2^M
=
\exp\left(
s_q n^2
-
\frac{\ln 2}{2}a\left(1-\frac{a}{q}\right)
\right).
\]
Together with $\widetilde Z_{n,r}\ge 2^M$, this proves the first inequality
in \eqref{eq:balanced-q-partite-lower}.

Finally, the function
\[
x\left(1-\frac{x}{q}\right)
\]
is maximized on $[0,q]$ at $x=q/2$, where its value is $q/4$. Since
$0\le a<q$, we have
\[
0\le
a\left(1-\frac{a}{q}\right)
\le
\frac q4.
\]
Hence
\[
-\frac{\ln 2}{2}a\left(1-\frac{a}{q}\right)
\ge
-\frac{q\ln 2}{8}.
\]
This gives the second inequality in \eqref{eq:balanced-q-partite-lower}.
\end{proof}

We now prove Theorem~\ref{thm:main-clique}\textup{(i)}. We first fix an
auxiliary clique coefficient $L$, then let $n\to\infty$, and only afterward
let $L\to\infty$.

\begin{proof}[Proof of Theorem~\ref{thm:main-clique}\textup{(i)}]
We prove the upper and lower bounds separately.

\smallskip
\noindent\emph{Upper bound.}
Fix $L>0$ and define
\[
\widetilde Z_{n,r}^{(L)}
\coloneqq
\sum_{G\in\mbG_n}
\exp\left(-Ln^2t(K_r,W_G)\right).
\]
This is an auxiliary partition function in which the clique penalty $L$
remains fixed as $n$ grows.

For every graphon $W$,
\[
0\le t(K_r,W)\le 1.
\]
Moreover, $t(K_r,W)$ is continuous in cut distance by
Lemma~\ref{lem:Kr-cut-lipschitz}. Hence the functional
\[
W\longmapsto -L\,t(K_r,W)
\]
is bounded and continuous. We may therefore apply
\cite[Thm.~3.1]{ChatterjeeDiaconis2013}, which gives
\begin{equation}\label{eq:fixed-clique-variational-limit}
\lim_{n\to\infty}
\frac1{n^2}\ln\widetilde Z_{n,r}^{(L)}
=
\sup_{W\in\mc W}
\left\{
\Ent(W)-L\,t(K_r,W)
\right\}.
\end{equation}

We next determine what happens to the right-hand side when $L$ becomes
large. In the notation of \cite[Thm.~7.1]{ChatterjeeDiaconis2013}, take
\[
H=K_r,
\qquad
\beta_1=0,
\qquad
\beta_2=-L.
\]
The rate functional $I$ used there satisfies
\[
I(W)
=
\frac12\int_{[0,1]^2}
\bigl[W\ln W+(1-W)\ln(1-W)\bigr]\,dx\,dy
=
-\Ent(W),
\]
with the usual convention at $0$ and $1$. Thus, at $\beta_2=-L$, their
variational value is exactly
\[
\sup_{W\in\mc W}
\left\{
\Ent(W)-L\,t(K_r,W)
\right\}.
\]
Moreover,
\[
\chi(K_r)=r
\qquad\text{and}\qquad
p=\frac{e^{2\beta_1}}{1+e^{2\beta_1}}=\frac12.
\]
Theorem~7.1 of that paper therefore gives the limiting value
\[
\frac{r-2}{2(r-1)}\ln 2
=
\frac{\ln 2}{2}\left(1-\frac1q\right)
=
s_q,
\]
where we used $q=r-1$. It follows that
\begin{equation}\label{eq:fixed-clique-large-negative-limit}
\lim_{L\to\infty}
\sup_{W\in\mc W}
\left\{
\Ent(W)-L\,t(K_r,W)
\right\}
=
s_q.
\end{equation}

We now compare the auxiliary model with the model in the theorem. Write
\[
L_n\coloneqq wn^{r-2}.
\]
By \eqref{eq:reduced-mass-graphon}, the reduced mass of a graph $G$ is
\[
\exp\bigl(-r!\,w\,K_r(G)\bigr)
=
\exp\left(-n^2L_n\,t(K_r,W_G)\right).
\]
Since $r\ge3$ and $w>0$, we have
\[
\lim_{n\to\infty}L_n=+\infty.
\]
Thus, for every fixed $L>0$, we have $L_n\ge L$ for all sufficiently
large $n$. Since $t(K_r,W_G)\ge0$, this gives
\[
\exp\left(-n^2L_n\,t(K_r,W_G)\right)
\le
\exp\left(-n^2L\,t(K_r,W_G)\right).
\]
Summing this inequality over all $G\in\mbG_n$ yields
\[
\widetilde Z_{n,r}
\le
\widetilde Z_{n,r}^{(L)}
\]
for all sufficiently large $n$. Taking logarithms, dividing by $n^2$,
and using \eqref{eq:fixed-clique-variational-limit}, we obtain
\[
\begin{aligned}
\limsup_{n\to\infty}F_{n,r}(w)
&\le
\lim_{n\to\infty}
\frac1{n^2}\ln\widetilde Z_{n,r}^{(L)}\\
&=
\sup_{W\in\mc W}
\left\{
\Ent(W)-L\,t(K_r,W)
\right\}.
\end{aligned}
\]
This holds for every fixed $L>0$. We may therefore let $L\to\infty$.
Equation~\eqref{eq:fixed-clique-large-negative-limit} then gives
\[
\limsup_{n\to\infty}F_{n,r}(w)\le s_q.
\]

\smallskip
\noindent\emph{Lower bound.}
A graph whose edges lie only between the parts of a $q$-partite
partition contains no copy of $K_r$, because $q=r-1$. Such a graph
therefore has reduced mass $1$. Lemma~\ref{lem:balanced-q-partite-lower}
counts these graphs and gives
\[
F_{n,r}(w)
\ge
s_q-\frac{q\ln 2}{8n^2}.
\]
Taking the lower limit gives
\[
\liminf_{n\to\infty}F_{n,r}(w)\ge s_q.
\]

The upper and lower bounds now agree. Therefore
\[
\lim_{n\to\infty}F_{n,r}(w)
=
s_q
=
\frac{\ln 2}{2}\left(1-\frac1q\right).
\]
\end{proof}

\subsection{Exponential concentration in cut distance}

We now prove Theorem~\ref{thm:main-clique}\textup{(ii)}. Fix
$\varepsilon>0$. The event we need to estimate is
\[
\delta_\square(W_{\mathsf G_n},B_q^\star)\ge\varepsilon.
\]
After choosing a small number $\delta>0$, we separate this event according to
the empirical clique density $t(K_r,W_{\mathsf G_n})$. The part with
\[
t(K_r,W_{\mathsf G_n})\ge\delta
\]
has very small total reduced mass by
Lemma~\ref{lem:large-Kr-density-negligible}. The remaining part has
\[
t(K_r,W_{\mathsf G_n})<\delta
\qquad\text{and}\qquad
\delta_\square(W_{\mathsf G_n},B_q^\star)\ge\varepsilon.
\]
For this part, Lemma~\ref{lem:near-gap-Kr} gives an entropy loss, and the
counting bound in Lemma~\ref{lem:entropy-counting} turns that entropy loss
into an exponential probability bound.

As explained in Remark~\ref{rem:relation-to-fixed-parameter-ergms},
\cite[Thm.~3.2]{ChatterjeeDiaconis2013} gives concentration for each fixed
clique coefficient $L$, but it does not provide concentration constants
uniform in $L$. Since the present model follows the diagonal sequence
$L_n=w n^{r-2}$, for which $\lim_{n\to\infty}L_n=+\infty$, the proof below
establishes the required uniform diagonal estimate directly.

\begin{proof}[Proof of Theorem~\ref{thm:main-clique}\textup{(ii)}]
Fix $\varepsilon>0$, and let $\mathsf G_n\sim\mbbP_{n,r}$. Choose
\[
\delta_0=\delta_{r,\varepsilon}>0
\qquad\text{and}\qquad
c'=c'_{r,\varepsilon}>0
\]
from Lemma~\ref{lem:near-gap-Kr}, and set
\[
\delta\coloneqq \min\{\delta_0,1\}.
\]
The conclusion of Lemma~\ref{lem:near-gap-Kr} still holds under the stronger condition
$t(K_r,W)\le\delta$.

Define
\[
\mc A_{n,\varepsilon}
\coloneqq
\left\{
G\in\mbG_n:\ \delta_\square(W_G,B_q^\star)\ge\varepsilon
\right\}.
\]
By \eqref{eq:showing-ergm},
\[
\mbbP(\mathsf G_n\in\mc A_{n,\varepsilon})
=
\frac{R_{n,\varepsilon}}{\widetilde Z_{n,r}},
\]
where
\[
R_{n,\varepsilon}
\coloneqq
\sum_{G\in\mc A_{n,\varepsilon}}
\exp\bigl(-r!\,w\,K_r(G)\bigr).
\]
Using \eqref{eq:reduced-mass-graphon}, set
\[
R_{n,\varepsilon}
=
R_{n,\varepsilon}^{\mathrm{large}}
+
R_{n,\varepsilon}^{\mathrm{small}},
\]
where
\[
R_{n,\varepsilon}^{\mathrm{large}}
\coloneqq
\sum_{\substack{G\in\mc A_{n,\varepsilon}\\ t(K_r,W_G)\ge\delta}}
\exp\bigl(-w\,t(K_r,W_G)n^r\bigr)
\]
and
\[
R_{n,\varepsilon}^{\mathrm{small}}
\coloneqq
\sum_{\substack{G\in\mc A_{n,\varepsilon}\\ t(K_r,W_G)<\delta}}
\exp\bigl(-w\,t(K_r,W_G)n^r\bigr).
\]

For every $M>0$, Lemma~\ref{lem:large-Kr-density-negligible} gives, for all
sufficiently large $n$,
\begin{equation}\label{eq:R-large-bound}
R_{n,\varepsilon}^{\mathrm{large}}
\le
\exp(-Mn^2).
\end{equation}

Every summand in $R_{n,\varepsilon}^{\mathrm{small}}$ is at most $1$, and every graph in
this sum satisfies
\[
W_G\in\mc F_{\delta,\varepsilon}^{(r,q)}.
\]
Hence
\[
R_{n,\varepsilon}^{\mathrm{small}}
\le
\left|
\left\{
G\in\mbG_n:\ W_G\in\mc F_{\delta,\varepsilon}^{(r,q)}
\right\}
\right|.
\]
If $\mc F_{\delta,\varepsilon}^{(r,q)}$ is empty, this upper bound is $0$. If it is
nonempty, then it is cut-closed by Lemma~\ref{lem:cut-closed-constraints}, and
Lemma~\ref{lem:near-gap-Kr} gives
\[
\sup_{W\in\mc F_{\delta,\varepsilon}^{(r,q)}}\Ent(W)
\le
s_q-c'.
\]
Lemma~\ref{lem:entropy-counting}, with counting error $c'/4$, now gives, for every
sufficiently large $n$,
\begin{equation}\label{eq:R-small-bound}
R_{n,\varepsilon}^{\mathrm{small}}
\le
\exp\left(\left(s_q-\frac{3c'}4\right)n^2\right).
\end{equation}

Set
\[
C_q\coloneqq \frac{q\ln 2}{8}.
\]
Lemma~\ref{lem:balanced-q-partite-lower} gives
\[
\widetilde Z_{n,r}
\ge
\exp(s_qn^2-C_q).
\]
Choose $M=1+c'$. Combining this estimate with
\eqref{eq:R-large-bound} and \eqref{eq:R-small-bound} gives
\begin{align*}
\mbbP(\mathsf G_n\in\mc A_{n,\varepsilon})
&\le
\exp\bigl(-(M+s_q)n^2+C_q\bigr)
+
\exp\left(-\frac{3c'}4n^2+C_q\right).
\end{align*}
For every sufficiently large $n$, both terms on the right are at most
$\exp(-c'n^2/2)$. Increasing the threshold for $n$ once more gives
\[
2\exp(-c'n^2/2)
\le
\exp(-c'n^2/4).
\]
Therefore, for every sufficiently large $n$,
\[
\mbbP\bigl(
\delta_\square(W_{\mathsf G_n},B_q^\star)\ge\varepsilon
\bigr)
\le
\exp(-c'n^2/4).
\]
The theorem follows with $c=c'/4$. The constant $c$ depends only on $\varepsilon$ and
$r$. The threshold for $n$ may also depend on $w$, through
Lemma~\ref{lem:large-Kr-density-negligible}.
\end{proof}

\subsection{Probability bounds needed for the rounding argument}

\begin{lem}\label{lem:prob-tools-balance}
The following probability bounds hold:
\begin{enumerate}
\item[\textup{(i)}] \textbf{Markov's inequality.}
If $Y\ge 0$ almost surely and $t>0$, then
\[
\mbbP(Y\ge t)
\le
\frac{\mbbE[Y]}{t}.
\]

\item[\textup{(ii)}] \textbf{Hoeffding's inequality.}
Let $m\ge 1$, and let $Y_1,\dots,Y_m$ be independent random variables. Suppose that,
for each $i\in[m]$, there are real numbers $a_i\le b_i$ such that
\[
a_i\le Y_i\le b_i
\]
almost surely. Set
\[
V\coloneqq \sum_{i=1}^m (b_i-a_i)^2.
\]
If $V>0$, then, for every $t\ge 0$,
\[
\mbbP\left(
\left|
\sum_{i=1}^m Y_i
-
\sum_{i=1}^m \mbbE[Y_i]
\right|
\ge t
\right)
\le
2\exp\left(-\frac{2t^2}{V}\right).
\]
If $V=0$, then
\[
\sum_{i=1}^m Y_i
=
\sum_{i=1}^m \mbbE[Y_i]
\]
almost surely. Consequently, for every $t>0$,
\[
\mbbP\left(
\left|
\sum_{i=1}^m Y_i
-
\sum_{i=1}^m \mbbE[Y_i]
\right|
\ge t
\right)
=
0.
\]

In particular, if $0\le Y_i\le 1$ almost surely for every $i\in[m]$, and if
\[
S\coloneqq \sum_{i=1}^m Y_i,
\]
then, for every $t\ge0$,
\[
\mbbP\bigl(|S-\mbbE[S]|\ge t\bigr)
\le
2\exp\left(-\frac{2t^2}{m}\right).
\]

\item[\textup{(iii)}] \textbf{McDiarmid's inequality.}
Let $m\ge1$, let $S_1,\dots,S_m$ be finite sets, and let
$X_1,\dots,X_m$ be independent random variables such that
\[
X_i\in S_i
\]
almost surely for every $i\in[m]$. Let
\[
f:S_1\times\cdots\times S_m\to\mbbR
\]
be a function. Assume that there is a number $L>0$ such that changing only
one coordinate can change the value of $f$ by at most $L$. More precisely,
\[
\left|
f(x_1,\dots,x_i,\dots,x_m)
-
f(x_1,\dots,x_i',\dots,x_m)
\right|
\le L
\]
for every $i\in[m]$, every $x_i,x_i'\in S_i$, and every fixed choice of the
other coordinates. Then, for every $t\ge0$,
\[
\mbbP\left(
\left|
f(X_1,\dots,X_m)
-
\mbbE[f(X_1,\dots,X_m)]
\right|
\ge t
\right)
\le
2\exp\left(-\frac{2t^2}{mL^2}\right).
\]
\end{enumerate}
\end{lem}

\begin{proof}
Item~\textup{(i)} is Markov's inequality; see
\cite[Thm.~3.1]{MitzenmacherUpfal2017}, with the notation there renamed as
$X=Y$ and $a=t$.

For item~\textup{(ii)}, first assume that $V>0$.

If $t>0$, Hoeffding's inequality
\cite[Thm.~4.14]{MitzenmacherUpfal2017} gives
\[
\mbbP\left(
\left|
\sum_{i=1}^m Y_i
-
\sum_{i=1}^m \mbbE[Y_i]
\right|
\ge t
\right)
\le
2\exp\left(-\frac{2t^2}{V}\right).
\]
If $t=0$, the same bound is also true, because the probability on the left is
at most $1$, while the right-hand side is $2$.

Now assume that
\[
V=0.
\]
Then
\[
\sum_{i=1}^m (b_i-a_i)^2=0.
\]
Each term in this sum is nonnegative, so
\[
b_i-a_i=0
\qquad
\text{for every }i\in[m].
\]
Thus $a_i=b_i$ for every $i$. Since
\[
a_i\le Y_i\le b_i
\]
almost surely, each $Y_i$ is almost surely constant. Hence
\[
Y_i=\mbbE[Y_i]
\qquad
\text{almost surely for every }i\in[m].
\]
Adding over $i=1,\dots,m$ gives
\[
\sum_{i=1}^m Y_i
=
\sum_{i=1}^m \mbbE[Y_i]
\]
almost surely. Therefore, for every $t>0$,
\[
\mbbP\left(
\left|
\sum_{i=1}^m Y_i
-
\sum_{i=1}^m \mbbE[Y_i]
\right|
\ge t
\right)
=
0.
\]

For the last part of item~\textup{(ii)}, assume that
\[
0\le Y_i\le1
\qquad
\text{almost surely for every }i\in[m],
\]
and set
\[
S\coloneqq \sum_{i=1}^m Y_i.
\]
Taking
\[
a_i=0
\qquad
\text{and}
\qquad
b_i=1
\]
for every $i$ gives
\[
V
=
\sum_{i=1}^m (b_i-a_i)^2
=
\sum_{i=1}^m 1
=
m.
\]
The bound already proved therefore gives, for every $t\ge0$,
\[
\mbbP\left(
|S-\mbbE[S]|\ge t
\right)
\le
2\exp\left(-\frac{2t^2}{m}\right).
\]

For item~\textup{(iii)}, the bounded-difference condition is exactly the
hypothesis of McDiarmid's inequality in
\cite[Sec.~13.5.1, Thm.~13.7]{MitzenmacherUpfal2017}. Applying that theorem with
$n=m$, $c=L$, and $\lambda=t$ gives the stated bound for $t>0$. The
theorem applies here because the random variables take values in finite sets.
For $t=0$, the bound is immediate, since the right-hand side is $2$.
\end{proof}

\subsection{Transforming cut-distance closeness to a vertex partition}
\label{Transforming cut-distance closeness to a vertex partition}

The next lemma turns cut-distance closeness to $B_q^\star$ into a partition of the finite
vertex set. It controls the part sizes, the edge density between every pair of parts, and the total
number of edges inside the parts.

\begin{lem}\label{lem:det-round-Bqstar}
Let $q\ge2$ and $0<\varepsilon\le1$, and set
\[
\rho\coloneqq \frac{\varepsilon}{100q^2}.
\]
There exists $N_{\mathrm{det}}=N_{\mathrm{det}}(\varepsilon,q)$ such that the following
holds for every $n\ge N_{\mathrm{det}}$. If $G$ is a graph on $[n]$ and
\[
\delta_\square(W_G,B_q^\star)<\rho,
\]
then there is a partition
\[
[n]=V_1\sqcup\cdots\sqcup V_q
\]
with
\[
\bigl||V_\alpha|-n/q\bigr|\le\varepsilon n
\qquad\text{for every }\alpha\in[q],
\]
each $V_\alpha$ has size at least $2$,
\[
\left|
\frac{e_G(V_\alpha,V_\beta)}{|V_\alpha||V_\beta|}-\frac12
\right|
\le\varepsilon
\qquad\text{for all distinct }\alpha,\beta\in[q],
\]
\[
\sum_{\alpha=1}^q e_G(V_\alpha)
\le
\frac{\varepsilon n^2}{10q},
\]
and
\[
\frac{\sum_{\alpha=1}^q e_G(V_\alpha)}
{\sum_{\alpha=1}^q\binom{|V_\alpha|}{2}}
\le\varepsilon.
\]
Moreover, if $G'$ is obtained from $G$ by deleting all edges whose endpoints lie in the
same part, then $G'$ is $q$-partite and
\[
d_{\mathrm{edit}}(G,G')
\le
\frac{\varepsilon}{10q}.
\]
\end{lem}

\begin{proof}
We use a random coloring argument and color the vertices independently using probabilities obtained from the graphon relabeling. We then show that the desired properties hold with positive probability and fix one coloring for which they hold.

Let $(a_{ij})_{i,j=1}^n$ be the adjacency matrix of $G$. Thus
\[
a_{ij}=a_{ji},
\qquad
a_{ii}=0.
\]

Let $W_q^\star$ and its canonical blocks $A_1,\dots,A_q$ be as in
Notation~\ref{not:Wqstar}. By that notation,
\[
\delta_\square(W_G,B_q^\star)
=
\delta_\square(W_G,W_q^\star)
<
\rho.
\]
Since the infimum in the definition of cut distance is strictly smaller
than $\rho$, there is a measure-preserving bijection
$\psi:[0,1]\to[0,1]$ such that
\[
\|W_G-(W_q^\star)^\psi\|_\square<\rho.
\]
Set
\[
\phi\coloneqq \psi^{-1}.
\]
Since the cut norm is invariant under measure-preserving relabelings, we get
\begin{equation}
\label{eq:rounding-cut-norm-closeness}
\begin{aligned}
\|W_G^\phi-W_q^\star\|_\square
&=
\|(W_G-(W_q^\star)^\psi)^\phi\|_\square \\
&=
\|W_G-(W_q^\star)^\psi\|_\square \\
&<
\rho.
\end{aligned}
\end{equation}

Let $I_1,\dots,I_n$ be the intervals from
Definition~\ref{def:empirical-graphon}. These intervals form a measurable
partition of $[0,1]$.

Define
\[
J_i\coloneqq \phi^{-1}(I_i),
\qquad i\in[n].
\]
Then $J_1,\dots,J_n$ form a measurable partition of $[0,1]$, and
\[
\lambda(J_i)=\frac1n
\qquad\text{for every }i\in[n].
\]

For $i\in[n]$ and $\alpha\in[q]$, set
\[
p_{i,\alpha}
\coloneqq
n\lambda(J_i\cap A_\alpha).
\]
For each fixed $i$, these numbers can be used as probabilities:
\[
0\le p_{i,\alpha}\le1
\qquad\text{and}\qquad
\sum_{\alpha=1}^q p_{i,\alpha}=1.
\]
Also, for each fixed $\alpha$,
\[
\sum_{i=1}^n p_{i,\alpha}
=
n\lambda(A_\alpha)
=
\frac nq.
\]

We now define an auxiliary random coloring of the vertices. Let
\[
\Sigma\coloneqq(\xi_1,\dots,\xi_n)
\]
be an $[q]^n$-valued random vector whose coordinates are independent and
satisfy
\[
\mbbP_{\mathrm{col}}(\xi_i=\alpha)=p_{i,\alpha},
\qquad i\in[n],\ \alpha\in[q].
\]
Equivalently, for every $C=(C_1,\dots,C_n)\in[q]^n$,
\[
\mbbP_{\mathrm{col}}(\Sigma=C)
=
\prod_{i=1}^n p_{i,C_i}.
\]
Here $\mbbP_{\mathrm{col}}$ and $\mbbE_{\mathrm{col}}$ denote probability
and expectation with respect to this product probability measure on $[q]^n$.

For a coloring $C=(C_1,\dots,C_n)\in[q]^n$, define its color classes by
\[
V_\alpha(C)
\coloneqq
\{i\in[n]:C_i=\alpha\},
\qquad \alpha\in[q].
\]

\smallskip
\noindent
\emph{The parts are almost balanced.}
For every $\alpha\in[q]$ and $i\in[n]$, set
\[
Z_i^{(\alpha)}
\coloneqq
\mathbf{1}_{\{\xi_i=\alpha\}}.
\]
For each fixed $\alpha$, the random variables
$Z_1^{(\alpha)},\dots,Z_n^{(\alpha)}$ are independent and satisfy
\[
0\le Z_i^{(\alpha)}\le1
\qquad\text{almost surely}.
\]
Moreover,
\[
|V_\alpha(\Sigma)|
=
\sum_{i=1}^n Z_i^{(\alpha)}
\]
and
\[
\mbbE_{\mathrm{col}}[|V_\alpha(\Sigma)|]
=
\sum_{i=1}^n\mbbE_{\mathrm{col}}[Z_i^{(\alpha)}]
=
\sum_{i=1}^n p_{i,\alpha}
=
\frac nq.
\]
Applying Lemma~\ref{lem:prob-tools-balance}\textup{(ii)} to
$Z_1^{(\alpha)},\dots,Z_n^{(\alpha)}$, with $m=n$ and
\[
t=\frac{\varepsilon n}{10q},
\]
gives
\[
\mbbP_{\mathrm{col}}\left(
\left||V_\alpha(\Sigma)|-\frac nq\right|
\ge
\frac{\varepsilon n}{10q}
\right)
\le
2\exp\left(-\frac{\varepsilon^2}{50q^2}n\right).
\]

Define
\[
\Omega_{\mathrm{bal}}
\coloneqq
\left\{
C\in[q]^n:
\left||V_\alpha(C)|-\frac nq\right|
\le
\frac{\varepsilon n}{10q}
\ \text{for every }\alpha\in[q]
\right\}.
\]
By the union bound,
\begin{equation}
\label{eq:rounding-bal-failure}
\mbbP_{\mathrm{col}}(\Omega_{\mathrm{bal}}^c)
\le
2q\exp\left(-\frac{\varepsilon^2}{50q^2}n\right).
\end{equation}

If $C\in\Omega_{\mathrm{bal}}$, then for every $\alpha\in[q]$,
\[
\left||V_\alpha(C)|-\frac nq\right|
\le
\frac{\varepsilon n}{10q}
\le
\varepsilon n.
\]
Moreover, since $\varepsilon\le1$,
\[
|V_\alpha(C)|
\ge
\frac nq-\frac{\varepsilon n}{10q}
\ge
\frac{9n}{10q}.
\]
Thus, if $n\ge2q$, then $|V_\alpha(C)|\ge9/5$. Since $|V_\alpha(C)|$ is an integer, every part has at least two vertices.

\smallskip
\noindent
After relabeling by $\phi$, the empirical graphon $W_G^\phi$ is equal to $a_{ij}$
on $J_i\times J_j$, except possibly on a null set. Hence, for all
$\alpha,\beta\in[q]$,
\begin{align}
\int_{A_\alpha\times A_\beta}W_G^\phi(x,y)\,dx\,dy
&=
\sum_{i,j=1}^n
a_{ij}\lambda(J_i\cap A_\alpha)\lambda(J_j\cap A_\beta) \notag\\
&=
\frac1{n^2}
\sum_{i,j=1}^n a_{ij}p_{i,\alpha}p_{j,\beta}.
\label{eq:coloring-integral-identity}
\end{align}

\smallskip
\noindent
\emph{Edges inside the parts.}
Fix $\alpha\in[q]$. Then
\[
2e_G(V_\alpha(\Sigma))
=
\sum_{i,j=1}^n
a_{ij}
\mathbf{1}_{\{\xi_i=\alpha\}}
\mathbf{1}_{\{\xi_j=\alpha\}}.
\]
The diagonal terms vanish because $a_{ii}=0$. For $i\neq j$, independence gives
\[
\mbbE_{\mathrm{col}}
\left[
\mathbf{1}_{\{\xi_i=\alpha\}}
\mathbf{1}_{\{\xi_j=\alpha\}}
\right]
=
p_{i,\alpha}p_{j,\alpha}.
\]
Using \eqref{eq:coloring-integral-identity}, we obtain
\[
\frac{2}{n^2}
\mbbE_{\mathrm{col}}[e_G(V_\alpha(\Sigma))]
=
\int_{A_\alpha\times A_\alpha}W_G^\phi.
\]

The graphon $W_q^\star$ is $0$ on $A_\alpha\times A_\alpha$. Therefore
\[
0
\le
\int_{A_\alpha\times A_\alpha}W_G^\phi
=
\int_{A_\alpha\times A_\alpha}(W_G^\phi-W_q^\star)
\le
\|W_G^\phi-W_q^\star\|_\square
<
\rho.
\]
Hence
\[
\mbbE_{\mathrm{col}}[e_G(V_\alpha(\Sigma))]
\le
\frac{\rho n^2}{2}.
\]
Summing over all $\alpha\in[q]$, we get
\[
\mbbE_{\mathrm{col}}\left[
\sum_{\alpha=1}^q e_G(V_\alpha(\Sigma))
\right]
\le
\frac{q\rho n^2}{2}.
\]

For $C\in[q]^n$, set
\[
Y(C)
\coloneqq
\sum_{\alpha=1}^q e_G(V_\alpha(C)).
\]
Define
\[
\Omega_{\mathrm{in}}
\coloneqq
\left\{
C\in[q]^n:
Y(C)\le \frac{\varepsilon n^2}{10q}
\right\}.
\]
By Lemma~\ref{lem:prob-tools-balance}\textup{(i)},
\begin{equation}
\label{eq:rounding-in-failure}
\begin{aligned}
\mbbP_{\mathrm{col}}(\Omega_{\mathrm{in}}^c)
&\le
\frac{\mbbE_{\mathrm{col}}[Y(\Sigma)]}
{\varepsilon n^2/(10q)} \\
&\le
\frac{q\rho n^2/2}{\varepsilon n^2/(10q)} \\
&=
\frac{5q^2\rho}{\varepsilon} \\
&=
\frac1{20}.
\end{aligned}
\end{equation}
Here the last equality uses $\rho=\varepsilon/(100q^2)$.

Lemma~\ref{lem:cauchy-schwarz}\textup{(ii)} gives
\[
\sum_{\alpha=1}^q\binom{|V_\alpha(C)|}{2}
=
\frac12\left(\sum_{\alpha=1}^q |V_\alpha(C)|^2-n\right)
\ge
\frac12\left(\frac{n^2}{q}-n\right).
\]
If $n\ge2q$, then
\[
\sum_{\alpha=1}^q\binom{|V_\alpha(C)|}{2}
\ge
\frac{n^2}{4q}.
\]
Thus every $C\in\Omega_{\mathrm{in}}$ satisfies
\[
\frac{\sum_{\alpha=1}^q e_G(V_\alpha(C))}
{\sum_{\alpha=1}^q\binom{|V_\alpha(C)|}{2}}
\le
\frac{\varepsilon n^2/(10q)}{n^2/(4q)}
=
\frac{2\varepsilon}{5}
<
\varepsilon.
\]

\smallskip
\noindent
\emph{Edges between distinct parts.}
For $1\le\alpha<\beta\le q$ and $C\in[q]^n$, define
\[
F_{\alpha,\beta}(C)
\coloneqq
e_G(V_\alpha(C),V_\beta(C)).
\]
Fix $1\le\alpha<\beta\le q$. Since $V_\alpha(\Sigma)$ and
$V_\beta(\Sigma)$ are disjoint,
\[
F_{\alpha,\beta}(\Sigma)
=
\sum_{i,j=1}^n
a_{ij}
\mathbf{1}_{\{\xi_i=\alpha\}}
\mathbf{1}_{\{\xi_j=\beta\}}.
\]
Each edge between the two color classes is counted exactly once. The
diagonal terms vanish because $a_{ii}=0$. For $i\neq j$, independence gives
\[
\mbbE_{\mathrm{col}}
\left[
\mathbf{1}_{\{\xi_i=\alpha\}}
\mathbf{1}_{\{\xi_j=\beta\}}
\right]
=
p_{i,\alpha}p_{j,\beta}.
\]
Therefore
\[
\mbbE_{\mathrm{col}}[F_{\alpha,\beta}(\Sigma)]
=
\sum_{i,j=1}^n a_{ij}p_{i,\alpha}p_{j,\beta}.
\]
Using \eqref{eq:coloring-integral-identity}, we obtain
\[
\frac1{n^2}
\mbbE_{\mathrm{col}}[F_{\alpha,\beta}(\Sigma)]
=
\int_{A_\alpha\times A_\beta}W_G^\phi(x,y)\,dx\,dy.
\]

On $A_\alpha\times A_\beta$, the graphon $W_q^\star$ equals $1/2$.
Since
\[
\lambda^2(A_\alpha\times A_\beta)=\frac1{q^2},
\]
we have
\[
\int_{A_\alpha\times A_\beta}
W_q^\star(x,y)\,dx\,dy
=
\frac1{2q^2}.
\]
By Definition~\ref{defin:cutnorm}, applied with
$S=A_\alpha$ and $T=A_\beta$, and by
\eqref{eq:rounding-cut-norm-closeness},
\begin{align*}
&
\left|
\int_{A_\alpha\times A_\beta}W_G^\phi
-
\int_{A_\alpha\times A_\beta}W_q^\star
\right| \\
&=
\left|
\int_{A_\alpha\times A_\beta}
(W_G^\phi-W_q^\star)
\right| \\
&\le
\|W_G^\phi-W_q^\star\|_\square \\
&<
\rho.
\end{align*}
Consequently,
\begin{equation}
\label{eq:rounding-cross-mean-bound}
\left|
\mbbE_{\mathrm{col}}[F_{\alpha,\beta}(\Sigma)]
-
\frac{n^2}{2q^2}
\right|
\le
\rho n^2.
\end{equation}
Thus $n^2/(2q^2)$ is the deterministic target value; in general it is
not the exact value of
$\mbbE_{\mathrm{col}}[F_{\alpha,\beta}(\Sigma)]$.

Changing one coordinate of $C$ can change the counted status only of edges
incident to the corresponding vertex. There are at most $n-1$ such edges,
and each contributes at most $1$ to the change. Hence
$F_{\alpha,\beta}$ changes by at most $n-1$, and therefore by at most $n$. By
Lemma~\ref{lem:prob-tools-balance}\textup{(iii)}, with
\[
t=\frac{\varepsilon n^2}{40q^2},
\]
we obtain
\[
\mbbP_{\mathrm{col}}\left(
\left|
F_{\alpha,\beta}(\Sigma)
-
\mbbE_{\mathrm{col}}[F_{\alpha,\beta}(\Sigma)]
\right|
>
\frac{\varepsilon n^2}{40q^2}
\right)
\le
2\exp\left(-\frac{\varepsilon^2}{800q^4}n\right).
\]

Define
\[
\Omega_{\mathrm{cross}}
\coloneqq
\left\{
C\in[q]^n:
\begin{array}{l}
\displaystyle
\left|
F_{\alpha,\beta}(C)
-
\mbbE_{\mathrm{col}}[F_{\alpha,\beta}(\Sigma)]
\right|
\le
\displaystyle\frac{\varepsilon n^2}{40q^2}
\\[2mm]
\text{for every }1\le\alpha<\beta\le q
\end{array}
\right\}.
\]
By the union bound,
\begin{equation}
\label{eq:rounding-cross-failure}
\mbbP_{\mathrm{col}}(\Omega_{\mathrm{cross}}^c)
\le
q(q-1)\exp\left(-\frac{\varepsilon^2}{800q^4}n\right).
\end{equation}

If $C\in\Omega_{\mathrm{cross}}$, then for every $\alpha<\beta$,
the definition of $\Omega_{\mathrm{cross}}$,
\eqref{eq:rounding-cross-mean-bound}, and the triangle inequality give
\begin{equation}
\label{eq:rounding-cross-edge-count-bound}
\begin{aligned}
&
\left|
e_G(V_\alpha(C),V_\beta(C))
-
\frac{n^2}{2q^2}
\right| \\
&\le
\left|
F_{\alpha,\beta}(C)
-
\mbbE_{\mathrm{col}}[F_{\alpha,\beta}(\Sigma)]
\right|
+
\left|
\mbbE_{\mathrm{col}}[F_{\alpha,\beta}(\Sigma)]
-
\frac{n^2}{2q^2}
\right| \\
&\le
\frac{\varepsilon n^2}{40q^2}
+
\rho n^2 \\
&=
\frac{7\varepsilon n^2}{200q^2} \\
&\le
\frac{\varepsilon n^2}{20q^2}.
\end{aligned}
\end{equation}

Assume also that $C\in\Omega_{\mathrm{bal}}$. Let
\[
|V_\alpha(C)|=\frac nq+\Delta_\alpha,
\qquad
|V_\beta(C)|=\frac nq+\Delta_\beta.
\]
Then
\[
|\Delta_\alpha|,|\Delta_\beta|
\le
\frac{\varepsilon n}{10q}.
\]
Hence
\begin{align*}
\left|
|V_\alpha(C)||V_\beta(C)|-\frac{n^2}{q^2}
\right|
&\le
\frac nq\bigl(|\Delta_\alpha|+|\Delta_\beta|\bigr)
+
|\Delta_\alpha\Delta_\beta|  \\
&\le
\frac{\varepsilon n^2}{5q^2}
+
\frac{\varepsilon^2 n^2}{100q^2}  \\
&\le
\frac{21\varepsilon n^2}{100q^2}  \\
&\le
\frac{\varepsilon n^2}{4q^2}.
\end{align*}
Therefore
\begin{equation}
\label{eq:rounding-part-product-bound}
\left|
\frac12|V_\alpha(C)||V_\beta(C)|
-
\frac{n^2}{2q^2}
\right|
\le
\frac{\varepsilon n^2}{8q^2}.
\end{equation}
Combining \eqref{eq:rounding-cross-edge-count-bound} and
\eqref{eq:rounding-part-product-bound}, and using the triangle inequality,
gives
\begin{align*}
&
\left|
e_G(V_\alpha(C),V_\beta(C))
-
\frac12|V_\alpha(C)||V_\beta(C)|
\right| \\
&\le
\left|
e_G(V_\alpha(C),V_\beta(C))
-
\frac{n^2}{2q^2}
\right|
+
\left|
\frac{n^2}{2q^2}
-
\frac12|V_\alpha(C)||V_\beta(C)|
\right| \\
&\le
\frac{\varepsilon n^2}{20q^2}
+
\frac{\varepsilon n^2}{8q^2}
=
\frac{7\varepsilon n^2}{40q^2}.
\end{align*}

The balance event also gives
\[
|V_\alpha(C)||V_\beta(C)|
\ge
\frac{81}{100}\frac{n^2}{q^2}
\ge
\frac45\frac{n^2}{q^2}.
\]
Thus
\begin{align*}
\left|
\frac{e_G(V_\alpha(C),V_\beta(C))}
{|V_\alpha(C)||V_\beta(C)|}
-
\frac12
\right|
&\le
\frac{7\varepsilon n^2/(40q^2)}{(4/5)n^2/q^2}  \\
&=
\frac{7\varepsilon}{32}  \\
&<
\varepsilon.
\end{align*}
By symmetry, the same estimate holds for every ordered pair of distinct parts.

\smallskip
\noindent
\emph{Fixing a coloring with the desired properties.}
Choose $N_{\mathrm{det}}(\varepsilon,q)$ so large that, for every
$n\ge N_{\mathrm{det}}$,
\begin{equation}
\label{eq:rounding-Ndet-conditions}
\begin{aligned}
n&\ge2q, \\
2q\exp\left(-\frac{\varepsilon^2}{50q^2}n\right)
&\le\frac1{40}, \\
q(q-1)\exp\left(-\frac{\varepsilon^2}{800q^4}n\right)
&\le\frac1{40}.
\end{aligned}
\end{equation}
It follows from \eqref{eq:rounding-bal-failure},
\eqref{eq:rounding-in-failure},
\eqref{eq:rounding-cross-failure}, and
\eqref{eq:rounding-Ndet-conditions} that
\[
\mbbP_{\mathrm{col}}(\Omega_{\mathrm{bal}}^c)
\le
\frac1{40},
\qquad
\mbbP_{\mathrm{col}}(\Omega_{\mathrm{in}}^c)
\le
\frac1{20},
\qquad
\mbbP_{\mathrm{col}}(\Omega_{\mathrm{cross}}^c)
\le
\frac1{40}.
\]
Therefore
\begin{align*}
\mbbP_{\mathrm{col}}
\bigl(
\Omega_{\mathrm{bal}}\cap\Omega_{\mathrm{in}}\cap\Omega_{\mathrm{cross}}
\bigr)
&\ge
1-\frac1{40}-\frac1{20}-\frac1{40}  \\
&=
\frac9{10}
>
0.
\end{align*}
So there exists a coloring
\[
C\in
\Omega_{\mathrm{bal}}\cap\Omega_{\mathrm{in}}\cap\Omega_{\mathrm{cross}}.
\]
Set
\[
V_\alpha\coloneqq V_\alpha(C),
\qquad \alpha\in[q].
\]
The estimates above show that this partition has the required part sizes, the required
between-part edge densities, and the required bound on the number of edges inside the
parts.

Finally, let $G'$ be obtained from $G$ by deleting all edges inside the parts
$V_1,\dots,V_q$. Then $G'$ is $q$-partite, and
\[
|E(G)\mathbin\triangle E(G')|
=
\sum_{\alpha=1}^q e_G(V_\alpha)
\le
\frac{\varepsilon n^2}{10q}.
\]
By the definition of $d_{\mathrm{edit}}$ in Section~\ref{sec:introduction},
\[
d_{\mathrm{edit}}(G,G')
=
\frac{|E(G)\mathbin\triangle E(G')|}{n^2}
\le
\frac{\varepsilon}{10q}.
\]
\end{proof}
We now use the cut-distance concentration theorem to prove the finite-graph
structure theorem.

\begin{proof}[Proof of Theorem~\ref{thm:main-clique}\textup{(iii)}]
Fix $\varepsilon>0$. Since Lemma~\ref{lem:det-round-Bqstar} is stated for
parameters at most $1$, set
\[
\eta\coloneqq \min\{\varepsilon,1\},
\qquad
\rho\coloneqq \frac{\eta}{100q^2}.
\]

Let $\mathsf G_n\sim\mbbP_{n,r}$. Applying
Theorem~\ref{thm:main-clique}\textup{(ii)} with the parameter $\rho$, we get constants
\[
c_0=c_0(\rho,r)>0
\qquad\text{and}\qquad
N_0=N_0(\rho,r,w)
\]
such that, for every $n\ge N_0$,
\[
\mbbP\left(
\delta_\square(W_{\mathsf G_n},B_q^\star)\ge \rho
\right)
\le
\exp(-c_0n^2).
\]

Let
\[
N_{\mathrm{det}}=N_{\mathrm{det}}(\eta,q)
\]
be the threshold from Lemma~\ref{lem:det-round-Bqstar}, and define
\[
n_0\coloneqq \max\{N_0,N_{\mathrm{det}}\}.
\]
Then, for every $n\ge n_0$, with probability at least $1-\exp(-c_0n^2)$, we have
\[
\delta_\square(W_{\mathsf G_n},B_q^\star)<\rho.
\]

On this event, Lemma~\ref{lem:det-round-Bqstar}, applied with $\eta$ in place of
$\varepsilon$, gives a partition
\[
[n]=V_1\sqcup\cdots\sqcup V_q
\]
satisfying all the conclusions of that lemma with parameter $\eta$. Since
\[
\eta\le \varepsilon,
\]
the same partition satisfies the required bounds in Theorem~\ref{thm:main-clique}\textup{(iii)} with
parameter $\varepsilon$.

In particular, if $\mathsf G_n'$ is obtained from $\mathsf G_n$ by deleting all
edges whose endpoints lie in the same part, then $\mathsf G_n'$ is $q$-partite and
\[
d_{\mathrm{edit}}(\mathsf G_n,\mathsf G_n')
\le
\frac{\eta}{10q}
\le
\frac{\varepsilon}{10q}
\le
\frac{\varepsilon}{2},
\]
because $q\ge2$.

Thus the theorem holds with
\[
c\coloneqq c_0.
\]
Since $\rho=\eta/(100q^2)$, and since $q=r-1$, the constant $c$ depends only on
$\varepsilon$ and $r$. The threshold $n_0$ may also depend on $w$.
\end{proof}


\section{Examples and consequences}
\label{sec:examples-consequences}

\noindent
We first discuss the boundary case $r=2$ and the first case covered by
Theorem~\ref{thm:main-clique}, namely $r=3$. We then show that the induced
subgraph on any fixed number of labeled vertices converges in distribution
to a stochastic block model, derive the limiting clique densities, and give
an exponential bound for colorability with fewer than $r-1$ colors. Finally,
we extend the model to several fixed clique penalties.

\subsection{The case $r=2$}
\label{sec:r-equals-two}

The notions defined in Section~\ref{sec:model}
are straightforwardly extended to the case $r = 2$, and we get
\[
N_{2}(G)=n^2-2|E(G)|,
\]
and 
\[
\widetilde Z_{n,2}
=
\left(1+e^{-2w}\right)^{\binom n2}.
\]
The possible edges are therefore independent, and each is present with
probability
\[
\frac{e^{-2w}}{1+e^{-2w}}
=
\frac1{1+e^{2w}}.
\]
Thus $\mbbP_{n,2}$ is exactly the
$G\left(n,1/(1+e^{2w})\right)$ distribution.
In particular, the $(r-1)$-partite behavior in
Theorem~\ref{thm:main-clique} begins at $r=3$.

\subsection{The case $r=3$}
\label{sec:triangle-case}

Here $q=r-1=2$. By Corollary~\ref{cor:ERGM}, the reduced mass of a graph $G$ is
\[
e^{-6wK_3(G)}.
\]
Thus triangles are penalized. By
Theorem~\ref{thm:main-clique}\textup{(ii)}, the empirical graphon converges
in probability, in cut distance, to the class represented by
$W_2^\star$. The canonical representative $W_2^\star$ has two equal blocks
and takes value $0$ within each block and $1/2$ between the blocks.

\subsection{A stochastic block model limit}
\label{sec:sbm-limit}

For $m\in\mbbN$ and probability measures $\alpha$ and $\beta$ on
$\mbG_m$, their \emph{total variation distance} is
\[
d_{\mathrm{TV}}(\alpha,\beta)
\coloneqq
\frac12\sum_{J\in\mbG_m}|\alpha(J)-\beta(J)|.
\]

\begin{defin}\label{defin:induced-density}
For $m\in\mbbN$, $J\in\mbG_m$, and a graphon $W$, define the
\emph{induced density} of $J$ in $W$ by
\[
t_{\mathrm{ind}}(J,W)
\coloneqq
\int_{[0,1]^m}
\prod_{\{i,j\}\in E(J)}W(x_i,x_j)
\prod_{\substack{1\le i<j\le m\\
                  \{i,j\}\notin E(J)}}
\bigl(1-W(x_i,x_j)\bigr)
\,dx_1\cdots dx_m.
\]
This is the induced-density formula in
\cite[Sec.~7.2, Eq.~(7.3)]{Lovasz2012}. Under the graphon sampling
construction from Section~\ref{sec:graphons}, it equals the probability of
obtaining the labeled graph $J$; see
\cite[Sec.~10.1, immediately before Eq.~(10.3)]{Lovasz2012}.
\end{defin}

\begin{cor}[Stochastic block model limit for fixed induced subgraphs]
\label{cor:sbm-limit}
Fix an integer $r\ge3$, set $q=r-1$, and let $w>0$ and $m\in\mbbN$.
For every $n\ge m$, let $\mathsf G_n\sim\mbbP_{n,r}$, and let
$\mathsf H_{n,m}$ be the subgraph of $\mathsf G_n$ induced by $[m]$.

Define a random graph $\mathsf S_m^{(q)}$ on $[m]$ as follows. Choose
$C_1,\dots,C_m$ independently and uniformly from $[q]$. Given these
labels, include the possible edges independently, with probability
\[
\mbbP_{\mathrm{SBM}}
\bigl(
\{i,j\}\in E(\mathsf S_m^{(q)})
\mid C_1,\dots,C_m
\bigr)
=
\begin{cases}
0, & C_i=C_j,\\
\tfrac12, & C_i\ne C_j,
\end{cases}
\qquad 1\le i<j\le m.
\]
Here $\mbbP_{\mathrm{SBM}}(S_m^{(q)})$ denotes probability of getting $S_m^{(p)}$ under this construction.
Define probability measures $\nu_{n,m}$ and $\nu_m^{(q)}$ on $\mbG_m$ by
\[
\nu_{n,m}(J)
\coloneqq
\mbbP(\mathsf H_{n,m}=J),
\qquad
\nu_m^{(q)}(J)
\coloneqq
\mbbP_{\mathrm{SBM}}(\mathsf S_m^{(q)}=J)
\quad \text{ for } \ J\in\mbG_m.
\]
Then
\[
\lim_{n\to\infty}
d_{\mathrm{TV}}\bigl(\nu_{n,m},\nu_m^{(q)}\bigr)
=
0.
\]
The limiting measure $\nu_m^{(q)}$ depends on $q=r-1$ but not on $w$.
\end{cor}

\begin{proof}
Fix $J\in\mbG_m$. For the induced density from
Definition~\ref{defin:induced-density}, inclusion-exclusion gives
\cite[Sec.~7.2, Eq.~(7.4)]{Lovasz2012}
\[
t_{\mathrm{ind}}(J,W)
=
\sum_{\substack{F\in\mbG_m\\E(J)\subseteq E(F)}}
(-1)^{|E(F)\setminus E(J)|}t(F,W).
\]
By \cite[Thm.~3.7(a)]{BorgsChayesLovaszSosVesztergombi2008}, each
homomorphism density in this sum is continuous in cut distance.
Therefore $t_{\mathrm{ind}}(J,\cdot)$ is also continuous in cut distance.

Theorem~\ref{thm:main-clique}\textup{(ii)} and the final identity in
Notation~\ref{not:Wqstar} give, for every $\varepsilon>0$,
\[
\lim_{n\to\infty}
\mbbP\left(
\left|
t_{\mathrm{ind}}(J,W_{\mathsf G_n})
-
t_{\mathrm{ind}}(J,W_q^\star)
\right|
\ge\varepsilon
\right)
=
0.
\]
Since the absolute difference is bounded by $1$, for every $\varepsilon>0$,
its expectation is at most $\varepsilon$ plus the probability that it is at
least $\varepsilon$. Hence,
\begin{equation}\label{eq:sbm-induced-density-average}
\lim_{n\to\infty}
\mbbE\left[
\left|
t_{\mathrm{ind}}(J,W_{\mathsf G_n})
-
t_{\mathrm{ind}}(J,W_q^\star)
\right|
\right]
=
0.
\end{equation}

Define a probability measure $\bar\nu_{n,m}$ on $\mbG_m$ by
\[
\bar\nu_{n,m}(J)
\coloneqq
\mbbE\bigl[t_{\mathrm{ind}}(J,W_{\mathsf G_n})\bigr].
\]
By \eqref{eq:showing-ergm}, the probability assigned to a graph $G$ depends
only on $K_r(G)$, which is invariant under vertex relabeling. Hence the
distribution of $\mathsf G_n$ is invariant under vertex relabeling.
Consequently, if $\iota:[m]\to[n]$ is an independent uniformly chosen
injection, then $\mathsf H_{n,m}$ has the same distribution as the graph on
$[m]$ in which $\{i,j\}$ is an edge exactly when
$\{\iota(i),\iota(j)\}\in E(\mathsf G_n)$.

On the other hand, sampling from $W_G$ corresponds to choosing vertices of
$G$ independently, so the same vertex may be chosen more than once.
The probability of choosing the same vertex more than once is at most
$\binom m2/n$. Therefore, by
\cite[Eq.~(10.1)]{Lovasz2012},
\[
d_{\mathrm{TV}}(\nu_{n,m},\bar\nu_{n,m})
\le
\frac{\binom m2}{n}.
\]

By Notation~\ref{not:Wqstar}, graphon sampling from $W_q^\star$ is
exactly the construction of $\mathsf S_m^{(q)}$, therefore
\[
t_{\mathrm{ind}}(J,W_q^\star)
=
\nu_m^{(q)}(J).
\]
It follows that
\[
\begin{aligned}
d_{\mathrm{TV}}\bigl(\nu_{n,m},\nu_m^{(q)}\bigr)
&\le
\frac{\binom m2}{n}
+d_{\mathrm{TV}}\bigl(\bar\nu_{n,m},\nu_m^{(q)}\bigr)\\
&\le
\frac{\binom m2}{n}
+\frac12\sum_{J\in\mbG_m}
\mbbE\left[
\left|
t_{\mathrm{ind}}(J,W_{\mathsf G_n})
-
t_{\mathrm{ind}}(J,W_q^\star)
\right|
\right].
\end{aligned}
\]
The sum is finite, and each term tends to zero by
\eqref{eq:sbm-induced-density-average}. This proves the result.
\end{proof}

\begin{rem}[The limiting model is not Erd\H{o}s-R\'enyi]
\label{rem:sbm-not-erdos-renyi}
Let $q\ge2$ and $m\ge2$, and consider the model
$\mathsf S_m^{(q)}$ defined above. For $1\le i<j\le m$, let
\[
X_{ij}
\coloneqq
\mathbf{1}_{\{\{i,j\}\in E(\mathsf S_m^{(q)})\}}.
\]
For each edge,
\[
\mbbP_{\mathrm{SBM}}(X_{ij}=1)
=
\mbbP_{\mathrm{SBM}}(C_i\ne C_j)\,\frac12
=
\frac{q-1}{q}\cdot\frac12
=
\frac{q-1}{2q}.
\]

Any two distinct edge indicators are independent. This is immediate
for vertex-disjoint edges because they use different labels and
different edge choices. If the edges share a vertex, say $\{i,j\}$
and $\{i,k\}$, then
\[
\mbbP_{\mathrm{SBM}}(X_{ij}=X_{ik}=1)
=
\left(\frac{q-1}{q}\right)^2\frac14
=
\left(\frac{q-1}{2q}\right)^2,
\]
which is the product of their individual probabilities.

However, if $m\ge3$, the three indicators associated with the triangle
on $\{1,2,3\}$ are not independent as a triple, since
\[
\mbbP_{\mathrm{SBM}}(X_{12}=X_{13}=X_{23}=1)
=
\frac{(q-1)(q-2)}{8q^2},
\]
whereas independence would give
\[
\left(\frac{q-1}{2q}\right)^3.
\]
Indeed,
\[
\left(\frac{q-1}{2q}\right)^3
-
\frac{(q-1)(q-2)}{8q^2}
=
\frac{q-1}{8q^3}
>
0.
\]
Thus any two edge indicators are independent, but the three triangle
indicators are not independent as a triple. In particular,
$\mathsf S_m^{(q)}$ is not an Erd\H{o}s-R\'enyi graph when $m\ge3$.
\end{rem}

\subsection{Limiting clique densities}
\label{sec:limiting-clique-spectrum}

For any fixed integer $r \geq 3$
we now determine the limiting density, with respect to $\mbbP_{n, r}$, of $K_s$ for all $s \geq 2$.

\begin{cor}
\label{cor:limiting-clique-spectrum}
Fix an integer $r\ge3$ and let $w>0$. Let
$\mathsf G_n\sim\mbbP_{n,r}$, and set $q=r-1$. For every fixed integer
$s\ge2$,
\begin{equation}\label{eq:limiting-clique-spectrum-constant}
t(K_s,W_q^\star)
=
\begin{cases}
\displaystyle
\frac{q(q-1)\cdots(q-s+1)}
{q^s2^{\binom s2}},
& 2\le s\le q,\\[3mm]
0,
& s>q.
\end{cases}
\end{equation}
Moreover, for every $\eta>0$, there are constants
$c=c(\eta,r,s)>0$ and $N=N(\eta,r,s,w)\ge s$ such that, for every
$n\ge N$,
\begin{equation}\label{eq:limiting-clique-spectrum-concentration}
\mbbP_{n,r}\left(
\left\{
G\in\mbG_n:
\left|
\frac{s!\,K_s(G)}{n^s}
-
t(K_s,W_q^\star)
\right|
\ge\eta
\right\}
\right)
\le e^{-cn^2}.
\end{equation}
Consequently, for every $\eta>0$,
\begin{equation}\label{eq:limiting-clique-spectrum-proportion}
\lim_{n\to\infty}
\mbbP\left(
\left|
\frac{K_s(\mathsf G_n)}{\binom ns}
-
t(K_s,W_q^\star)
\right|
\ge\eta
\right)
=
0.
\end{equation}
Thus every fixed clique $K_s$ with $2\le s\le r-1$ has a positive
limiting density, while every fixed clique $K_s$ with $s\ge r$ has
limiting density zero.
\end{cor}

\begin{proof}
Suppose first that $2\le s\le q$. By
Definition~\ref{defin:hom-density},
\[
t(K_s,W_q^\star)
=
\int_{[0,1]^s}
\prod_{1\le i<j\le s}
W_q^\star(x_i,x_j)
\,dx_1\cdots dx_s.
\]
Equivalently, choose $x_1,\dots,x_s$ independently and uniformly at random
from $[0,1]$. By Notation~\ref{not:Wqstar}, the index of the canonical block
containing each $x_i$ is uniform on $[q]$, and these block indices are
independent. Since $W_q^\star$ is zero within each block, the integrand is
nonzero only when $x_1,\dots,x_s$ lie in distinct blocks. The set of such
points has measure
\[
\frac{q(q-1)\cdots(q-s+1)}{q^s}.
\]
On this set, every factor in the integrand is equal to $1/2$. Since
$K_s$ has $\binom s2$ edges, it follows that
\[
t(K_s,W_q^\star)
=
\frac{q(q-1)\cdots(q-s+1)}
{q^s2^{\binom s2}}.
\]

If $s>q$, then among any $s$ points $x_1,\dots,x_s$, at least two lie in
the same canonical block. The corresponding factor in the defining
integral is zero. Hence
\[
t(K_s,W_q^\star)=0.
\]
This proves \eqref{eq:limiting-clique-spectrum-constant}.

By Lemma~\ref{lem:empirical-Kr-density}, with $s$ in place of $r$, every
graph $G\in\mbG_n$ satisfies
\[
t(K_s,W_G)=\frac{s!\,K_s(G)}{n^s}.
\]
Lemma~\ref{lem:Kr-cut-lipschitz} and
Notation~\ref{not:Wqstar} therefore give
\[
\left|
\frac{s!\,K_s(G)}{n^s}
-
t(K_s,W_q^\star)
\right|
\le
\binom s2\,\delta_\square(W_G,B_q^\star).
\]
Thus, if the left-hand side is at least $\eta$, then
\[
\delta_\square(W_G,B_q^\star)
\ge
\frac{\eta}{\binom s2}.
\]
Applying Theorem~\ref{thm:main-clique}\textup{(ii)} with
$\eta/\binom s2$ in place of $\varepsilon$ proves
\eqref{eq:limiting-clique-spectrum-concentration}.

Finally, for $n\ge s$, set
\[
X_n\coloneqq \frac{s!\,K_s(\mathsf G_n)}{n^s},
\qquad
a_n\coloneqq \frac{n(n-1)\cdots(n-s+1)}{n^s}.
\]
Then
\[
\frac{K_s(\mathsf G_n)}{\binom ns}=\frac{X_n}{a_n}.
\]
Equation~\eqref{eq:limiting-clique-spectrum-concentration} says that
$X_n$ converges in probability to $t(K_s,W_q^\star)$. Since $a_n>0$ and
\[
\lim_{n\to\infty}a_n=1,
\]
the quotient $X_n/a_n$ converges in probability to the same limit. This is
exactly \eqref{eq:limiting-clique-spectrum-proportion}.
\end{proof}

\begin{exam}
For example, let $r=4$, so that $q=3$, and let
$\mathsf G_n\sim\mbbP_{n,4}$. Then
\[
t(K_3,W_3^\star)
=
\frac{3\cdot2\cdot1}{3^3\,2^3}
=
\frac1{36},
\qquad
t(K_4,W_3^\star)=0.
\]
Hence, for every $\eta>0$,
\[
\lim_{n\to\infty}
\mbbP\left(
\left|
\frac{K_3(\mathsf G_n)}{\binom n3}
-
\frac1{36}
\right|
\ge\eta
\right)
=
0,
\qquad
\lim_{n\to\infty}
\mbbP\left(
\frac{K_4(\mathsf G_n)}{\binom n4}
\ge\eta
\right)
=
0.
\]
Thus a positive proportion of triples still form triangles,
even though the proportion of four-element sets that form a copy of $K_4$
tends to zero.
\end{exam}

\subsection{Colorability with fewer than
\texorpdfstring{$r-1$}{r-1} colors}
\label{sec:exact-lower-colorability}

\begin{cor}
\label{cor:exact-lower-colorability}
Fix an integer $r\ge3$ and let $w>0$. Set $q=r-1$, and let
$\mathsf G_n\sim\mbbP_{n,r}$. Recall that $\chi(\mathsf G_n)$ denotes the
chromatic number of $\mathsf G_n$. For every $n\in\mbbN$ and every integer
$k$ with $1\le k < q$,
\begin{equation}\label{eq:exact-lower-colorability-bound}
\mbbP\bigl(\chi(\mathsf G_n)\le k\bigr)
\le
\exp\left(
-\frac{(q-k)\ln 2}{2kq}n^2
+n\ln k
+\frac{q\ln 2}{8}
\right).
\end{equation}
Consequently, there are constants $c=c(r,k)>0$ and
$N=N(r,k)\in\mbbN$ such that
\[
\mbbP\bigl(\chi(\mathsf G_n)\le k\bigr)
\le e^{-cn^2}
\qquad\text{for every }n\ge N.
\]
In particular, when $r\ge4$, the probability that $\mathsf G_n$ is
$(r-2)$-partite tends to zero exponentially in $n^2$.
\end{cor}

\begin{proof}
Fix a map from $[n]$ to a set of $k$ colors, and let
$m_1,\dots,m_k$ be the sizes of its color classes. A graph for which this
map is a proper coloring can have edges only between different classes.
The number of pairs in different classes is
\[
\frac12\left(n^2-\sum_{i=1}^k m_i^2\right)
\le
\frac12\left(1-\frac1k\right)n^2,
\]
where the inequality follows from
Lemma~\ref{lem:cauchy-schwarz}\textup{(ii)}. Hence at most
\[
2^{(1-1/k)n^2/2}
\]
graphs have this fixed proper coloring. Since there are $k^n$ such maps,
\begin{equation}\label{eq:k-colorable-graph-count}
\bigl|\{G\in\mbG_n:\chi(G)\le k\}\bigr|
\le
k^n2^{(1-1/k)n^2/2}.
\end{equation}

A $k$-colorable graph has no clique with more than $k$ vertices. Since
$k<q<r$, every graph counted in
\eqref{eq:k-colorable-graph-count} is $K_r$-free and has reduced mass $1$.
Hence \eqref{eq:k-colorable-graph-count} and
\eqref{eq:balanced-q-partite-lower} give
\[
\mbbP\bigl(\chi(\mathsf G_n)\le k\bigr)
\le
\exp\left(
\frac{\ln 2}{2}\left(1-\frac1k\right)n^2
+n\ln k
-s_qn^2
+\frac{q\ln 2}{8}
\right).
\]
Using \eqref{eq:sq-definition} gives
\eqref{eq:exact-lower-colorability-bound}. Since $q-k>0$, the coefficient
of $n^2$ is negative, which proves the final estimate.
\end{proof}

\subsection{Several fixed clique penalties}
\label{sec:multi-clique-collapse}

We now consider a set-up with soft constraints for several clique sizes.
Fix an integer $\ell\ge3$. For each $k\in\{3,\dots,\ell\}$
let a weight $w_k \geq 0$ be fixed.
{\em Assume that at least one of these $w_k$ is positive} and set
\[
r\coloneqq
\min\{k\in\{3,\dots,\ell\}:w_k>0\},
\qquad
q\coloneqq r-1.
\]
Thus $w_k = 0$ for $k < r$.
For all 
$k\in\{3,\dots,\ell\}$ and every $G \in \mbG_n$, define
\[
N_{k}(G):=n^k-k!\,K_k(G).
\]
Hence $N_{k}(G)$ is the number of $k$-tuples of vertices of $G$ that do not form a $k$-clique in $G$.
We now define the ``{\em unnormalized mass}'' of
$G\in\mbG_n$ as the number
\[
\exp\left(\sum_{k=3}^{\ell}w_kN_k(G)\right)
=
\exp\left(\sum_{k=r}^{\ell}w_kn^k\right)
\exp\left(-\sum_{k=r}^{\ell}k!\,w_k\,K_k(G)\right).
\]
The first factor is independent of $G$ and cancels after normalization.
We therefore define the ``{\em reduced mass}'' of $G \in \mbG_n$ by
\[
M_n^{\mathrm{mc}}(G)
\coloneqq
\exp\left(-\sum_{k=r}^{\ell}k!\,w_k\,K_k(G)\right),
\]
and then the  corresponding reduced partition function, probability measure, and
reduced free energy are defined by
\begin{equation}\label{eq:multi-clique-definitions}
\begin{aligned}
\widetilde Z_n^{\mathrm{mc}}
&\coloneqq
\sum_{G\in\mbG_n}M_n^{\mathrm{mc}}(G),\\
\mbbP_n^{\mathrm{mc}}(A)
&\coloneqq
\frac{1}{\widetilde Z_n^{\mathrm{mc}}}
\sum_{G\in A}M_n^{\mathrm{mc}}(G),
\quad \text{ for } \ A\subseteq\mbG_n, \ \ \text{ and}\\
F_n^{\mathrm{mc}}
&\coloneqq
\frac{1}{n^2}\ln\widetilde Z_n^{\mathrm{mc}}.
\end{aligned}
\end{equation}

\begin{theor}[The smallest positively weighted clique determines the limit structure]
\label{thm:multi-clique-collapse}
Let
\(
\mathsf G_n^{\mathrm{mc}}\sim\mbbP_n^{\mathrm{mc}},
\)
that is, $\mathsf G_n$ is a random graph from $\mbG_n$ under the probability distribution $\mbbP_n^{\mathrm{mc}}$.
Then the following statements hold:

\begin{enumerate}
\item[\textup{(i)}]
\(
\lim_{n\to\infty}F_n^{\mathrm{mc}}
=
\tfrac{\ln 2}{2}\left(1-\frac{1}{r-1}\right).
\)

\item[\textup{(ii)}]
For every $\varepsilon>0$, there are constants
\[
c=c(\varepsilon,r)>0
\qquad\text{and}\qquad
N=N(\varepsilon,r,w_r)
\]
such that, for every $n\ge N$,
\[
\mbbP\left(
\delta_\square
\left(W_{\mathsf G_n^{\mathrm{mc}}},B_q^\star\right)
\ge\varepsilon
\right)
\le e^{-cn^2}.
\]

\item[\textup{(iii)}]
For every $\varepsilon>0$, there are constants
\[
c=c(\varepsilon,r)>0
\qquad\text{and}\qquad
N=N(\varepsilon,r,w_r)
\]
such that, for every $n\ge N$, with probability at least $1-e^{-cn^2}$,
there is a partition
\[
[n]=V_1\sqcup\cdots\sqcup V_q
\]
with the following properties:
\begin{enumerate}
\item For every $i\in[q]$,
\[
\bigl||V_i|-n/q\bigr|\le\varepsilon n.
\]

\item For all distinct  $i,j\in[q]$, 
\[
\left|
\frac{e_{\mathsf G_n^{\mathrm{mc}}}(V_i,V_j)}{|V_i||V_j|}
-\frac12
\right|
\le\varepsilon.
\]

\item The total number of edges inside parts is
\[
\sum_{i=1}^q e_{\mathsf G_n^{\mathrm{mc}}}(V_i)
\le
\frac{\varepsilon n^2}{10q}.
\]

\item The total density of edges inside the parts is at most $\varepsilon$:
\[
\frac{\sum_{i=1}^q e_{\mathsf G_n^{\mathrm{mc}}}(V_i)}
{\sum_{i=1}^q\binom{|V_i|}{2}}
\le\varepsilon.
\]
\end{enumerate}
If $\mathsf G_n'$ is obtained by deleting all edges whose endpoints lie
in the same part, then $\mathsf G_n'$ is $q$-partite and
\[
d_{\mathrm{edit}}
\left(\mathsf G_n^{\mathrm{mc}},\mathsf G_n'\right)
\le
\frac{\varepsilon}{10q}
\le
\frac{\varepsilon}{2}.
\]
\end{enumerate}
In particular, the reduced free-energy limit and the cut-distance limit
depend on the weight vector $(w_3,\dots,w_\ell)$ only through $r$. Once
$r$ and $w_r$ are fixed, the constants in parts~\textup{(ii)} and
\textup{(iii)} can be chosen independently of $\ell$ and of the weights
$w_{r+1},\dots,w_\ell$.
\end{theor}

\begin{proof}
Let $\widetilde Z_{n,r}$, $F_{n,r}(w_r)$, and $\mbbP_{n,r}$ denote the
reduced partition function, reduced free energy, and probability measure of
the single-$K_r$ model with weight $w_r$, and set
\[
C_q\coloneqq\frac{q\ln 2}{8}.
\]
Since the higher weights are nonnegative, every $G\in\mbG_n$ satisfies
\begin{equation}\label{eq:multi-single-mass-comparison}
0<M_n^{\mathrm{mc}}(G)
\le
\exp\bigl(-r!\,w_rK_r(G)\bigr)
\le1.
\end{equation}
Every graph in the balanced $q$-partite family used in the proof of
Lemma~\ref{lem:balanced-q-partite-lower} is $K_r$-free and hence
$K_k$-free for every $k\ge r$. Its multi-clique reduced mass is therefore $1$. 
Let 
\[
s_q := \frac{\ln 2}{2}\bigg(1 - \frac{1}{q}\bigg).
\]
The same lower-bound construction, together with
\eqref{eq:multi-single-mass-comparison}, gives
\begin{equation}\label{eq:multi-clique-partition-comparison}
\exp\bigl(s_qn^2-C_q\bigr)
\le
\widetilde Z_n^{\mathrm{mc}}
\le
\widetilde Z_{n,r}.
\end{equation}
This comparison explains the role of the smallest positively weighted
clique; the lower bound is unaffected by any clique penalty, while the
upper bound retains only the $K_r$ penalty.

For part~\textup{(i)}, taking logarithms in
\eqref{eq:multi-clique-partition-comparison} and dividing by $n^2$ gives
\[
s_q-\frac{C_q}{n^2}
\le
F_n^{\mathrm{mc}}
\le
F_{n,r}(w_r).
\]
The conclusion follows from
Theorem~\ref{thm:main-clique}\textup{(i)}.

For parts~\textup{(ii)} and~\textup{(iii)}, set
\[
R_n\coloneqq
\frac{\widetilde Z_{n,r}}{\widetilde Z_n^{\mathrm{mc}}}.
\]
By \eqref{eq:multi-single-mass-comparison}, every event
$\mc A\subseteq\mbG_n$ satisfies
\begin{equation}\label{eq:multi-clique-measure-comparison}
\mbbP_n^{\mathrm{mc}}(\mc A)
\le
R_n\,\mbbP_{n,r}(\mc A).
\end{equation}
Moreover, \eqref{eq:multi-clique-partition-comparison} implies
\[
0
\le
\frac{1}{n^2}\ln R_n
\le
F_{n,r}(w_r)-s_q+\frac{C_q}{n^2},
\]
and hence, by Theorem~\ref{thm:main-clique}\textup{(i)},
\begin{equation}\label{eq:multi-clique-ratio-subexponential}
\lim_{n\to\infty}\frac{1}{n^2}\ln R_n=0.
\end{equation}
The upper bound depends only on $n$, $r$, and $w_r$, and it tends to $0$.
Since the same bound holds simultaneously for every $\ell\ge r$ and every
choice of $w_{r+1},\dots,w_\ell\ge0$, the convergence is uniform in those
parameters.

Consequently, suppose that a sequence of events $(\mc A_n)$ satisfies
\[
\mbbP_{n,r}(\mc A_n)\le e^{-an^2}
\]
for some $a>0$ and all sufficiently large $n$. After increasing the
threshold if necessary,
\eqref{eq:multi-clique-ratio-subexponential} gives
$R_n\le e^{an^2/2}$, and
\eqref{eq:multi-clique-measure-comparison} gives
\[
\mbbP_n^{\mathrm{mc}}(\mc A_n)
\le
e^{-an^2/2}.
\]

For part~\textup{(ii)}, apply this observation to
\[
\mc A_n
=
\left\{
G\in\mbG_n:
\delta_\square(W_G,B_q^\star)\ge\varepsilon
\right\}
\]
and use Theorem~\ref{thm:main-clique}\textup{(ii)}.

For part~\textup{(iii)}, let $\mc B_{n,\varepsilon}$ be the set of graphs
that admit no partition satisfying the five properties stated there.
Theorem~\ref{thm:main-clique}\textup{(iii)} gives an exponential bound for
$\mbbP_{n,r}(\mc B_{n,\varepsilon})$, so the same observation gives the
required bound under $\mbbP_n^{\mathrm{mc}}$. The edit-distance conclusion
follows from any partition satisfying those five properties, exactly as in
Theorem~\ref{thm:main-clique}\textup{(iii)}.

For each of parts~\textup{(ii)} and~\textup{(iii)}, one may choose
$c=c(\varepsilon,r)>0$ and $N=N(\varepsilon,r,w_r)$. Neither choice
depends on $\ell$ or on the higher weights
$w_{r+1},\dots,w_\ell$.
\end{proof}


\end{document}